\documentclass[ijoc,sglanonrev]{informs4}

\MANUSCRIPTNO{IJOC-0001-2024.00}

\usepackage{xr}
\usepackage{algorithm}
\usepackage{algpseudocode}
\usepackage{mathrsfs}
\usepackage{amsmath}
\usepackage{amssymb}
\usepackage{amsfonts}
\usepackage{mathtools}
\usepackage{bm}
\usepackage{cite}
\usepackage{hyperref}
\usepackage{url} 
\usepackage{xcolor}
\usepackage{tikz}
\usepackage{pgfplots}
\pgfplotsset{width=10cm,compat=1.9}
\usepackage{makecell} 

\usetikzlibrary{calc, positioning}
\usepackage{pdflscape}
\usepackage{bbm}
\usepackage{enumitem}
\usepackage{longtable}
\usepackage{adjustbox}
\usepackage{multirow}
\usepackage{multicol}
\usepackage{array}

\usepackage{scalerel}

\usepackage{booktabs}
\allowdisplaybreaks[1]
\usepackage{wrapfig}

\let\theoremstyle\relax

\usepackage{amsthm}

\usepackage{capt-of}

\newtheorem{theorem}{Theorem}
\newtheorem{proposition}{Proposition}

\newtheorem{lemma}{Lemma}

\newtheorem{definition}{Definition}

\newtheorem{assumption}{Assumption}
\newtheorem*{remark}{Remark}

\usepackage{graphicx}
\usepackage{subcaption}
\usepackage{afterpage}
\graphicspath{{../../}}

\usepackage{algpseudocode}   

\DeclareCaptionSubType*{algorithm}

\makeatletter
\renewcommand\p@subalgorithm{\thealgorithm}
\makeatother

\newcommand{\bxi}{\bm{\xi}}
\newcommand{\bmu}{\bm{\mu}}
\newcommand{\bSigma}{\bm{\Sigma}}

\newcommand{\bzeta}{\bm{\zeta}}
\newcommand{\blambda}{\bm{\lambda}}
\newcommand{\bv}{\bm{v}}
\newcommand{\bc}{\bm{c}}

\newcommand{\bD}{\bm{D}}
\newcommand{\bH}{\bm{H}}
\newcommand{\bd}{\bm{d}}
\newcommand{\ba}{\bm{a}}
\newcommand{\bh}{\bm{h}}

\newcommand{\bu}{\bm{u}}

\newcommand{\bx}{\bm{x}}
\newcommand{\by}{\bm{y}}
\newcommand{\bw}{\bm{w}}
\newcommand{\bz}{\bm{z}}
\newcommand{\bA}{\bm{A}}
\newcommand{\bB}{\bm{B}}
\newcommand{\bQ}{\bm{Q}}

\newcommand{\balpha}{\bm{\alpha}}
\newcommand{\bt}{\bm{t}}

\newcommand{\beq}{\begin{equation}}
\newcommand{\eeq}{\end{equation}}
\newcommand{\bal}{\begin{aligned}}
\newcommand{\ea}{\end{aligned}}
\newcommand{\bdm}{\begin{displaymath}}
\newcommand{\edm}{\end{displaymath}}

\DeclarePairedDelimiterX\Set[2]{\lbrace}{\rbrace}%
 { #1 \,\delimsize| \,\mathopen{} #2 }
\newcommand{\bs}[1]{\boldsymbol{#1}}

\newcommand{\mG}{\mathcal{G}}
\newcommand{\mH}{\mathcal{H}}

\newcommand{\mX}{\mathcal{X}}

\newcommand{\zz}{\check{z}}
\newcommand{\bzz}{\bm{\zz}}

\newcommand{\ones}{\bm{1}}
\newcommand{\zeros}{\bm{0}}

\DeclarePairedDelimiterX{\inp}[2]{\langle}{\rangle}{#1, #2}

\newcommand{\ubar}[1]{\underline{#1}}
\newcommand{\specialfootnote}[2]{%
  \begingroup
  \renewcommand{\thefootnote}{#1}%
  \footnote{#2}%
  \endgroup
}
\newcommand{\daggerfootnote}{\textsuperscript{†}}
\newcommand{\dollarfootnote}{\textsuperscript{§}}
\newcommand{\treefootnote}{\textsuperscript{‡}}
\newcommand{\starfootnote}{\textsuperscript{*}}
\usepackage{tablefootnote}

\usepackage[normalem]{ulem}
\newcommand{\replace}[2]{%
  \ifx\empty#1\empty%
    {\color{blue}#2}%
  \else%
    \ifx\empty#2\empty%
      {\color{orange}\sout{#1}}%
    \else%
	  {\color{orange}\sout{#1}}%
      {\color{blue}#2}%
    \fi%
  \fi%
}

\begin{document}
\RUNAUTHOR{Dey et al.}

\RUNTITLE{On Solving Chance-Constrained Models with GM Distribution}

\TITLE{On Solving Chance-Constrained Models with Gaussian Mixture Distribution}

\ARTICLEAUTHORS{
\AUTHOR{Shibshankar Dey}
\AFF{Department of Industrial Engineering and Management Science, Northwestern University, Evanston, IL \\
\EMAIL{shibshankardey2025@u.northwestern.edu }}

\AUTHOR{ Sanjay Mehrotra}
\AFF{Department of Industrial Engineering and Management Science, Northwestern University, Evanston, IL \\ 
\EMAIL{mehrotra@northwestern.edu}}
\AUTHOR{Anirudh Subramanyam}
\AFF{Department of Industrial and Manufacturing Engineering, Pennsylvania State University, University Park, PA \\
\EMAIL{subramanyam@psu.edu}}
} 

\ABSTRACT{%
\normalsize
	We study linear chance-constrained problems where the coefficients follow a Gaussian mixture distribution. We provide mixed-binary quadratic programs that give inner and outer approximations of the chance constraint based on piecewise linear approximations of the standard normal cumulative density function. We show that $O\left(\sqrt{\ln(1/\tau)/\tau} \right)$ pieces are sufficient to attain $\tau$-accuracy in the chance constraint. We also show that any desired optimality gap can be achieved under a constraint qualification condition by controlling the approximation accuracy.
       Extensive computations using a commercial solver show that problems with up to one thousand random coefficients specified with up to fifteen Gaussian mixture components, generated under diverse settings, can be solved to near optimality within 18 hours, while satisfying chance constraint satisfaction probabilities of up to $0.999$. The solution times are significantly lower for problems with fewer random coefficients and mixture terms. For example, problems with one hundred random coefficients, ten mixture terms, and a constraint satisfaction probability of $0.999$ can be solved in a minute or less. Sample average approximations fail to provide meaningful solutions even for the smaller problems. 
}%


\maketitle


\section{Introduction}\label{sec:intro}
 Chance-constrained models \cite{charnes1959chance,miller1965chance} ensure that constraints involving random parameters are satisfied with a desired probability. These models are known to be hard to solve, even in the presence of a single chance constraint, due to their non-convex feasible set \cite{luedtke2008sample, nemirovski2007convex}. We study linear chance-constrained optimization problems of the form:
\begin{subequations} \label{eq:ccp}
\begin{align}
Z^\star(\theta) := \min_{\bx \in \mX} \; & \bc^\top \bx  \\
 \text{s.t.} \; &  \mathbb{P}\left[\bxi^\top \bx \leq b \right] \geq \theta. \label{eq:cc}
\end{align}
\end{subequations}
Here, $\bx \in \mathbb{R}^n$ is the decision vector and $\mX$ represents its feasible region specified by deterministic constraints. We assume that $\mX$ is compact. We suppose that 
$\bxi$ is a random vector taking realizations in $ \mathbb{R}^n$. The constant $\theta \in (0, 1)$ specifies the desired probability of constraint satisfaction. 

We focus on the solvability of \eqref{eq:ccp} when the probability distribution of $\bxi$ is described by a $K$-component Gaussian mixture model (GMM). Specifically, the density function of $\bxi$ is $\sum_{k=1}^K w_k \mathcal{N}(\cdot \; \vert\; \bmu_k, \bSigma_k)$,
where $\mathcal{N}(\cdot\; |\; \bmu_k, \bSigma_k)$ is the density function of a multivariate normal vector with mean $\bmu_k \in \mathbb R^n$ and covariance matrix $\bSigma_k \in \mathbb R^{n \times n}$.  The vector $\bw \in \mathbb R^K_{+}$ specifies the mixture weights, satisfying $\ones^\top \bw = 1$, where $\ones$ denotes the vector of ones. We assume that $\bSigma_k$ is positive definite for all $k \in [K]$, where we define $[K] \coloneqq \{1, 2, \ldots, K\}$. Our motivation for specifying the chance constraint using a GMM stems from the fact that any probability density function can be approximated using a Gaussian mixture distribution to arbitrary accuracy \cite{titterington1985statistical,scott2015multivariate}.
This allows chance constraint models of problems where the uncertain data is multi-modal.
We list several recent applications of Gaussian mixture models to optimization problems arising in various domains in Section~\ref{sec:literature}.

\begin{remark}
	The chance constraint~\eqref{eq:cc} is an example of left-hand side uncertainty. Although we assume the right-hand side coefficient $b$ to be a given constant, it may also follow a (univariate) Gaussian mixture distribution. This case can be reduced to~\eqref{eq:ccp} by augmenting the decision vector $\bx$ to $(\bx, x')$,
    adding the constraint $x' = 1$ to $\mX$, replacing the objective coefficients $\bc$ with $(\bc, 0)$, and updating the chance constraint to $\mathbb{P}[\bxi^\top \bx - bx' \leq 0] \geq \theta$. 
    Similarly, the form of problem~\eqref{eq:ccp} is also general enough to accommodate constraints of the form, $\bxi^\top \left( \bH \bx + \bh \right) + \ba^\top \bx \leq b$, where $\bH \in \mathbb{R}^{m \times m'}$, $\bh \in \mathbb{R}^m$, and $\ba \in \mathbb{R}^{m'}$. This can be captured by setting $n = m+m'$, augmenting the decision vector $\bx$ to $(\bx, \bx')$, adding the constraints $\bx' = \bH \bx + \bh$ to $\mathcal{X}$, and replacing the objective coefficients $\bc$ with $(\bc, \zeros)$. 
\end{remark}

\begin{assumption}\label{assmp:positive-definiteness}
The covariance matrix $\Sigma_k$ is positive definite for all $k \in [K]$; specifically,
$\bx^\top \bSigma_k \bx > 0$ for all $\bx \in \mathbb{R}^n \backslash \{\mathbf{0}\}$. 
\end{assumption}

Define $p(\bx) \coloneqq \mathbb{P} \left[ \bxi^\top \bx \leq b \right]$ to be the probability function. Then, Assumption~\ref{assmp:positive-definiteness} implies that we can equivalently write $p(\bx)$ as
\vspace{-10pt}
\begin{subequations}\label{eq:cc_p_x_definition}
\begin{align}
	p(\bx) 
    &= \sum_{k=1}^K w_k \mathbb{P}\left[\bxi_k^\top \bx \leq b \, \big\vert \, \bxi_k \sim \mathcal N (\bmu_k, \bSigma_k) \right] \label{eq:basic-GMM-cc}\\
	&= \sum_{k=1}^K w_k \mathbb{P}\left[\bxi_k^\top \bx \leq b \, \big\vert \, \bxi_k^\top \bx \sim \mathcal N (\bmu_k^\top \bx, \bx^\top \bSigma_k \bx) \right] \label{eq:cc_probability_components} \\
	& = \sum_{k=1}^K w_k p_k(\bx), \\
    \text{where } p_k(\bx) &\coloneqq
    \begin{dcases*}\label{eq:cc_probability_standardized_form}
        \mathbbm{1}_{\geq 0}(b), & if $\bx = \mathbf{0}$, \\
        \Phi \left( \frac{b - \bmu_k^\top \bx}{\sqrt{\bx^\top \bSigma_k \bx}} \right) & otherwise.
    \end{dcases*}
\end{align}
\end{subequations}
Here, $\Phi$ denotes the standard normal cumulative distribution function (CDF) and $\mathbbm{1}_{\geq 0}$ denotes the indicator function of the nonnegative reals. We will see in Proposition~\ref{prop:reformulation0} and Lemma~\ref{lem:uniform-compactness} that  problem~\eqref{eq:ccp} is infeasible if $b < 0$. The chance constraint~\eqref{eq:cc} can now be equivalently expressed as $p(\bx) \geq \theta$. We denote the feasible region of problem~\eqref{eq:ccp} as:
\vspace{-10pt}
\begin{align}\label{def: P_theta}
    P(\theta) \coloneqq \{\bx \in \mathbb{R}^n \; | \; \bH\bx = \bh, \bA \bx \geq \bd, \; p(\bx)  \geq \theta \},
\end{align}
whenever $\mX = \{\bx \in \mathbb R^n | \; \bH\bx = \bh, \bA \bx \geq \bd \}$ is specified explicitly using linear equality and inequality constraints.  

\begin{remark}\label{rem:positive-definiteness}
 Assumption~\ref{assmp:positive-definiteness} and the resulting reformulation~\eqref{eq:cc_probability_standardized_form} of the chance constraint~\eqref{eq:cc} is without loss of generality. Indeed, suppose that $\bSigma_k$ has rank $r_k < n$. Then, it admits an eigenvalue decomposition; specifically, there exists $\bQ_k \in \mathbb{R}^{n \times n}$ with columns formed from the eigenvectors of $\bSigma_k$ such that $\bD_k = \bQ_k^\top \bSigma_k \bQ_k = \begin{pmatrix}
     \bD'_k & \zeros \\
     \zeros & \zeros
 \end{pmatrix}$, where $\bD'_k \in \mathbb{R}^{r_k \times r_k}$ is a diagonal matrix.
 Now, augment the decision vector $\bx$ to $(\bx, \by_k)$ and add the constraints, $\by_k = \bQ_k \bx$, to $\mathcal{X}$.
 Also, partition $\by_k = (\by'_k, \by''_k)$ such that $\by'_k \in \mathbb R^{r_k}$.
 Then, it can be readily verified that the associated $p_k(\bx)$ in~\eqref{eq:cc_probability_standardized_form} can be written as $\Phi\Big(\frac{b - {\mu_k}^\top \bx}{\sqrt{{\by_k'}^\top \bD_k' \by_k'}}\Big)$ with ${\by_k'}^\top \bD_k' \by_k' > 0$ for all $\by_k' \neq \zeros$. 
\end{remark}

The chance constraint, $p(\bx) \geq \theta$, can be nonconvex even in simple cases.
To visualize this, consider a two-dimensional problem with $\mX = [-15, 15]^2$, $K = 2$, eigenvalues of $\bSigma_1$ and $\bSigma_2$ are $\bD_1 = \mathrm{diag}(1.15, \, 0.65)$ and $\bD_2 = \mathrm{diag}(1.47, \, 0.33)$, respectively, with eigenvectors $\bQ_1$ and $\bQ_2$ shown below:
\begin{gather*}
w_1 = w_2 = 0.5, \;
\bmu_1 = \bmu_2 = \begin{bmatrix} 0.875 \\ 1.784 \end{bmatrix}, \;
\bQ_1 = \begin{bmatrix} 1 & -0.08 \\ 0.08 & 1 \end{bmatrix}, \;
\bQ_2 = \begin{bmatrix} 1 & -0.02 \\ 0.02 & 1 \end{bmatrix}.
\end{gather*}
The chance-constraint right-hand side coefficient $b = 6.7$. 
The left-hand side of Figure~\ref{fig:3D-surface} plots $p_1(\bx)$ and $p_2(\bx)$ whereas the right-hand side plots $p(\bx)$.
Observe that $p(\bx)$ is nonconvex over $\mathcal{X}$.
The non-convexity can also be seen from Figure~\ref{fig:contour}, which shows the contour plots of $p_1(\bx)$, $p_2(\bx)$ on the left and of $p(\bx)$ on the right.

\vspace{-10pt}
\begin{figure}[h]
	\centering
	\begin{subfigure}[b]{0.45\textwidth}
		\centering
		\includegraphics[width=\textwidth]{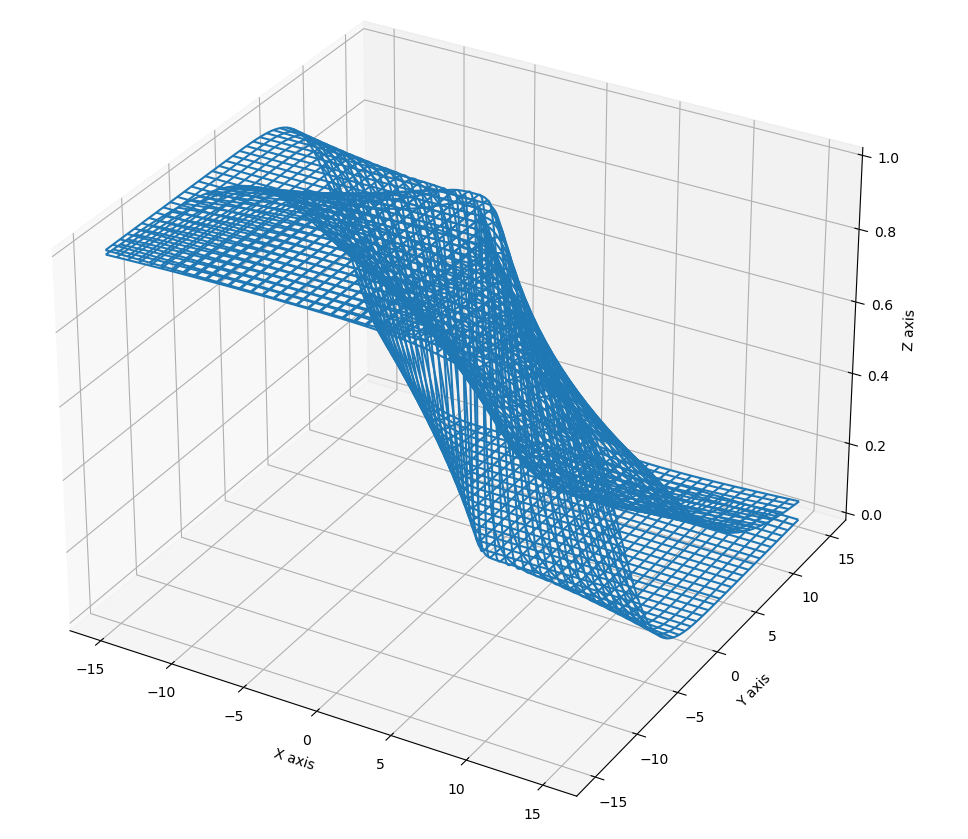
		}
	\end{subfigure}
	\hfill
	\begin{subfigure}[b]{0.45\textwidth}
		\centering
		\includegraphics[width=\textwidth]{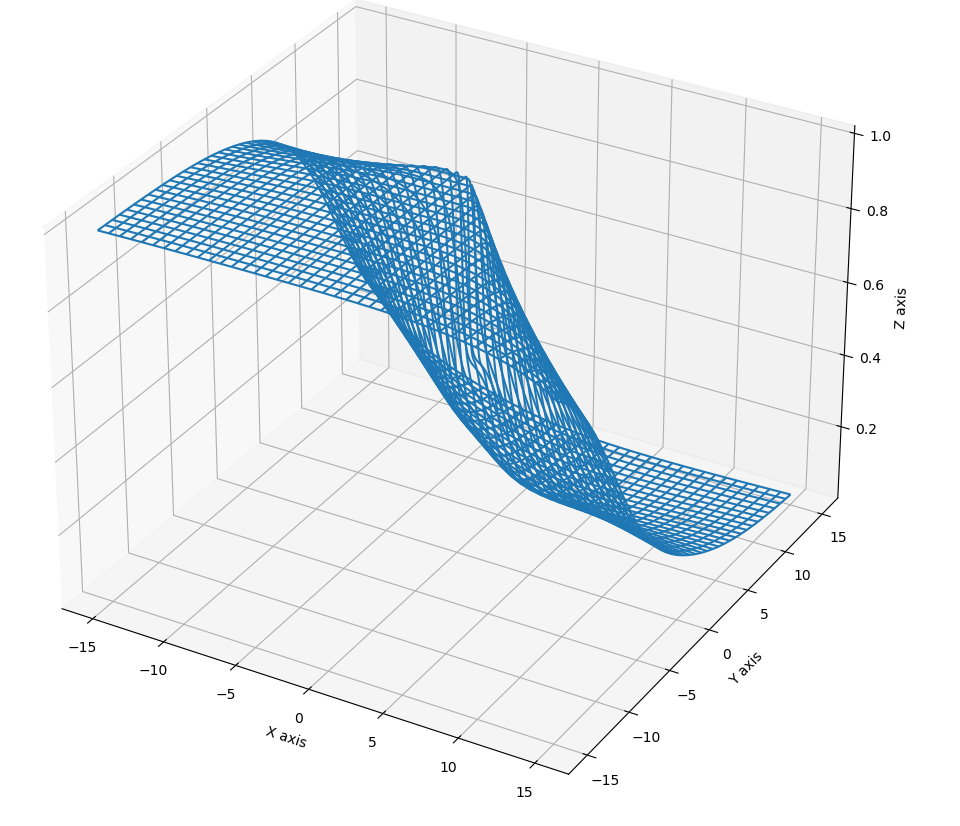}
	\end{subfigure}
	\caption{\centering Plots of $p_1(\bx)$, $p_2(\bx)$ (left) and $p(\bx)$ (right) over $\bx \in \mX$.}
	\label{fig:3D-surface}
\end{figure}
\vspace{-15pt}
\begin{figure}[h]
	\centering
	\begin{subfigure}[b]{0.45\textwidth}
		\centering
		\includegraphics[width=\textwidth]{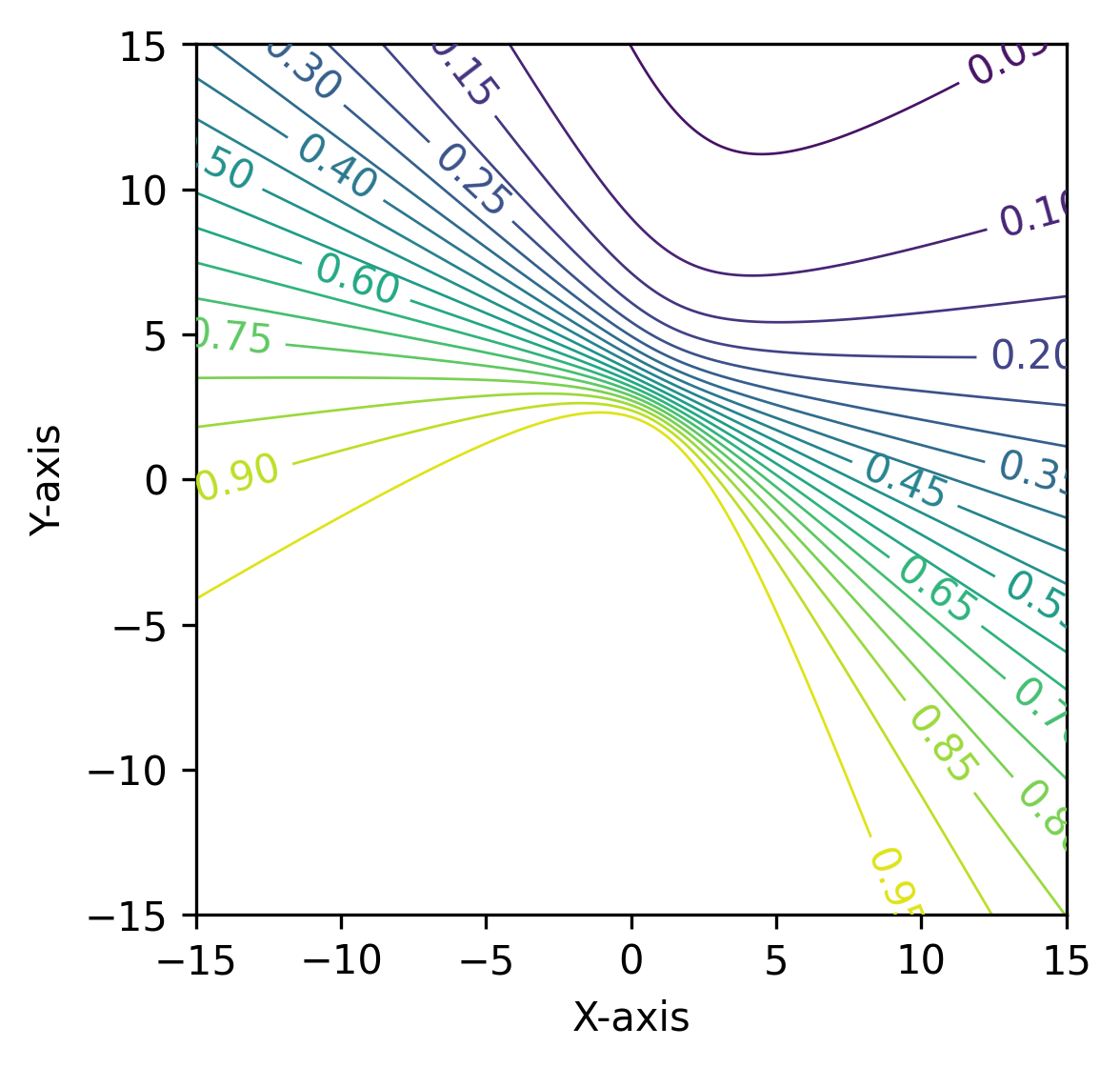
		}
	\end{subfigure}
	\hfill
	\begin{subfigure}[b]{0.45\textwidth}
		\centering
		\includegraphics[width=\textwidth]{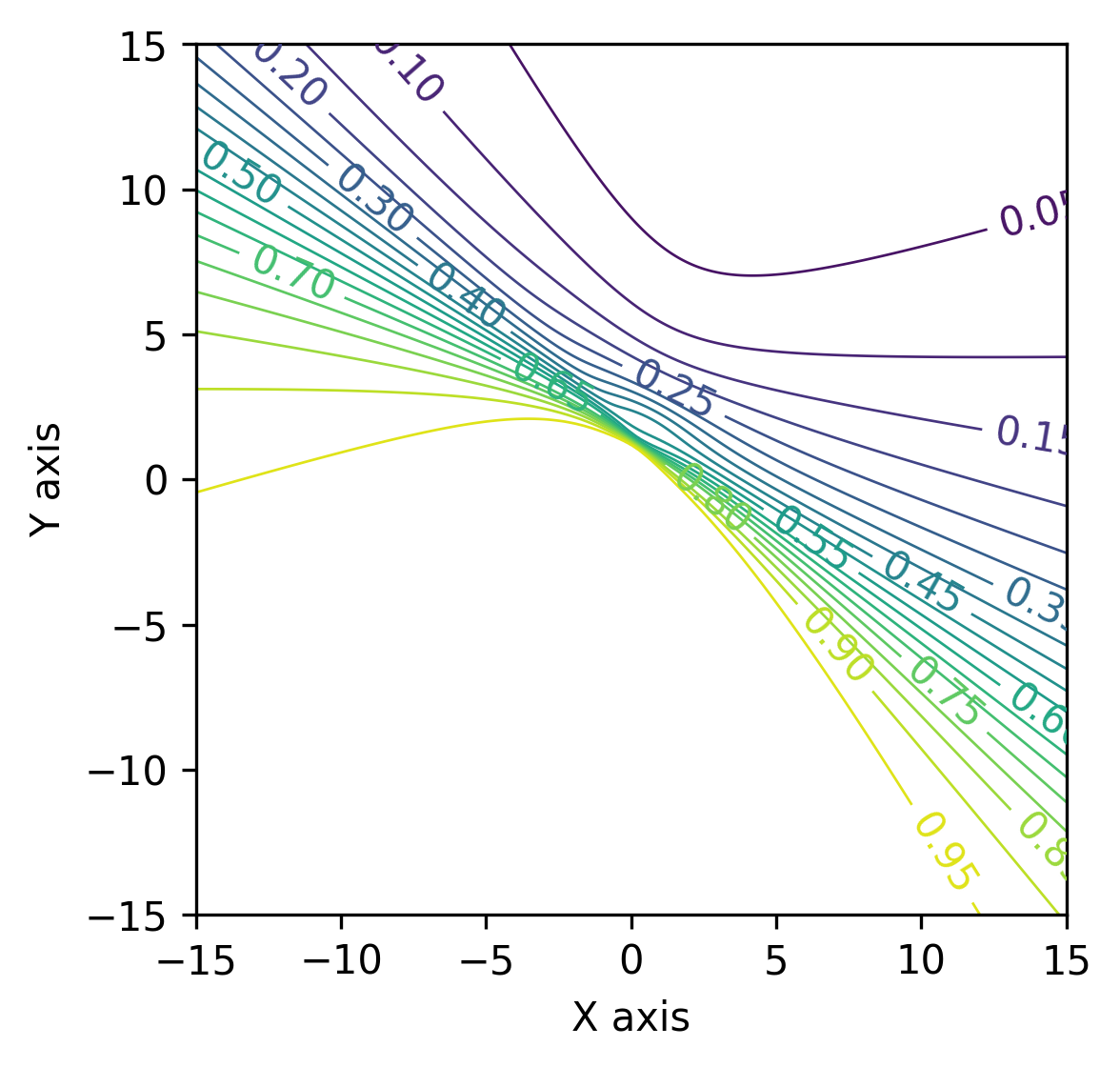}
	\end{subfigure}
	\caption{\centering Contour plots of $p_1(\bx)$, $p_2(\bx)$ (left) and $p(\bx)$ (right) over $\bx \in \mX$.}
	\label{fig:contour}
\end{figure}

The shape of the distribution of $\bxi^\top \bx$ can also change significantly as a function of $\bx$.
This is illustrated in Figure~\ref{table:Fig_tables}, where each plot corresponds to different solutions sampled from the feasible region of a problem in which the random vector follows a GMM with $K=5$ components
\footnote{\scriptsize The data for this experiment is generated as follows. We set $\theta = 0.9$, $n = 100$, $\mX = [-20, 20]^{100}$, and each entry of $\bc \in \mathbb{R}^{100}$ is randomly sampled from $[-1, 1]$. The chance constraint right-hand side coefficient $b = 1303.223$ is set to be the average of $\bx^\top \bmu_k + 4\sqrt{\bx^\top\bSigma_k\bx}$ across all $k$ and one thousand randomly sampled decisions $\bx \in \mX$. We set $K=5$, $\bw = (0.30, 0.10, 0.20, 0.05, 0.35)$, and each entry of $\bmu_k$ is sampled from $[\underline{\mu}_k, \bar{\mu}_k]$, which are $[10^2, 10^3]$, $[5 \cdot 10^2, 1.2 \cdot 10^3]$, $[-20 \cdot 10^3, -8 \cdot 10^3], [1.4 \cdot 10^3, 2.2 \cdot 10^3]$, $[-3 \cdot 10^3, -2 \cdot 10^3]$, for $k = 1, 2, \ldots, 5$, respectively. The eigenvalues; i.e., diagonal entries of $\bD_k$, are uniformly sampled from $[0, \bar{\nu}_k]$, where $\bar{\nu}_1, \bar{\nu}_2, \ldots, \bar{\nu}_5$ are evenly spaced points in $(0, 5]$. The eigenvector matrices $\bQ_k$ are randomly generated orthogonal matrices 
to ensure that $\bSigma_k = \bQ^\top_k \bD_k \bQ_k$ is positive definite. 
\label{dataGen-footnote}}.
In each plot, we normalize the $x$-axis and show the individual component-specific normal densities of $\bxi_k^\top \bx$, see equation~\eqref{eq:cc_probability_components}, for $k = 1, 2, \ldots, 5$ using blue, green, cyan, magenta, and orange colored curves, respectively. The density of $\bxi^\top \bx$ is shown in black color.

\vspace{-20pt}
\begin{center}
\begin{longtable}[!htbp]{ccc}
    \includegraphics[width=0.32\textwidth]{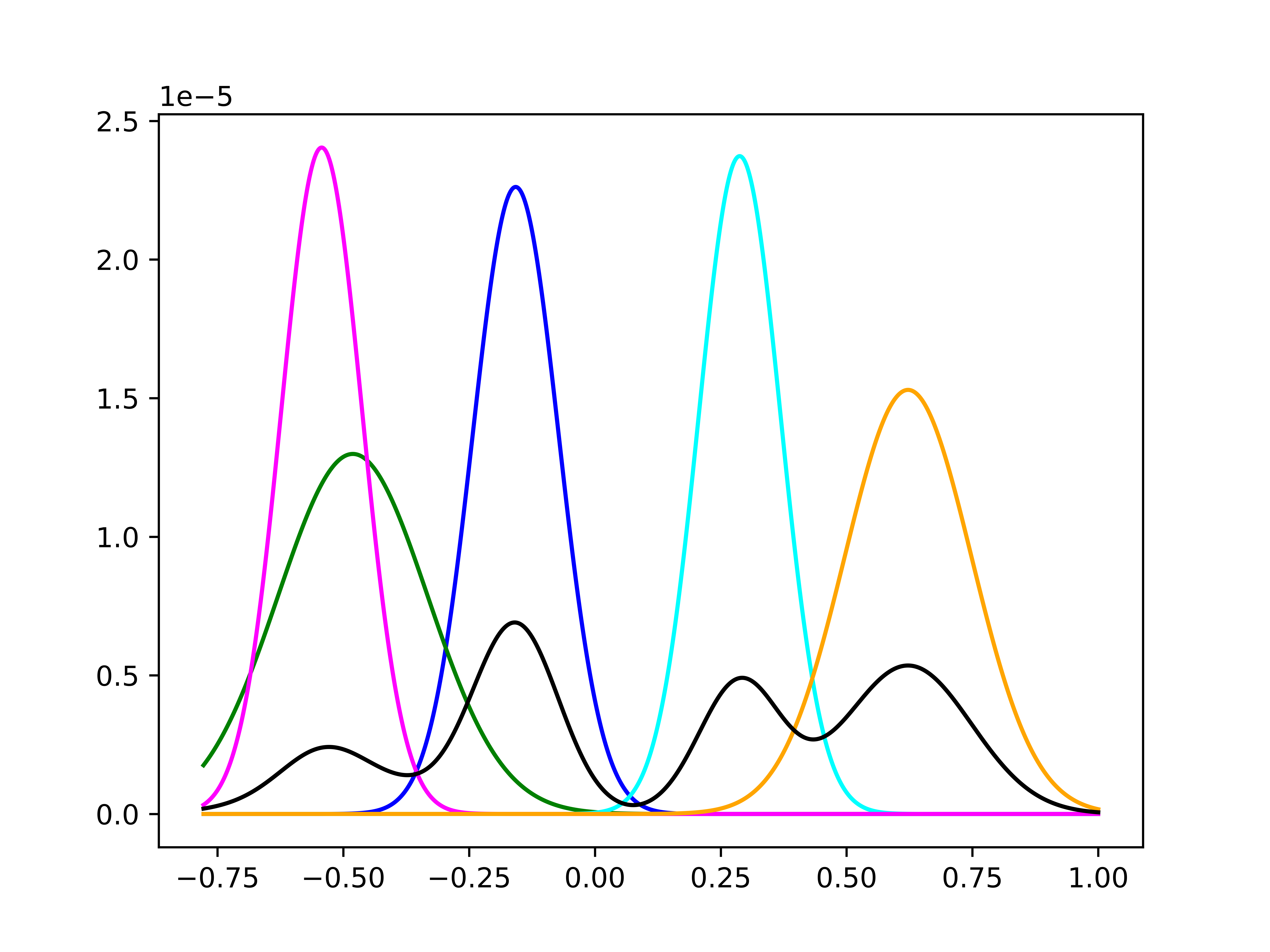} & 
    \includegraphics[width=0.32\textwidth]{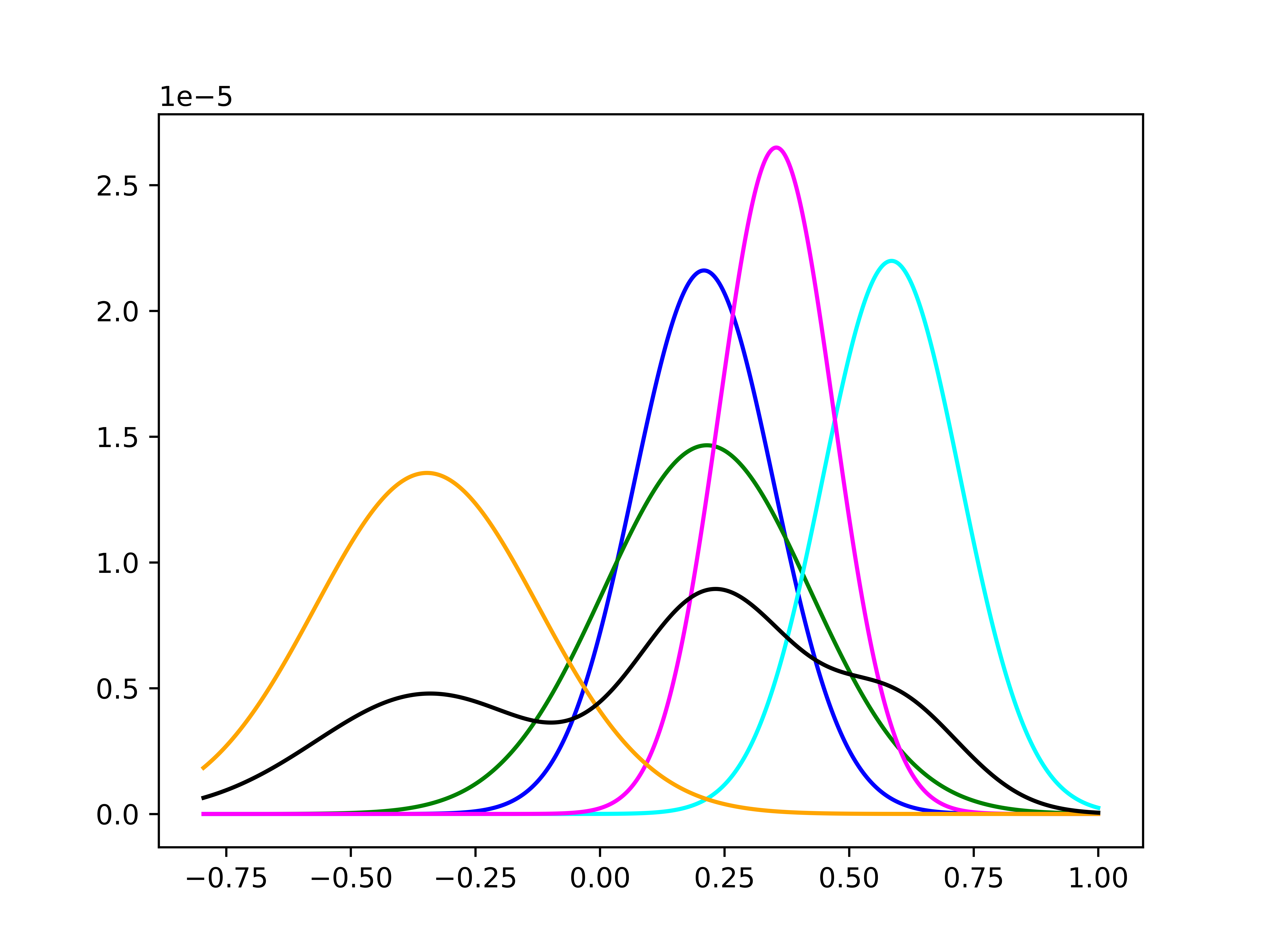} &
    \includegraphics[width=0.32\textwidth]{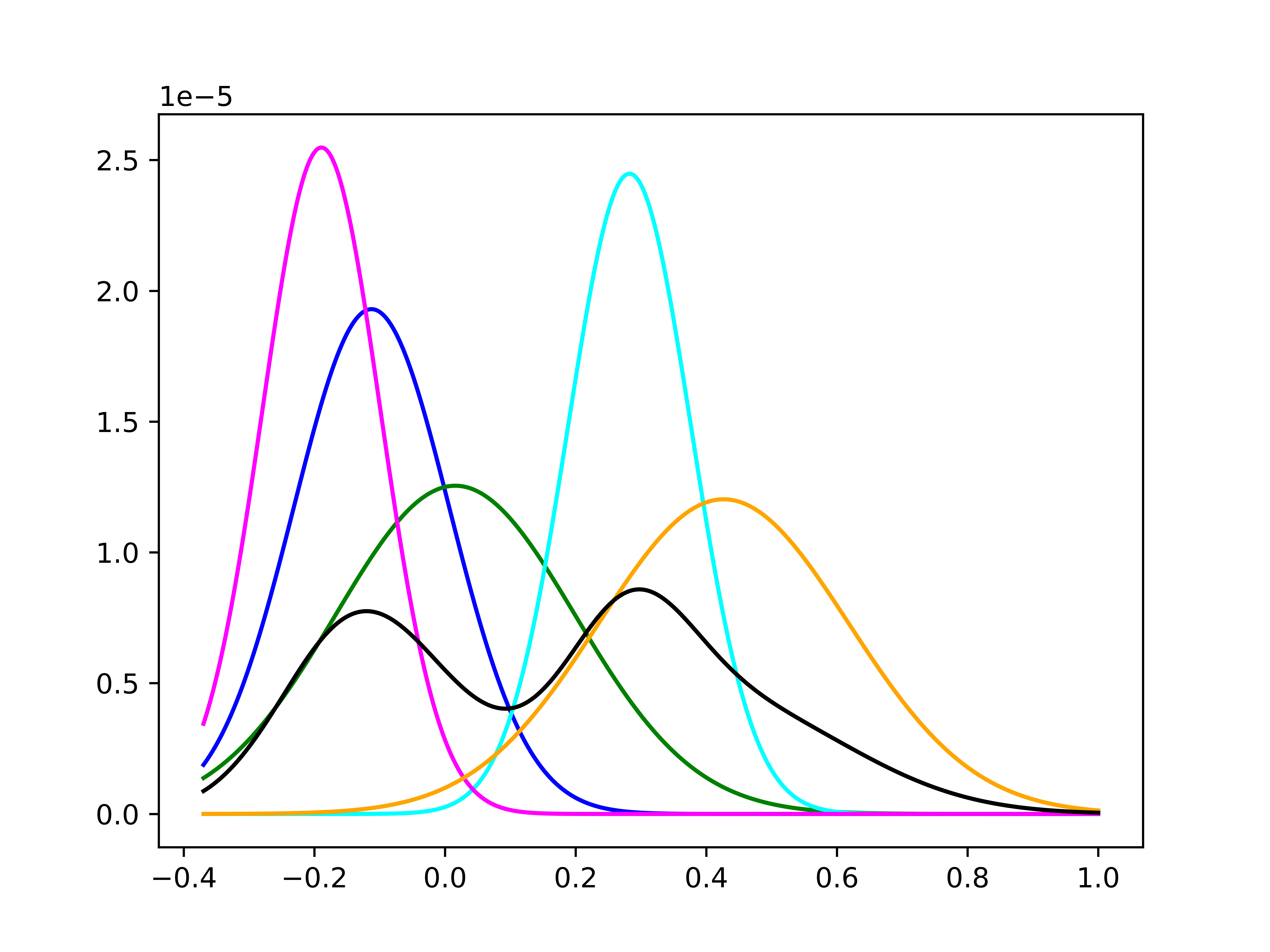} \\
    \vspace{-10pt}
    \includegraphics[width=0.32\textwidth]{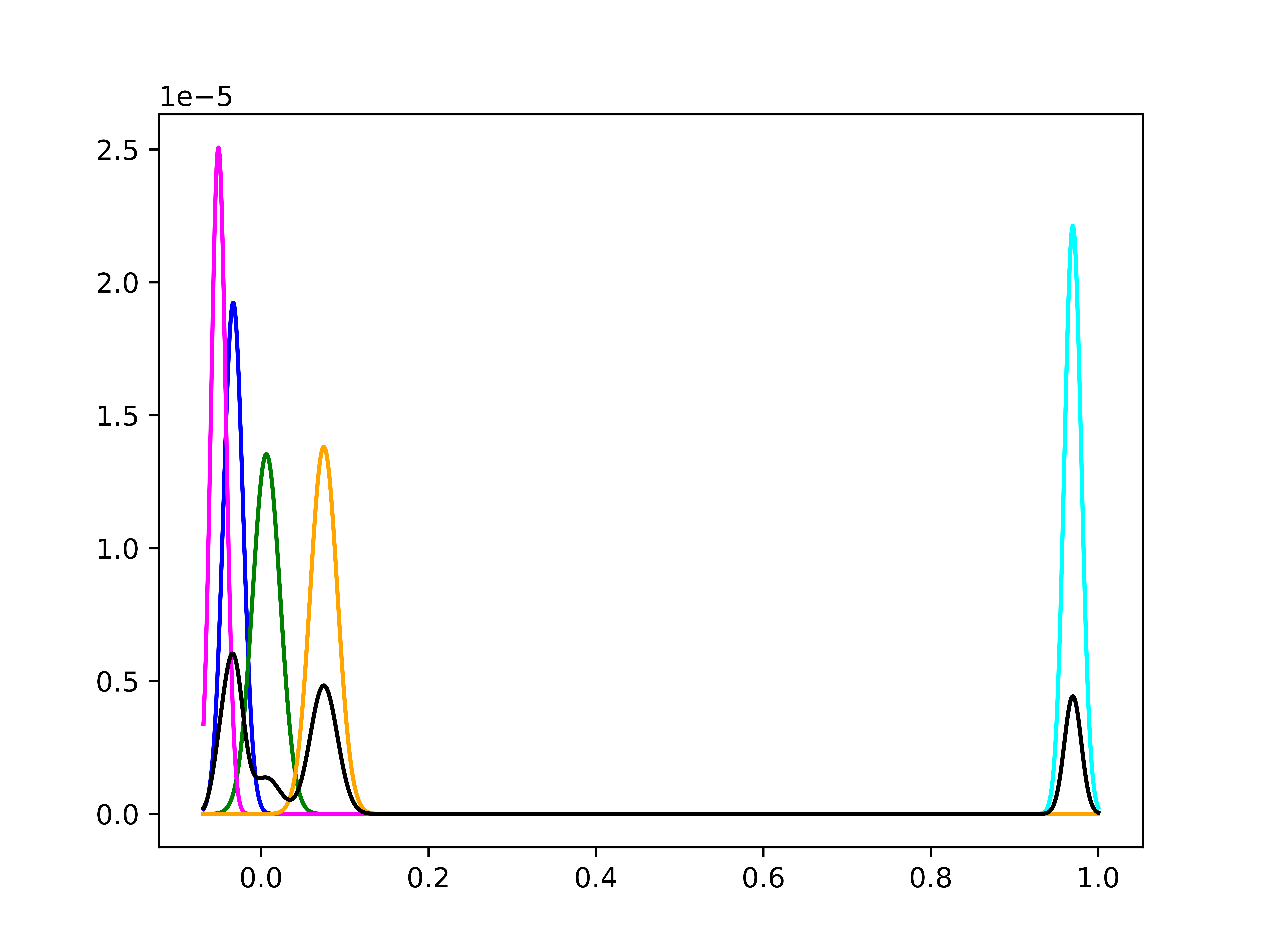}  & 
    \includegraphics[width=0.32\textwidth]{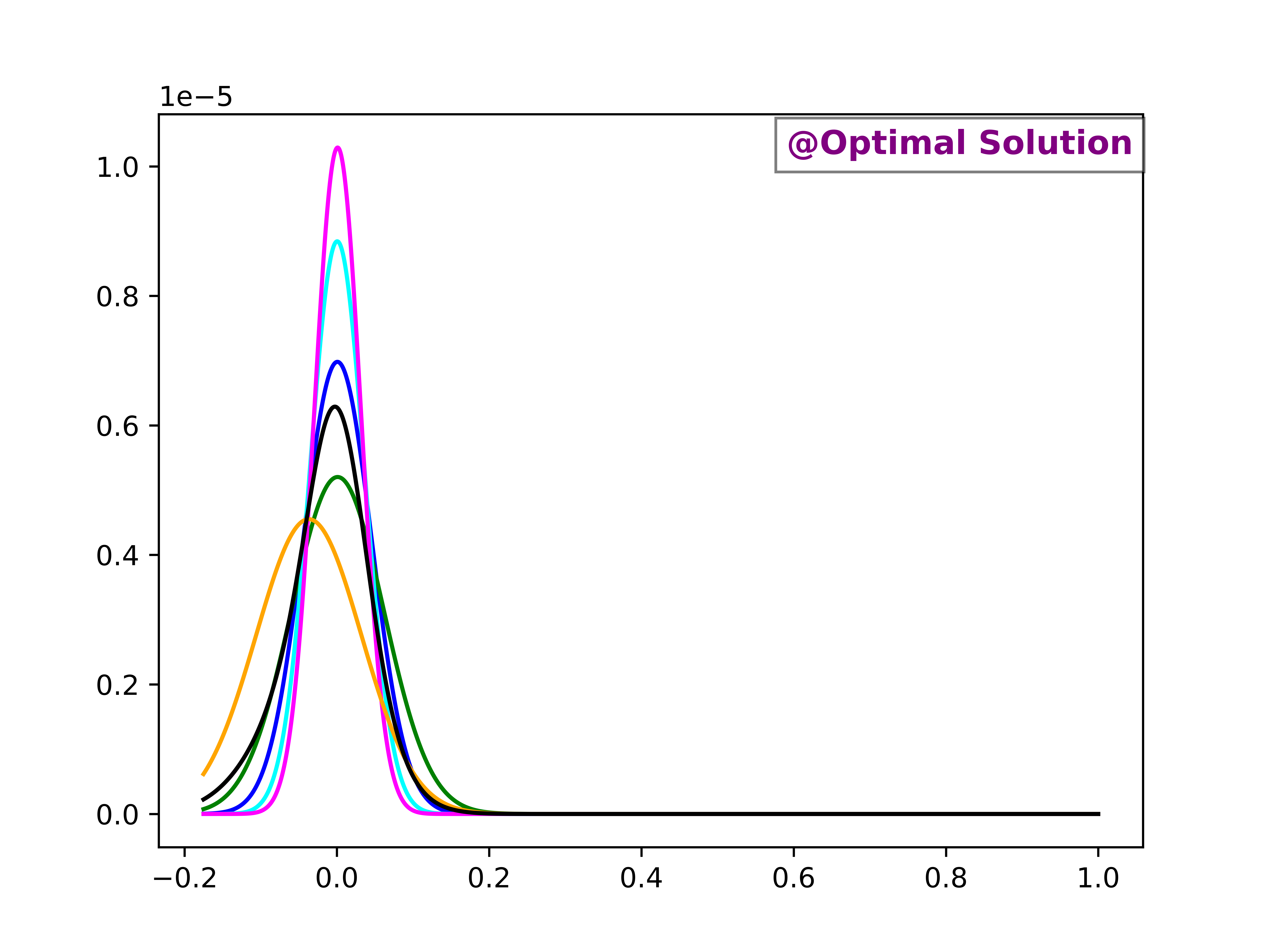} &
    \includegraphics[width=0.32\textwidth]{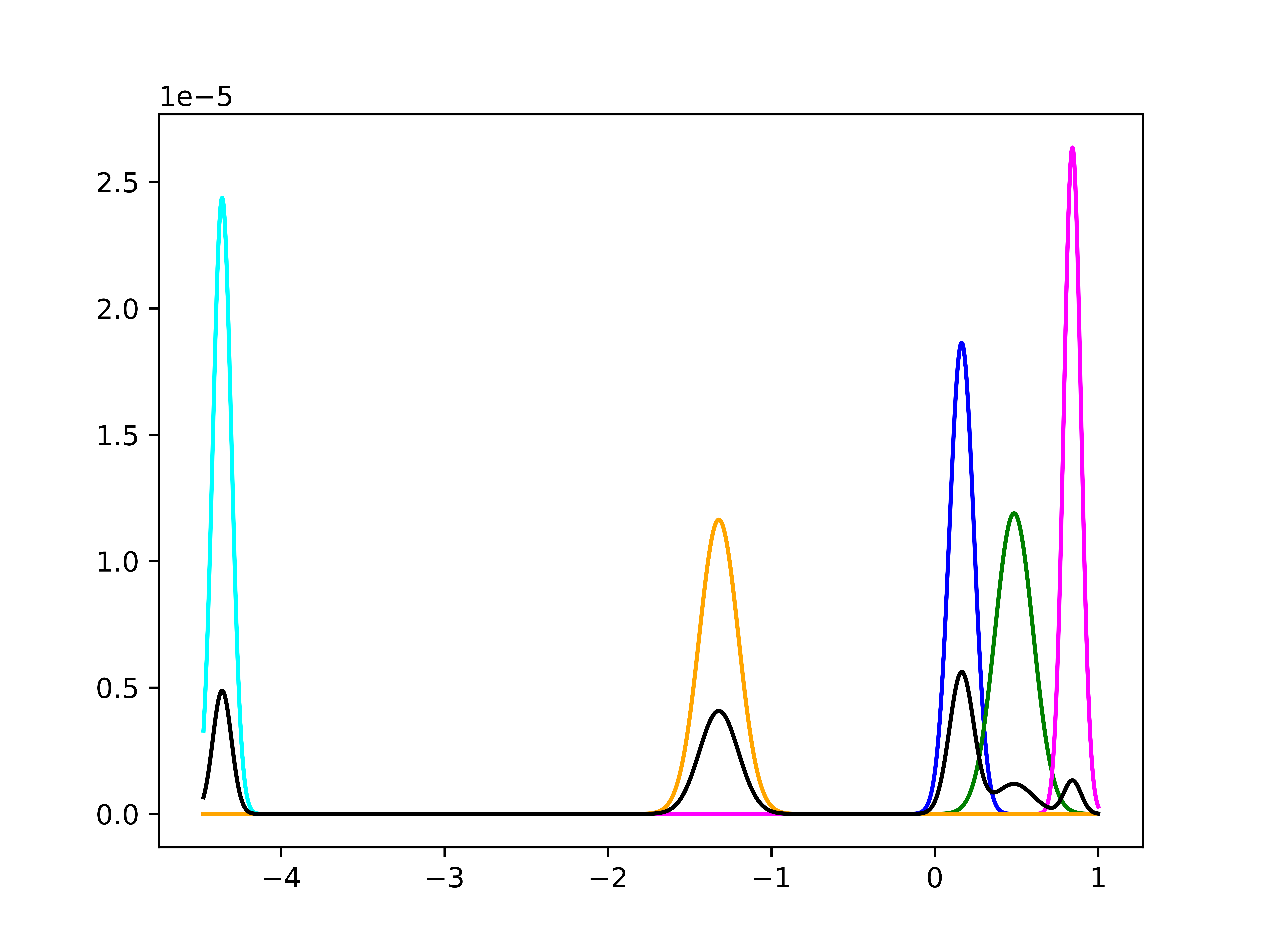} \\
        \vspace{-10pt}
    \includegraphics[width=0.32\textwidth]{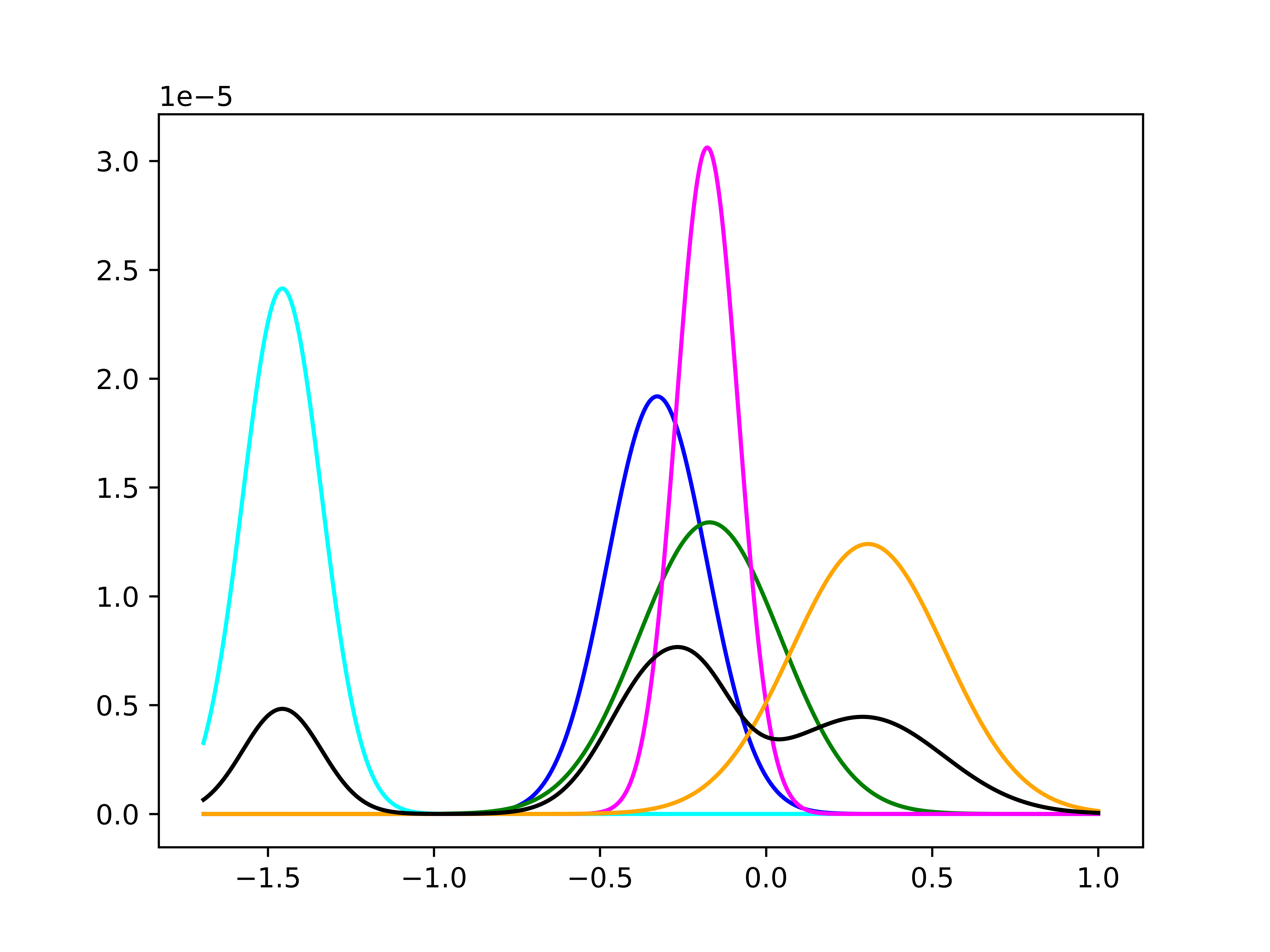} &
    \includegraphics[width=0.32\textwidth]{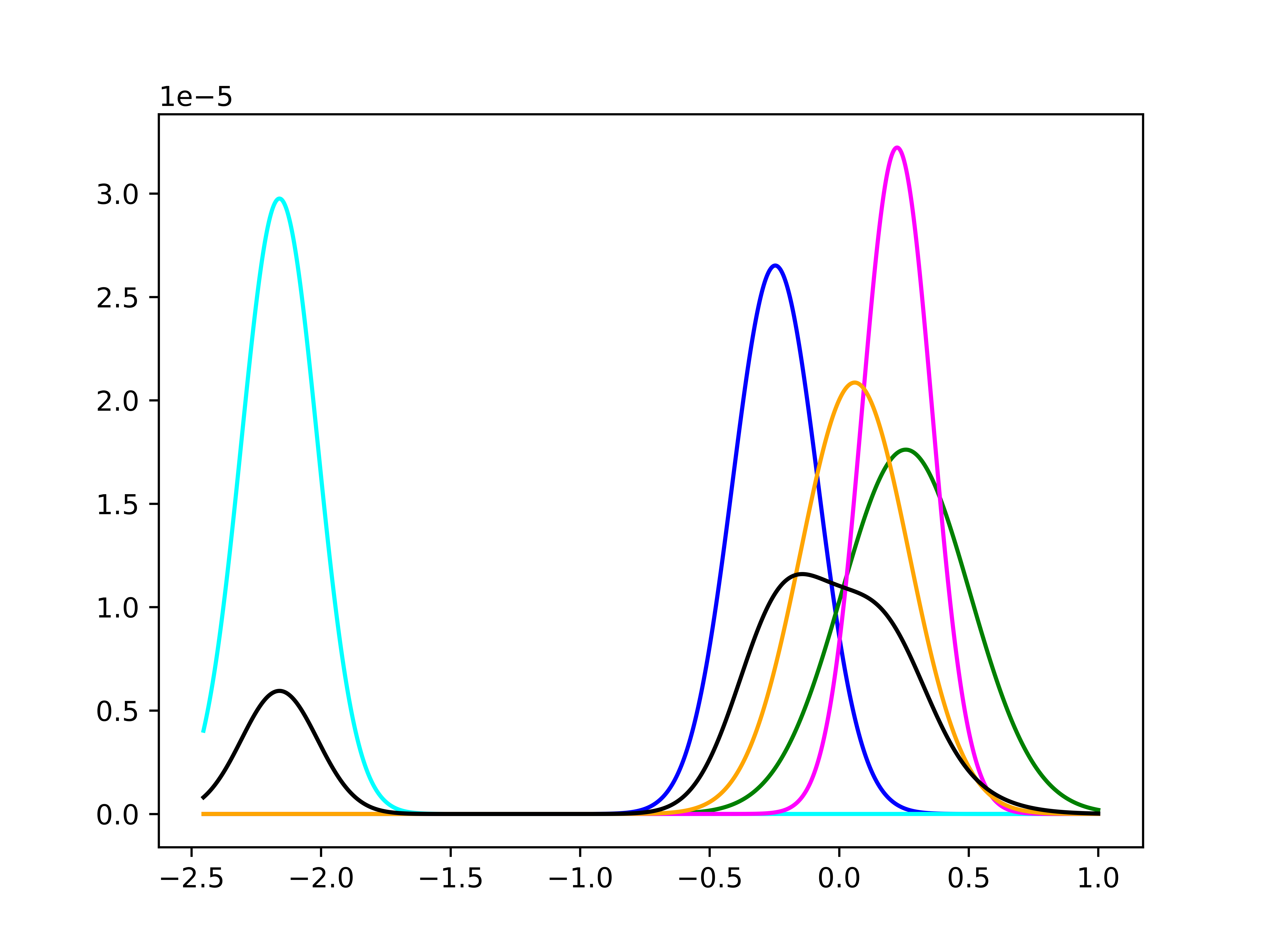} & 
    \includegraphics[width=0.32\textwidth]{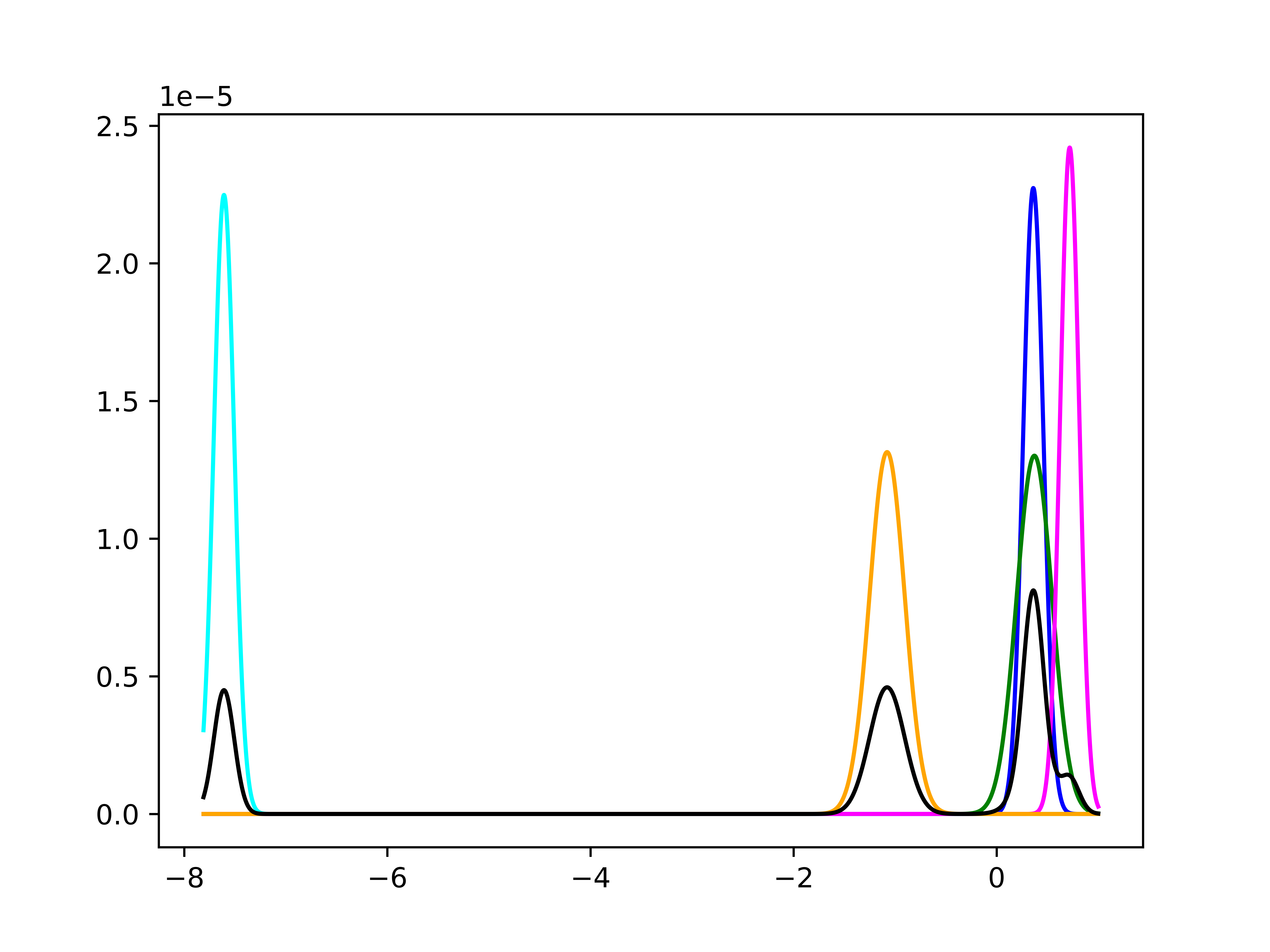}
        \vspace{-20pt}
\end{longtable}
\captionof{figure}{\centering \small \textbf{Probability density of $\bxi_k^\top \bx$ for $k=1, 2, \ldots, 5$ in blue, green, cyan, magenta, and orange colors, respectively, and of $\bxi^\top \bx$ in black color. Each plot uses a distinct $\bx$ sampled from the feasible region.\label{table:Fig_tables}}}%
\addtocounter{table}{-1}%
\end{center}
\vspace{-10pt}

We observe from Figure~\ref{table:Fig_tables} that for certain decisions (e.g., middle row and bottom right), the means of most component-specific distributions tend to cluster around a particular value. In contrast, for certain others (e.g., top left and right), the means are spread further apart. Additionally, certain component-specific distributions can completely shift from the negative end of the $x$-axis (e.g., blue, magenta, green in the middle left plot) to the positive end (e.g., see bottom and middle right plots). Moreover, the mass of $\bxi^\top \bx$ can vary significantly (e.g., middle and bottom right) with the decision.
The plots highlight the difficulty in estimating or predicting the location and shape of the overall mixture distribution, let alone the individual component-specific distributions, at the optimal solution $\bx^*(\theta)$.

\subsection{Contributions}\label{sec:contributions}
The key technical challenge lies in handling the weighted sum of the cumulative distribution functions in reformulation~\eqref{eq:cc_probability_standardized_form}, each of which is composed with a nonlinear fractional function of $\bx$.
When $\bxi$ is multivariate normal (i.e., when $K = 1$), it is well known that the chance constraint can be reformulated as a tractable second-order conic constraint \cite{ruszczynski2009stochastic}.
Unfortunately, this is no longer the case when $\bxi$ follows a GMM with $K > 1$ components.

To address this gap, we develop and analyze approximations of problem~\eqref{eq:ccp}. This is achieved by exploiting the special convex-concave structure of the univariate standard normal CDF $\Phi$ and replacing it with an efficient piecewise linear (PWL) approximation. We provide mixed-integer optimization formulations that can be iteratively refined to either inner- or outer-approximate the feasible set, namely $P(\theta)$, to any arbitrary user-specified accuracy. We refer to these formulations as `PWL-I' and `PWL-O', respectively. Unlike previous studies (reviewed in the next section), we do not impose any restrictive structural assumptions on the GMM, such as dependence structures among the component covariance matrices, or bounds on the total number of components.
Under suitable regularity assumptions, we also show that our formulations provide solutions up to any desired optimality tolerance. 
All of the proposed formulations can be solved using standard mixed-integer quadratic programming solvers. 

We perform an extensive computational study on the PWL-I, and PWL-O formulations, across five thousand synthetically generated problem instances with 100, 500, and 1,000 variables and up to fifteen mixture components for chance constraint satisfaction levels $\theta$ set at $0.95$, $0.99$, and $0.999$. If the complement of the chance constraint in~\eqref{eq:cc} models an unsafe system condition that causes failure, then $\theta = 0.999$ can be interpreted as a minimum reliability requirement of $99.9\%$, or equivalently, as a maximum allowed failure probability of $0.001$, thus modeling a decision-dependent rare event \cite{subramanyam2022chance,blanchet2024optimization}. 
Our experimental findings show that the majority of PWL-I and PWL-O approximation models are solved to optimality within a time limit of 18~hours, while also attaining the desired probabilities of constraint satisfaction. Most instances consisting of ten and five mixture components are solved within 12 and 4 hours, respectively. 
Finally, we also compare the computational performance of the proposed formulations with classical sample average approximations (SAA) \cite{luedtke2008sample, pagnoncelli2009sample, kim2015guide}.
We find that the SAA models take very large computational times and are unable to produce solutions with the desired optimality or feasibility tolerances.

\subsection{Notation}
All vectors and matrices as well as vector- and matrix-valued functions are typeset in boldface, whereas all scalars and scalar-valued functions are typeset in normal font.
The cumulative distribution and the probability density functions of the standard normal random variable are denoted by $\Phi$ and $\phi$, respectively.
The derivative of $\phi$ is denoted by $\phi'$. 
We use $\mathbb{R}_{+}$ to denote the set of non-negative real numbers. 
For any non-negative integer $N$, we use $[N]$ to denote the set $\{1, 2, \ldots, N\}$ and $[N]_0$ to denote $\{0, 1, 2, \ldots, N\}$.
Unless stated otherwise, $\|\cdot\|$ denotes the Euclidean norm for vectors and the Frobenius norm for matrices; we sometimes also use $\|\cdot\|_F$ to denote the latter. 
We use $\stackrel{\mathcal{U}}{\sim}$ to indicate sampling from a uniform distribution.

\section{Related Literature}\label{sec:literature}

For an overview of the theory and algorithms for chance-constrained optimization problems, we refer the reader to \cite{prekopa2003,dentcheva2009optimization}.
Approaches based on distributionally robust optimization are reviewed in \cite{xie2021distributionally,kucukyavuz2022chance,pang2023chance}.
For the models we study in this paper, the decision vector $\bx$ and the random parameters $\bxi$ are not separable. In the case of separable chance-constrained models, one can move the uncertainty entirely to the right-hand side. In this case, existing approaches are based on integer programming ideas \cite{sen1992relaxations,luedtke2010integer, kuccukyavuz2012mixing} and so-called $p$-efficient points to approximate the feasible set \cite{prekopa1990dual, dentcheva2013regularization}.

In the general non-separable case, existing methodologies for solving chance-constrained optimization problems can be categorized into sample-based and sample-free analytical approaches. 
The sample average approximation (SAA) method \cite{calafiore2005uncertain,nemirovski2006scenario,blanchet2024efficient,choi2024reduced, luedtke2008sample,pagnoncelli2009sample}, belonging to the first category, is a general strategy to approximate the probability in \eqref{eq:cc}. The method is based on drawing several, say $N$, i.i.d. Monte Carlo samples, $\bxi^1, \bxi^2, \ldots, \bxi^N$, of the random parameters and then approximating the probabilistic constraint function using the empirical estimate based on these samples,
\vspace{-15pt}
\begin{equation}\label{eq:saa}
p(\bx)=\mathbb{E}[\mathbbm{1}_{\geq 0}\left(b - \bx^\top \bxi \right)] \approx \frac{1}{N}\sum_{i=1}^N \mathbbm{1}_{\geq 0}\left(b - \bx^\top \bxi^i \right),
\end{equation}
where $\mathbbm{1}_{\geq 0}(z)$ is the indicator function of the non-negative reals that equals $1$ if $z \geq 0$ and $0$ otherwise.
Alternative sample-based approaches include quasi-Monte Carlo-based methods tailored for elliptically symmetric distributions \cite{vanAckooij2014,van2018sub}.
The quality of the resulting SAA solutions depends critically on the number $N$ of generated samples.
A key issue is their deviation from the true (unknown) optimal solution. The sampled-based solution might be unstable  when $N$ is small, meaning that small changes in the samples may have a significant impact on the obtained solutions \cite{henrion2004holder,ahmed2008solving,demiguel2009generalized}. Larger sample sizes are thus required to better approximate the optimal solution, but they come at the cost of increased computational difficulty \cite{Henrion2013:sample_size, Wang2021ChanceConstrainedBinPacking}. This is exacerbated in high-dimensional regimes where incorporating large sample sizes renders the solution of the SAA models impractical (see Section~\ref{sec:ComputationalResults}).  

Indeed, the challenge in SAA is the discontinuity of the indicator function $\mathbbm{1}_{\geq 0}$ which needs to be further reformulated. This reformulation is typically done using additional binary variables to yield a mixed-integer optimization problem \cite{luedtke2010integer,curtis2018sequential}, or approximated using continuous nonlinear functions, such as difference-of-convex functions \cite{hong2011sequential}, smooth differentiable functions \cite{geletu2017inner,pena2020solving}, or monotone convex functions \cite{nemirovski2007convex} that also subsume the conditional-value-at-risk (CVaR) approximation. Each of these strategies requires introducing at least $N$ new variables or constraints, thus increasing the computational complexity of solving the chance-constrained problem.

In contrast to sample-based methods, sample-free approaches typically exploit the structure of the distribution of $\bxi$ to obtain analytically computable approximations or reformulations of the chance constraints. Early pioneering results are based on so-called $\alpha$-concave measures and functions \cite{prekopa1971logarithmic,prekopa1973logarithmic, borell1974convex}. In these cases, one can show that feasible region is convex, so that classical results in convex optimization can be adapted to design algorithms and optimality conditions. 
Convexity of the feasible set can also be achieved when $\bxi$ has symmetric logarithmically concave density function \cite{lagoa2005probabilistically}. 

Another line of work aims to build (convex) outer approximations of the chance constraints. 
The popular CVaR approximation and other approaches based on Bernstein-type large deviation inequalities are examples of this approach \cite{rockafellar2000optimization,nemirovski2007convex}. More recent sample-free approaches based on large deviation theory develop analytical closed-form (albeit nonlinear) expressions of the probabilistic constraint function for GMM distributed uncertainties \cite{tong2022optimization}.

The use of GMMs to model uncertainties in stochastic optimization models has gained significant interest in recent literature. 
Applications of GMM to quantify uncertainties can be found in portfolio optimization \cite{luxenberg2024portfolio}, chemical engineering \cite{yang2023optimal,yang2024two}, and power system operations \cite{ke2015novel,yang2019analytical,xu2021tractable,fathabad2023asymptotically,yi2024discrete}. 
Joint chance-constrained optimization problems involving Gaussian mixture distributed randomness have been studied in \cite{yang2023optimal, boone2022spacecraft}, although such problems typically rely on using Boole's inequality to obtain conservative individual chance-constrained problems. 

The most closely relevant studies to our work are \cite{hu2022chance, fathabad2023asymptotically, pang2023chance,wei2024enhanced}. All of these consider GMM-distributed uncertainties affecting either individual or two-sided chance constraints. In \cite{hu2022chance,pang2023chance}, methods based on nonlinear optimization techniques, specifically gradient-based and spatial branch-and-bound algorithms, are proposed to solve chance-constrained problems to global optimality. 
The work of~\cite{wei2024enhanced} improves the performance of this spatial branch-and-bound algorithm by proposing an enhanced pruning strategy. In  \cite{fathabad2023asymptotically}, a second-order cone programming (SOCP) model is proposed as a candidate approximation of chance-constrained problems under GMM-distributed uncertainties. Similar to our work, this reformulation also uses a piecewise linear approximation of the standard normal CDF by selecting an optimal number of breakpoints to obtain an SOCP approximation. However, the component-specific covariance matrices in their GMM are assumed to be proportional to the same general covariance matrix: $\bSigma_k = \eta_k \bSigma$ for some fixed $\eta_k > 0$. Also, their resulting SOCP model is optimal only when the risk threshold $\theta$ is sufficiently large; namely, when $\theta \geq 1 - \frac12\min\{w_1, w_2, \ldots, w_K\}$. In contrast, our study considers a general form of the GMM without additional constraints on the component-specific weights, means, or covariances, and focuses on finding provably optimal solutions of linear chance-constrained models without restrictive assumptions on its problem parameters.

\section{Mixed-Integer Optimization Formulations}\label{sec:reform_approx}

We first present a mixed-integer quadratically constrained reformulation 
of the chance-constrained optimization problem~\eqref{eq:ccp} that motivates our subsequent developments. Our key insight is to isolate and separate the nonconvexities in the probability function $p(\bx)$, shown in~\eqref{eq:cc_p_x_definition}, stemming from the fractional term, $(b - \bmu_k^\top \bx)/\sqrt{\bx^\top \bSigma_k \bx}$, and from the standard normal CDF, $\Phi$.
\begin{proposition}\label{prop:reformulation0}
	Problem~\eqref{eq:ccp} has the following equivalent reformulation.
    \vspace{-10pt}
	\begin{equation}\label{eq:ccp_intermediate}
		\begin{aligned}
			\min \; & \,\, \bc^\top \bx \\
			\text{s.t.} \; 
			&  \, \, \sum_{k=1}^K w_k \zeta_k \geq
			\theta, 
            \quad \bx \in \mathcal{X},
            \quad \bz \in \mathbb{R}^K, 
            \quad \bzeta \in [0, 1]^K, 
            \quad \blambda \in \mathbb{R}^K_{+},
			  \\
			& \,\, \left.\begin{aligned}
				& \Phi \left( z_k \right) \geq \zeta_k, \quad  b - \bx^\top \bmu_k \geq z_k \, \lambda_k,  \quad  \bx^\top \bSigma_k \bx = \lambda_k^2 
                \end{aligned}\right\}  \; k \in [K].
		\end{aligned}
	\end{equation}
\end{proposition}

\begin{proof}
    See Appendix~\ref{appndx:proof}.
\end{proof}

Observe that the original non-convexities are captured in three separate sets of constraints for each $k \in [K]$; namely,
\textit{(i)} the constraint, $\Phi \left( z_k \right) \geq \zeta_k$, specified using the standard normal CDF,
\textit{(ii)} the bilinear constraint with the term $z_k\lambda_k$, and
\textit{(iii)} the quadratic equality constraint. 
Note that each of these three sets of constraints is non-convex.
Recent developments in quadratic optimization solvers allow efficient handling of the nonconvex bilinear and quadratic constraints. Therefore,  we only focus on approximating the constraint, $\Phi \left( z_k \right) \geq \zeta_k$, using suitable piecewise linear (PWL) approximations. 
In particular, we exploit the convex-concave structure of $\Phi(z)$, which is convex for $z \in (-\infty, 0]$ and concave for $z \in [0, \infty)$.

\subsection{Piecewise Linear Outer Approximation}\label{sec:PWL-outer}
We now provide an outer approximation  model called `PWL-O' whose optimal objective value provides a lower bound on the optimal objective value of the original chance-constrained problem~\eqref{eq:ccp} and which ``under-satisfies'' the desired chance constraint probability. 
For notational simplicity, we temporarily drop subscript $k$ and focus on approximating $\Phi(z) \geq \zeta$, where $z, \zeta \in \mathbb R$.

Let $\bzz = (\zz_{-L}, \zz_{-L+1}, \ldots, \zz_{-1}, \zz_{0}, \zz_{1}, \ldots, \zz_{R-1}, \zz_{R}) \in \mathbb{R}^{L+R+1}$ denote a valid array of breakpoints, parameterized by integers $L > 0$ and $R > 0$ with
\vspace{-10pt}
\begin{equation}
  -\infty < \zz_{-L} < \zz_{-L+1} < \ldots  < \zz_{-1} < \zz_{0} = 0 < \zz_{1} < \ldots < \zz_{R-1} < \zz_{R} < \infty.   \label{eqn:ValidArrayBreakPoints}
  \vspace{-5pt}
\end{equation}
The proposed piecewise linear outer approximation uses $L$ and $R+1$ linear pieces on the non-positive (left) and non-negative (right) sides of the domain of $\Phi$, respectively. We allow $L \neq R$; as we shall later discuss, the approximation of $\Phi$ on the convex and concave sides of $\Phi$ allows identification of a different number of breakpoints to achieve the same accuracy.  We define the following quantities: 
\vspace{-10pt}
\begin{equation*}
\begin{alignedat}{3}
	& g_i \coloneqq \phi(\zz_i), \; && g_i^0 = \Phi(\zz_i) - \phi(\zz_i) \zz_i, \; && i \in [R]_0, \\
	& h_i \coloneqq \frac{\Phi\left(\zz_{-i+1}\right) - \Phi\left(\zz_{-i}\right)}{\zz_{-i+1} - \zz_{-i}}, \; && h_i^0 = \Phi(\zz_{-i}) - \frac{\Phi\left(\zz_{-i+1}\right) - \Phi\left(\zz_{-i}\right)}{\zz_{-i+1} - \zz_{-i}} \zz_{-i}, \; && i \in [L],
\end{alignedat}
\end{equation*}
where $g$ represents approximation of $\Phi$ on the non-negative reals and $h$ represents its approximation on the negative reals. Given the array of breakpoints $\bzz$, we define the PWL outer approximation of $\Phi$ as follows:
\vspace{-10pt}
\begin{equation}\label{def:def_phi_pwl_outer}
\bar{\Phi}(z; \bzz) \coloneqq
\begin{dcases*}
    \min\left\{1, \min_{i \in [R]_0} \left\{ g_i z + g_i^0 \right\} \right\},
    & if $z \geq 0$, \;\;\; \text{(Tangent)} \\
    \max\left\{\Phi(\zz_{-L}), \max_{i \in [L]} \left\{ h_i z + h_i^0 \right\} \right\}, 
    & otherwise. \; \text{(Secant)}
\end{dcases*}
\end{equation}
\begin{remark}
    Note that \eqref{def:def_phi_pwl_outer} ensures correctness of the outer approximation of $\bar{\Phi}(z, \zz)$ for $z \geq \zz_{R}$ and $z \leq \zz_{-L}$, respectively, by using $ \min\left\{1, \min_{i \in [R]_0} \left\{ g_i z + g_i^0 \right\} \right\}$ and $\max\left\{\Phi(\zz_{-L}), \max_{i \in [L]} \left\{ h_i z + h_i^0 \right\} \right\}$ to approximate the tails of the Gaussian distribution. In particular, $\bar{\Phi}(z, \zz)$ equals $\Phi(\zz_{-L})$ for any $z \leq \zz_{-L}$, and $\bar{\Phi}(z, \zz)$ equals $1$ for $z \geq \zz_{R}$. As we shall show, this allows us to bound the outer approximation error even when $z$ is outside $[\zz_{-L}, \zz_{R}]$. 
\end{remark}

By construction, the graph of $\bar{\Phi}(\cdot; \bzz)$, is above the graph of $\Phi$, as shown in Figure~\ref{fig:pwl_phi_outer}.
This is because on the non-negative real line $(z \geq 0)$, it uses a tangent-based PWL approximation, whereas on the negative real line $(z < 0)$, it uses a secant-based PWL approximation. 
\begin{figure}[h]
\centering
\includegraphics[width=0.5\textwidth]{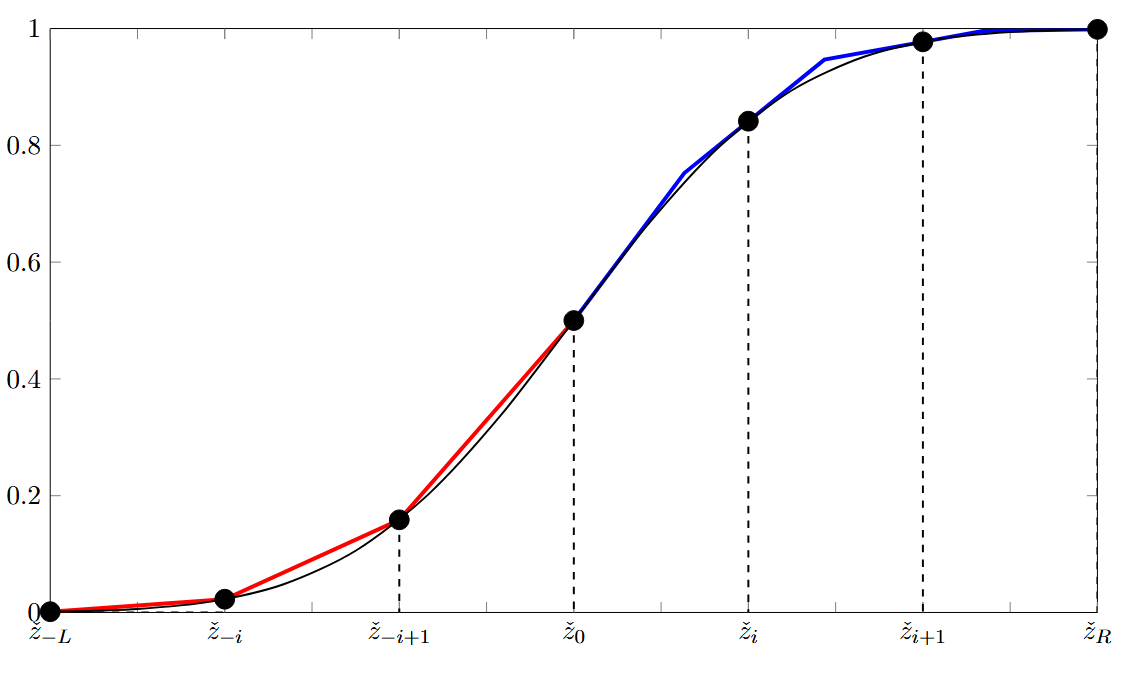}
\caption{\centering Piecewise linear outer approximation of $\Phi$.}
\label{fig:pwl_phi_outer}
\vspace{-1em}
\end{figure}

Similar to $\Phi(z)$, the PWL approximation in \eqref{def:def_phi_pwl_outer} is also convex on the domain $z \leq 0$, and concave when $z \geq 0$. On the concave part of the domain, constraint $\bar{\Phi}(z; \bzz) \geq \zeta$ can be represented using linear constraints, whereas on the convex part it is reformulated using Type-2 Special Ordered Set (SOS2) constraints. This is formalized in the following proposition.

\begin{proposition}\label{prop:phi_pwl_outer_reformulation}
	If $\bzz$ is any valid array of breakpoints as in \eqref{eqn:ValidArrayBreakPoints} and $\ubar{z}, \bar{z} \in \mathbb R$ are given numbers satisfying $\ubar{z} \leq \zz_{-L} < \zz_{R} \leq \bar{z}$, then there exist $(z, \zeta) \in [\ubar{z}, \bar{z}] \times \mathbb R$ satisfying $\bar{\Phi} \left(z ; \bzz\right) \geq \zeta$ if and only if $\mH(\bzz) \neq \emptyset$, where
    \vspace{-20pt}
    	\begin{equation}\label{eq:reform-outer-sub}
		\mH(\bzz) \coloneqq 
		\Set*{
			\begin{aligned}
				&\balpha \in \mathbb{R}_{+}^{L}, \\
				& \bt \in \{0, 1\}^3, \\
				& \by \in \mathbb{R}^3, \\
				& z, \zeta \in \mathbb{R}
			\end{aligned}
		}
		{
			\begin{aligned}
				& \Phi(\zz_{-L})t_1 + t_2 + g_i y_3 + g_i^0 t_3  \geq \zeta \quad\forall i \in [R]_0, \\
				& \Phi(\zz_{-L})t_1 + \sum_{i=0}^{L} \alpha_{i} \Phi(\zz_{-i})  + t_3 \geq \zeta, \\
				& t_1 + t_2 + t_3 = 1, \; z = y_1 + y_2 + y_3, \\
				& - t_1 \ubar{z} \leq y_1 \leq t_1 \zz_{-L}, \quad 0 \leq y_3 \leq t_3 \bar{z},\\
				& y_2 = \sum_{i=0}^{L} \alpha_{i} \zz_{-i}, \;\; t_2 = \sum_{i=0}^{L} \alpha_{i}, \;\;  \bm{\alpha} \in \mathop{\mathrm{SOS2}}
			\end{aligned}
		}.
	\end{equation}
\end{proposition}
\begin{proof}
    See Appendix\ref{appndx:proof}.
\end{proof}
In formulation~\eqref{eq:reform-outer-sub}, the binary variables, $t_1$, $t_2$, and $t_3$, are used to indicate if $z \in [\ubar{z}, \zz_{-L}]$, $z \in (\zz_{-L}, 0)$, and $z \in [0, \bar{z}]$, respectively.
The nonnegative variables $\balpha$ can be nonzero only when $z \in (\zz_{-L}, 0)$.
The SOS2 constraint ensures that at most two components of $\balpha$ are positive, and that if two components are indeed positive, then these component indices are consecutive integers in the set $[L]_0$.
Several modern mixed-integer optimization solvers enforce SOS2 constraints algorithmically during branching, as opposed to reformulating them as explicit constraints.

An application of Proposition~\ref{prop:phi_pwl_outer_reformulation} to each of the $K$ constraints, $\bar{\Phi}(z_k) \geq \zeta_k$, yields the following MIQP approximation of the reformulated problem~\eqref{eq:ccp_intermediate}, and hence, of the original chance constrained problem~\eqref{eq:ccp}:
\begin{equation}\label{eq:ccp_pwl-outer}
	\begin{aligned}
		\mathop{\text{min}} \; & \, \, \bc^\top \bx \\
		\text{s.t.} \; 
		& \, \,  \sum_{k=1}^K w_k \zeta_k
		\geq
		\theta, \quad \bx \in \mathcal{X}, \quad   \bz \in \mathbb{R}^K, \quad \bzeta \in [0, 1]^K, \quad \blambda \in \mathbb{R}^K_{+} \\
		& \, \, \left.\begin{aligned}
			& (\bm{\alpha}_k,  \bt_k, \by_k, z_k, \zeta_k) \in \mH(\bzz), \quad  b - \bx^\top \bmu_k \geq z_k \lambda_k, \quad  \bx^\top \bSigma_k \bx = \lambda_k^2 \ \\
		\end{aligned}\right\} \; k \in [K]. 
	\end{aligned}
\end{equation}

The quality of the above PWL-O model depends on the accuracy of $\bar{\Phi}$, which in turn depends on the vector of breakpoints $\bzz$.
The following lemma, inspired from \cite[Appendix~C.1]{Luo2021DecompositionWasserstein}, shows that $\bar{\Phi}$ is a $\tau$-accurate outer-approximation of $\Phi$ whenever the discretization of the breakpoint array, $\bzz$, is small in regions where the magnitude of the second derivative of $\Phi$ is large and when $\zz_{-L}$ and $\zz_{R}$ are sufficiently far from the origin.

\begin{lemma}\label{lem: epsilon-reltn-z-phi-outer}
For any $z_1, z_2 \in \mathbb R$,  let $C(z_1, z_2) \coloneqq \max_{z \in [z_1, z_2]} |\Phi''(z)|$, $\bzz$ be any valid array of breakpoints as in \eqref{eqn:ValidArrayBreakPoints}, and $\tau >0$. Then the following holds:
\begin{itemize}
    \item[a)] If $\zz_{i} - \zz_{i-1} \leq \sqrt{\frac{2 \tau}{C(\zz_{i-1}, \zz_{i})}}$ for all $i \in [R]$, then $0 \leq \bar{\Phi}(z; \bzz) - \Phi(z) \leq \tau$ for all $z \geq 0$.

    \item[b)] If $\zz_{-i+1} - \zz_{-i} \leq 2 \sqrt{\frac{2 \tau}{C(\zz_{-i}, \zz_{-i+1})}}$ for all $i \in [L]$, then $0 \leq \bar{\Phi}(z; \bzz) - \Phi(z) \leq \tau $ for all $z < 0$.
\end{itemize}
\end{lemma}
\begin{proof}
    See Appendix\ref{appndx:proof}.
\end{proof}

Lemma~\ref{lem: epsilon-reltn-z-phi-outer} generalizes the approximation conditions established in~\cite{fathabad2023asymptotically, Luo2021DecompositionWasserstein}. It shows that we only need the discretization interval to be small in regions where $C(\zz_{i-1}, \zz_{i})$ is large.
This is particularly important to control the number of breakpoints since  $|\Phi''(z)|$ is small for large values of $z$ (see Figure~\ref{fig:Breaks points and Curvature change in Standard Normal CDF}). The number of breakpoints differs in (a) and (b) by a factor of 2 because, in comparison to the secant-based approximation, the tangent-based approximation has lower approximation error estimates. 
\begin{figure}[!htbp]
\center
\includegraphics[width=0.55\textwidth, height=0.35\linewidth]{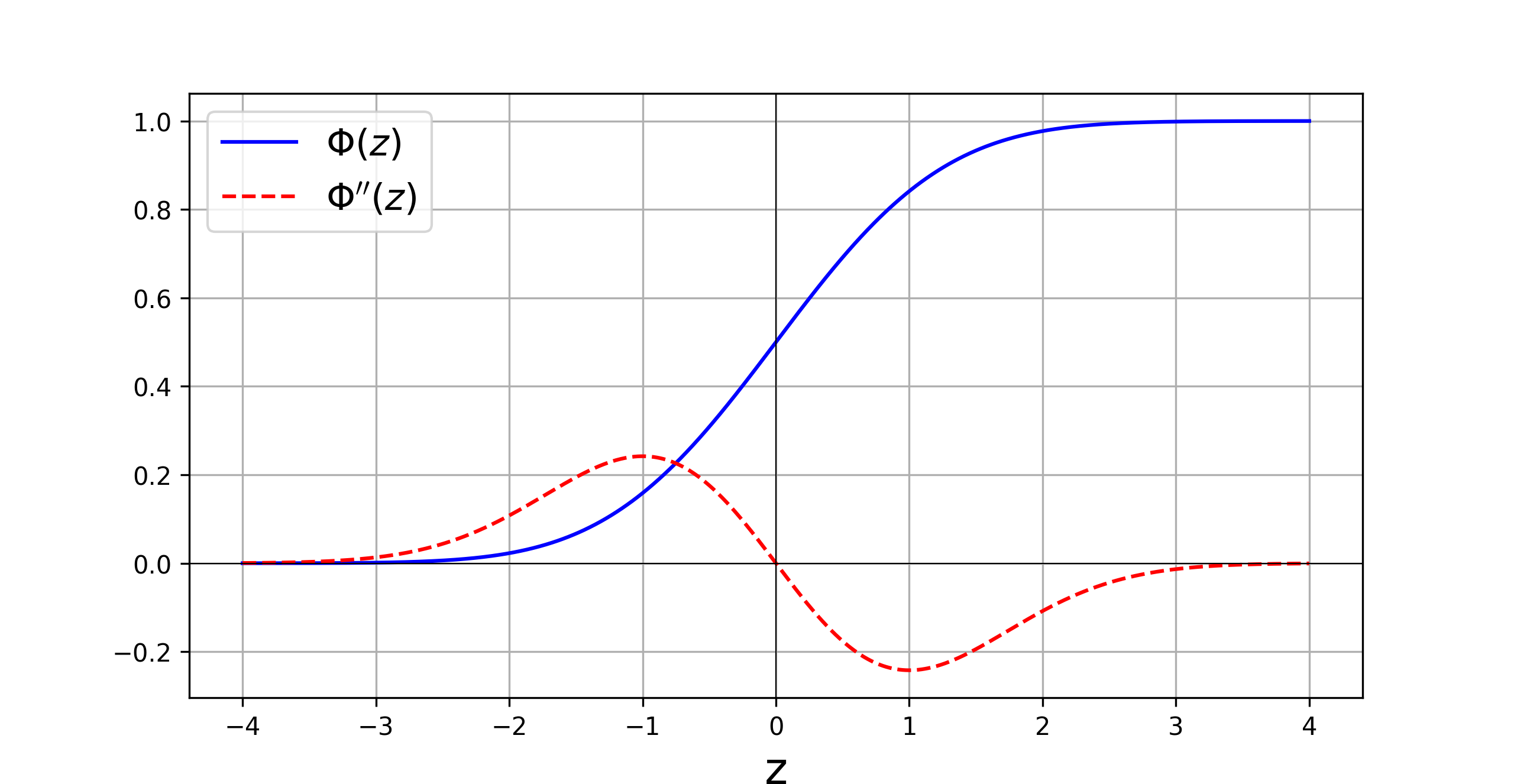}
\caption{\centering CDF of standard Normal distribution and its second derivative} \label{fig:Breaks points and Curvature change in Standard Normal CDF}
\end{figure}

\vspace{-5pt}
\begin{table}[h]
\small
  \centering
  \captionsetup{type=algorithm} 
  \caption{Breakpoints for approximating $\Phi(z)$}
  \label{tab:phi-breakpoints}

  \begin{subtable}[t]{0.48\linewidth}
  \vspace{-10pt}
    \centering
    \caption{Tangent approximation of $\Phi(z), z \ge 0$}
    \label{alg:breakpoint_find}
    \begin{algorithmic}[1]
      \Require{$\zz_{R}, \tau$}
      \Ensure{$\mathcal{A}^R$}

      \State \textbf{Part 1:} $0\leq z \leq 1$
      \State $i \gets 1, \check{z}_{i} \gets 1, \mathcal{A}^{R} \gets \mathcal{A}^{R} \cup \{\check{z}_{i}\}$
      \State $ \check{z}_{i+1} =  \check{z}_i - \sqrt{\frac{2\tau}{\frac{e^{-0.5}}{\sqrt{2 \pi}}}}$
      \While{$\check{z}_{i+1} > 0$}
      \vspace{5pt}
      \State $\mathcal{A}^{R} \gets \mathcal{A}^{R} \cup \{\check{z}_{i+1}\}, i \gets i + 1$
      \State $\check{z}_{i+1} = \check{z}_{i} -  \sqrt{\frac{2\tau}{\phi(\check{z}_{i})\check{z}_{i}}}$
      \EndWhile
      \State $\mathcal{A}^L \gets \mathcal{A}^{R} \cup \{0\}, i \gets i+1$
      \State Reverse the order of $\mathcal{A}^{R}$
      \State $\mathcal{A}^{R} \gets \mathcal{A}^{R} \backslash \{1\}$
      \State \textbf{Part 2:} $1\leq z \leq \zz_{R}$
      \State $i \gets i + 1, \check{z}_{i-1} \gets 1, \mathcal{A}^{R} \gets \mathcal{A}^{R} \cup \{\check{z}_{i-1}\}$
      \State $\check{z}_i = \check{z}_{i-1} + \sqrt{\frac{2\tau}{\phi(\check{z}_{i-1})
      \check{z}_{i-1}}}$
      \While{$\check{z}_i < \zz_{R}$}
      \vspace{5pt}
      \State $\mathcal{A}^{R} \gets \mathcal{A}^{R} \cup \{\check{z}_i\}, i \gets i + 1$
      \State $\check{z}_i = \check{z}_{i-1} + \sqrt{\frac{2\tau}{\phi(\check{z}_{i-1})
      \check{z}_{i-1}}}$
      \EndWhile
      \State $\mathcal{A}^{R} \gets \mathcal{A}^{R} \cup \{\zz_{R}\}$
    \end{algorithmic}
  \end{subtable}\hfill
  \begin{subtable}[t]{0.48\linewidth}
    \vspace{-10pt}
    \centering
    \caption{Secant approximation of $\Phi(z)$, $z < 0$}
    \label{alg:breakpoint_find_negative}
    \begin{algorithmic}[1]
      \Require{$\zz_{-L}, \tau$}
      \Ensure{$\mathcal{A}^{L}$}

      \State \textbf{Part 1:} $-1\leq z \leq 0$
      \State $i \gets 1, \check{z}_{-i} \gets -1, \mathcal{A}^{L} \gets \mathcal{A}^{L} \cup \{\check{z}_{-i}\}$
      \State $ \check{z}_{-i-1} =  \check{z}_{-i} - 2\sqrt{\frac{2\tau}{\phi(\check{z}_{-i})
      \check{z}_{-i}}}$
      \While{$\check{z}_{-i-1} < 0$}
      \vspace{5pt}
      \State $\mathcal{A}^{L} \gets \mathcal{A}^{L} \cup \{\check{z}_{-i-1}\}, i \gets -i - 1$
      \State $\check{z}_{-i-1} = \check{z}_{-i} - 2\sqrt{\frac{2\tau}{\phi(\check{z}_{-i})
      \check{z}_{-i}}}$
      \EndWhile
      \State Reverse the order of $\mathcal{A}^{L}$
      \State $\mathcal{A}^{L} \gets \mathcal{A}^{L} \backslash {-1}$
      \State \textbf{Part 2:} $\zz_{-L} \leq z \leq -1$
      \State $i \gets i + 1, \check{z}_{-i+1} \gets 1$,
      
       $\mathcal{A}^{L} \gets \mathcal{A}^{L} \cup \{\check{z}_{-i+1}\}$
      \State $\check{z}_{-i} = \check{z}_{-i+1} + 2 \sqrt{\frac{2\tau}{\phi(\check{z}_{-i+1})
      \check{z}_{-i+1}}}$
      \While{$\check{z}_{-i} > \zz_{-L}$}
      \vspace{5pt}
      \State $\mathcal{A}^{L} \gets \mathcal{A}^{L} \cup \{\check{z}_{-i}\}, i \gets i + 1$
      \State $\check{z}_{-i} = \check{z}_{-i+1} + 2 \sqrt{\frac{2\tau}{\phi(\check{z}_{-i+1})
      \check{z}_{-i+1}}}$
      \EndWhile
      \State $\mathcal{A}^{L} \gets \mathcal{A}^{L} \cup \{\zz_{-L}\}$
    \end{algorithmic}
  \end{subtable}
\end{table}

We now discuss the calculation of the constant $C(\zz_{i-1}, \zz_{i})$ appearing in Lemma~\ref{lem: epsilon-reltn-z-phi-outer}. First note that $\Phi''(z) = -\phi(z) \, z$, and the maximum and minimum of $\Phi''(z)$ is attained at $\pm 1$ with the value $\pm \frac{e^{-0.5}}{\sqrt{2 \pi}}$. Thus,  $|\Phi''(z)| \leq \frac{e^{-0.5}}{\sqrt{2 \pi}}, \forall z \in \mathbb{R}$. A worst-case bound on the number of breakpoints required to achieve a $\tau$-approximation in Lemma~\ref{lem: epsilon-reltn-z-phi-outer} can be obtained based on this absolute constant. However, as we note from Figure~\ref{fig:Breaks points and Curvature change in Standard Normal CDF}, $\Phi''(z)$ varies significantly on its domain, and it exhibits monotonicity on parts of its domain. A significant reduction in the number of breakpoints is possible when we take advantage of the monotonically changing values of $\Phi''(\cdot)$. These breakpoints are computed \emph{a priori} based on the value of $\tau$ and varying $\Phi''(z)$. We now show how to achieve this.

Observe that $\Phi''(\cdot)$ is monotonically increasing in the intervals $[\zz_L, -1]$ and $[1, \zz_R]$, and it is monotonically decreasing in the interval $[-1,1]$. For $z \geq 0$, we first divide the values of $z$ in intervals $[0,1]$ (Part 1), $[1, \zz_R]$ (Part 2)  {(see lines 1-10 and 11-18 of Algorithm~\ref{alg:breakpoint_find} respectively for Part 1 and Part 2). For the interval $[0,1]$, we determine the breakpoints as a decreasing sequence starting at $z=1$. We use the value of $\Phi''$ at the current point $\zz_{i}$ for $C(\zz_{i+1}, \zz_{i})$ that determines the next point $\zz_{i+1}$. At the end of Part 1, we reverse the order of the breakpoints calculations (note that array $\mathcal{A}^{R}$ maintains the breakpoints in increasing order (line 9)). $z=1$ serves as the starting point for Part 2. Starting from $z=1$, for the interval $[1,\zz_R]$ we determine the breakpoints as an increasing sequence. Once again, we use the value of $\Phi''$ at the current point $\zz_{i-1}$ for $C(\zz_{i-1},\zz_i)$ to determine the next point $\zz_i$. Similar to a tangent-based approximation, Algorithm~\ref{alg:breakpoint_find} outlines a secant-based approximation for $z < 0$ in two parts. Here, both parts start from $z = -1$ and add breakpoints at twice the interval used for tangent-based approximation and constitute an ordered sequence of breakpoints $\mathcal{A}^{L}$. Combining these two sequences, we achieve $\mathcal{A} \gets \mathcal{A}^{L} \cup \mathcal{A}^{R}$.

Theorem~\ref{thm:complexity-breakpoint} bounds the worst-case number of breakpoints required by Algorithms~\ref{alg:breakpoint_find}-\ref{alg:breakpoint_find_negative} to achieve a $\tau$ error in the resulting PWL approximation.
\begin{theorem}\label{thm:complexity-breakpoint}
Let $\zz_{-L}, \zz_{R}$  be such that $\max\big\{1 - \Phi(\zz_{R}), \Phi(\zz_{-L})\big\} \leq \tau$, where $0 <\tau \ll \min\{ |\zz_{-L}|, \,  \zz_{R}\}$. Then $O(\sqrt{\frac{1}{\tau}\log(\frac{1}{\tau})})$ breakpoints are sufficient ($|\mathcal{A}|$ in Algorithms~\ref{alg:breakpoint_find}-\ref{alg:breakpoint_find_negative}) to ensure $0 \leq \bar{\Phi}(z; \bzz) - \Phi(z) \leq \tau$, $\forall z \in \mathbb{R}.$ 
\end{theorem}
\begin{proof}
    See Appendix\ref{appndx:proof}.
\end{proof}

\subsection{Piecewise Linear Inner Approximation}\label{sec:PWL-inner}

An inner approximation of $\Phi(z)$ can be constructed similar to the outer approximation by taking secant approximation for $z \geq 0$ and tangent approximation for $z <0$. This ensures that the inner approximation $\underline{\Phi}(\cdot; \bzz)$ is always below the original curve $\Phi(z)$. More formally,
\vspace{-5pt}
\begin{equation}\label{def:def_phi_pwl_inner}
	\underline{\Phi}(z; \bzz)
	=
	\begin{dcases*}
		\min\left\{\Phi(\zz_{R}), \min_{i \in [R]} \left\{ g_i z + g_i^0 \right\} \right\},
		& if $z \geq 0$, \quad \; \text{(Secant)}\\
		\max\left\{ 0, \max_{i \in [L]_0}\left\{ h_i z + h_i^0 \right\} \right\},
		& otherwise, \; \text{(Tangent)}
	\end{dcases*}
\end{equation}
where
\vspace{-5pt}
\begin{equation*}
\begin{alignedat}{3}
    & g_i := \frac{\Phi(\zz_i) - \Phi(\zz_{i-1})}{\zz_i - \zz_{i-1}}, \; &&  g^0_i := \Phi(\zz_{i-1}) - \frac{\Phi(\zz_i) - \Phi(\zz_{i-1})}{\zz_i - \zz_{i-1}} \zz_{i-1}, \; && i \in [R] \\ 
	& h_i \coloneqq \phi(\zz_i), \; && h_i^0 = \Phi(\zz_i) - \phi(\zz_i) \zz_i, \; && i \in [L]_0.
\end{alignedat}
\end{equation*}
Note that we have followed the convention that $g$ represents approximation of $\Phi(\cdot)$ on the non-negative reals and $h$ represents its approximation on the negative reals. Similar to $\Phi(z)$ and its outer approximation, its PWL inner approximation, $\underline{\Phi}\left(z; \bzz\right)$, is also concave over the domain $z \geq 0$.
Therefore, the constraint, $\underline{\Phi}(z; \bzz) \geq \zeta$, can be reformulated using linear constraints over this domain.
For $z \leq 0$, this function is convex, and it is reformulated using binary variables. The following proposition provides this formulation.

\begin{proposition}\label{prop:phi_pwl_inner_reformulation}
	If $\bzz$ is any valid array of breakpoints and $\ubar{z}, \bar{z} \in \mathbb R$ are given numbers satisfying $\ubar{z} \leq \zz_{-L} < \zz_{R} \leq \bar{z}$, then there exist $(z, \zeta) \in [\ubar{z}, \bar{z}] \times \mathbb R$ satisfying $\underline{\Phi} \left(z ; \bzz\right) \geq \zeta$ if and only if $\mG(\bzz) \neq \emptyset$, where
\vspace{-20pt}
    	\begin{equation}\label{eq:reform-inner-sub}
		\mG(\bzz) \coloneqq 
		\Set*{
			\begin{aligned}
				&\balpha \in \{0, 1\}^L, \\
				& \bt \in \{0, 1\}^4, \\
				& \by \in \mathbb{R}^3, \\
				& z, \zeta \in \mathbb{R} \\
                & \mathfrak{z} \in \mathbb{R}^L
			\end{aligned}
		}
		{
			\begin{aligned}
				& t_1  + t_2 + g_i y_3 + g_i^0 t_3 + \Phi(\zz_{R})t_4 \geq \zeta \quad \forall i \in [R], \\
				& h_L y_1 + h_L^0 t_1 + \sum_{i = 0}^{L-1} (h_i \mathfrak{z}_i + h_i^0 \alpha_i) + t_3 + \Phi(\zz_{R})t_4 \geq \zeta,\\
                &  h_L y_1 + h_L^0 t_1 \geq 0, \quad t_1 + t_2 + t_3 + t_4 = 1, \\
                & z = y_1 + y_2 + y_3 + y_4, \quad  -\, \ubar{z} \, t_1 \leq y_1 \leq \zz_{-L} t_1,\\
				& t_2 = \sum_{i=0}^{L-1} \alpha_{i}, \,\,\, \zz_{-i-1} \alpha_i \leq \mathfrak{z}_i \leq 0 \,\,\,\, \forall i \in [L-1]_0, \;\;\,  y_2 = \sum_{i=0}^{L-1} \mathfrak{z}_i, \\
                &  0 \leq y_3 \leq \zz_{R} t_3, \quad \bar{z} \, t_4 \geq y_4 \geq \zz_{R} t_4
			\end{aligned}
		}.
	\end{equation}
\end{proposition}
\begin{proof}
    See Appendix~\ref{appndx:proof}.
\end{proof}

  In the above proposition, the binary variables, $t_1, t_2, t_3,$ and $t_4$, are used to indicate if $z \in (\ubar{z}, \zz_{-L}]$, $z \in [\zz_{-L}, 0]$, $z \in [0, \zz_R]$, and $z \in [\zz_R, \bar{z})$, respectively.
 Applying Proposition~\ref{prop:phi_pwl_inner_reformulation} to each of the $K$ constraints, $\hat{\Phi}(z_k) \geq \zeta_k$, we obtain the following MIQP approximation of the reformulated problem~\eqref{eq:ccp_intermediate}.
 \vspace{-10pt}
\begin{equation}\label{eq:ccp_pwl-inner}
	\begin{aligned}
		\mathop{\text{min}} \; & \,\,  \bc^\top \bx \\
		\text{s.t.} \; 
		&\, \,  \sum_{k=1}^K w_k \zeta_k \geq
		\theta, \quad  \bx \in \mathcal{X}, \quad \bz \in \mathbb{R}^K, \quad  \bzeta \in [0, 1]^K, \quad \blambda \in \mathbb{R}^K_{+},  \\
		& \, \, \left.\begin{aligned}
			& (\bm{\alpha}_k,  \bt_k, \by_k, z_k, \zeta_k) \in \mG(\bzz), \quad b - \bx^\top \bmu_k \geq z_k \lambda_k, \quad \bx^\top \bSigma_k \bx = \lambda_k^2
		\end{aligned}\right\} \; k \in [K]. 
	\end{aligned}
\end{equation}
We state the counterparts of Lemma~\ref{lem: epsilon-reltn-z-phi-outer} and Theorem~\ref{thm:complexity-breakpoint} for the inner approximation in the Appendix~\ref{appndx: PWL-I} as Lemma~\ref{lem: epsilon-reltn-z-phi-inner} and Theorem~\ref{thm:complexity-breakpoint-inner}, respectively. 

\begin{remark}
If the original problem~\eqref{eq:ccp} is feasible, the PWL-O model~\eqref{eq:ccp_pwl-outer} is always feasible for any valid array of breakpoints. In contrast, the PWL-I model~\eqref{eq:ccp_pwl-inner} may not always be feasible despite the feasibility of the original problem. However, note that a feasible inner approximation always over-satisfies the chance constraint.    
\end{remark}

\subsection{Optimality Guarantees to a Given Tolerance}\label{sec:accuracy-approx-obj}

We now show that the optimal objective values of the PWL-O and PWL-I models can be arbitrarily close to the optimal objective of the original chance-constrained model by controlling the accuracy threshold $\tau$ in the breakpoint calculation algorithms.
To that end, we require the following additional assumptions and definition.

\begin{assumption}\label{assmp:non-empty-interior}
There exists $\rho > 0$ such that the feasible set,
$P(\theta')$, is nonempty for all $\theta' \in [\theta - \rho, \theta + \rho]$.
\end{assumption}

\begin{assumption}\label{assmp:strict-strict-interior}
If $b = 0$, then there exists $\hat{\delta} > 0$ such that $P(\theta) \cap \{\bx \in \mathbb{R}^n | \|\bx\| \le \hat{\delta} \}  = \emptyset$. In other words, if $b = 0$, then $\bx = \mathbf{0}$ is not a feasible solution of the chance-constrained problem~\eqref{eq:ccp}.

\end{assumption}

\begin{assumption}\label{assmp:LICQ}
The linear independence constraint qualification is satisfied at all optimal solutions of the chance-constrained problem~\eqref{eq:ccp}.
\end{assumption}

\begin{definition}[Uniform compactness] \cite[Definition~1.3]{gauvin1977differential}.\label{def: uniform-compactness}
   The set $P(\theta)$ is uniformly compact near $\theta$ if there is a neighborhood $[\theta - \rho, \theta + \rho]$ of $\theta$ for some $\rho > 0$ such that the closure of the set, $\bigcup_{\theta' \in [\theta - \rho, \theta + \rho]} P(\theta')$, is compact.
\end{definition}

\begin{remark}
    Assumption~\ref{assmp:strict-strict-interior} rules out $\mathbf{0}$ as an optimal solution when $b=0$. This is a mild assumption even when $b \neq 0$ since this can be explicitly checked separately, and it might not even be of practical interest.
\end{remark}
\vspace{-30pt}
\begin{equation}\label{eq:G}
\text{Now, let} \, \, G : \mathbb{R}^n \to \mathbb{R} \,\, \text{be defined such that}\,\,
G(\bx) = \begin{dcases*}
\mathbbm{1}_{\geq 0}(b), & if $\bx = \mathbf{0}$, \\
\Phi\left(\frac{b-\bmu^\top \bx}{\sqrt{\bx^\top \bSigma \bx}}\right), & otherwise.
\end{dcases*}
\end{equation}
We show some graphs of $G(\bx)$ and its derivative in a one-dimensional setting below.
We find that $G(\bx)$ and its derivative are well-behaved near $\bx = 0$. We will formally prove this in the general case.
\vspace{-10pt}
\begin{figure}[h]
	\centering
	\begin{subfigure}[b]{0.48\textwidth}
		\centering
		\includegraphics[width=\textwidth]{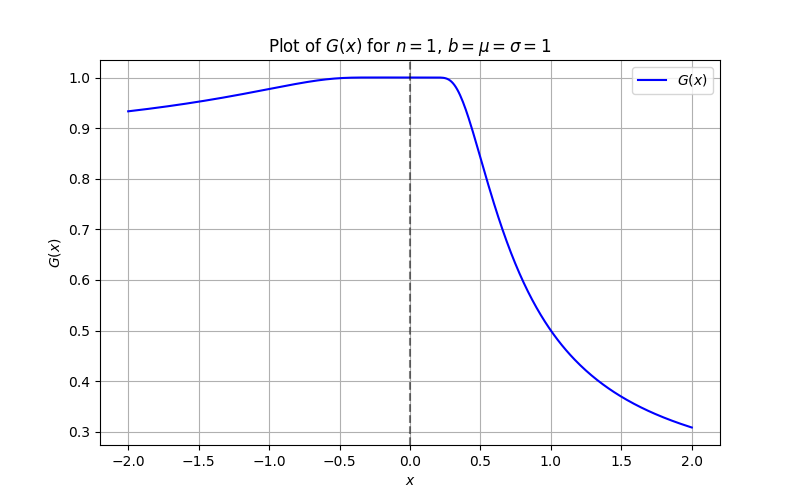
		}
	\end{subfigure}
	\hfill
	\begin{subfigure}[b]{0.48\textwidth}
		\centering
		\includegraphics[width=\textwidth]{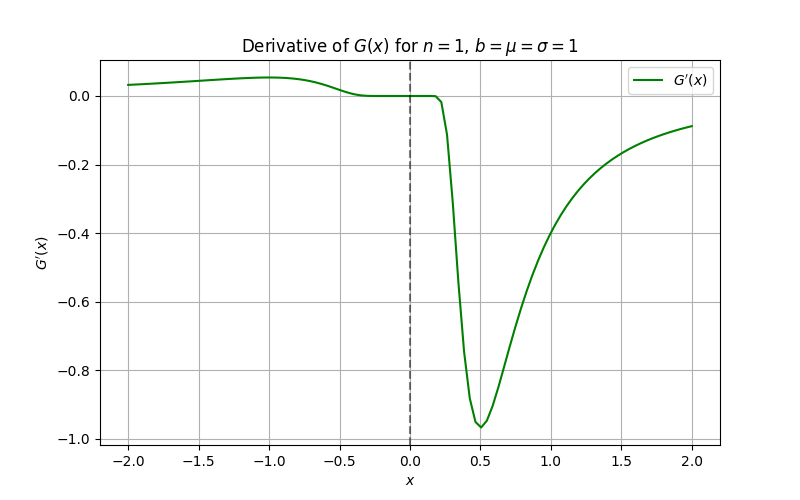}
	\end{subfigure}
	\caption{\centering The graph of $G(\bx)$ (left) and $G'(\bx)$ (right) when $\bx$ is one-dimensional.}
	\label{fig:derivative-behaviour}
\end{figure}

\begin{lemma}\label{lem:uniform-compactness}
Let $G$ be defined as in~\eqref{eq:G}. Then, the following holds:
\begin{enumerate}
    \item [(i)] If $b <0$, then $\mathbf{0} \notin P(\theta)$. 
    \item[(ii)] Under Assumptions~\ref{assmp:positive-definiteness}-\ref{assmp:strict-strict-interior}, $G(\bx)$ is continuous on $P(\theta)$.
    \item[(iii)] Under Assumptions~\ref{assmp:positive-definiteness}-\ref{assmp:strict-strict-interior}, $P(\theta)$ is uniformly compact near $\theta$. 
\end{enumerate}
\end{lemma}

\begin{proof}
 We prove part \emph{(i)} using \eqref{eq:basic-GMM-cc} which clearly shows that for any $\theta > 0$,
 $ \sum_{k=1}^K w_k \mathbb{P}\left[0 \leq b \, \big\vert \, \bxi_k \sim \mathcal N (\bmu_k, \bSigma_k) \right] \geq \theta$ does not hold with $b < 0$.

We prove part~\emph{(ii)} by establishing continuity of $G$ at an arbitrary point $\bar{\bx} \in P(\theta)$. We consider two cases: (case-1) $\bar{\bx} \neq \mathbf{0}$ and (case-2) $\bar{\bx} = \mathbf{0}$.

To show (case-1), observe that $G(\bx) = \Phi(g(\bx))$ in this case, where $g(\bx) = \frac{b-\bmu^\top \bx}{\sqrt{\bx^\top \bSigma \bx}}$. Now, let $\{\bx_m\}^\infty_{m=1}$ be any sequence in $\mathbb{R}^n \setminus \{\mathbf{0}\}$ converging to $\bar{\bx}$. Then, since Assumption~\ref{assmp:positive-definiteness} implies $\bx^\top \bSigma \bx > 0$ for any $\bx \neq \mathbf{0}$, and $\bar{\bx} \neq \mathbf{0}$, we have:
\begin{align*}
    \lim_{m \to \infty} g(\bx_m) = \lim_{m \to \infty} \frac{b-\bmu^\top \bx_m}{\sqrt{\bx^\top_m \bSigma \bx_m}} = \frac{\lim_{m \to \infty} (b-\bmu^\top \bx_m)}{\lim_{m \to \infty} \sqrt{\bx^\top_m \bSigma \bx_m}} = \frac{b-\bmu^\top \bar{\bx}}{\sqrt{\bar{\bx}^\top \bSigma \bar{\bx}}} = g(\bar{\bx}).
\end{align*}

\noindent The second equality holds by \cite[Theorem~4.4]{rudin1964principles} and the second-to-last equality follows due to the continuity of the functions, $b-\bmu^\top \bx$ and $\bx^\top \bSigma \bx$. 
Thus $g$ is continuous at $\bx = \bar{\bx} \neq \mathbf{0}$. Additionally, since $\Phi$ is continuous on $\mathbb{R}$, $\Phi\big(g(\bx)\big)$ is also continuous at all $\bx \neq \mathbf{0}$ by \cite[Theorem~4.7]{rudin1964principles}.

To prove (case-2), note that part~\emph{(i)} implies that we must have $b \geq 0$. Moreover, Assumption~\ref{assmp:strict-strict-interior} ensures that it suffices to consider only $b>0$. 
Observe that the definition of $G$ implies that it is continuous at $\mathbf{0}$, if for any $\varepsilon > 0$, there exists some $\delta > 0$ such that all $\bx \neq \mathbf{0}$ with $\|\bx\| < \delta$ satisfy $\big|1 - \Phi\big( \frac{b-\bmu^\top \bx}{\sqrt{\bx^\top \bSigma \bx}} \big) \big| < \varepsilon$, which is equivalent to $1 - \Phi\big( \frac{b-\bmu^\top \bx}{\sqrt{\bx^\top \bSigma \bx}} \big) \leq \varepsilon$ since $\Phi$ is always less than $1$.
To that end, define $\delta = \frac{b}{\Phi^{-1}(1-\varepsilon) \|\bSigma^{\frac12}\| \kappa} > 0$ with $\kappa = 1 + \frac{\|\bmu\|}{\Phi^{-1}(1-\varepsilon)\|\bSigma^{\frac12}\|}$. Then, observe that
\vspace{-10pt}
\begin{align*}
    & \|\bx\| \leq  \frac{b}{\Phi^{-1}(1-\varepsilon) \|\bSigma^{1/2}\| \kappa}
    \implies \; \sqrt{\bx^\top \bSigma \bx} \leq \|\bSigma^{1/2}\| \|\bx\| \leq  \frac{b}{\Phi^{-1}(1-\varepsilon) \kappa} \\
     \implies & \; \Phi^{-1}(1-\varepsilon) \sqrt{\bx^\top \bSigma \bx} \leq \frac{b}{ \kappa}
    \implies \; \bmu^\top \bx + \Phi^{-1}(1-\varepsilon) \sqrt{\bx^\top \bSigma \bx} \leq  \frac{b}{ \kappa} + \bmu^\top \bx  \\
     &  \leq  \frac{b}{ \kappa} + \|\bmu\|\frac{b}{\Phi^{-1}(1-\varepsilon) \|\bSigma^{1/2}\| \kappa}   \leq  \frac{b}{ \kappa} \big( 1 + \frac{\|\bmu\|}{\Phi^{-1}(1-\varepsilon) \|\bSigma^{1/2}\|}\big) = b \\
     \implies & \Phi^{-1}(1-\varepsilon) \sqrt{\bx^\top \bSigma \bx}   \leq b - \bmu^\top \bx 
     \implies 1-\varepsilon  \leq \Phi\big(\frac{b - \bmu^\top \bx}{\sqrt{\bx^\top \bSigma \bx}}\big)
     \implies  1  - \Phi\big(\frac{b - \bmu^\top \bx}{\sqrt{\bx^\top \bSigma \bx}}\big) \leq \varepsilon.
\end{align*}

To prove part ~\emph{(iii)}, we use the continuity property from part~\emph{(ii)}. By definition~\ref{def: uniform-compactness}, uniform compactness of  $P(\theta)$ near $\theta$ requires the closure of the set $\bigcup_{\theta' \in [\theta - \rho, \theta + \rho]} P(\theta')$ to be closed and bounded. Note that by Assumption~\ref{assmp:non-empty-interior}, $\bigcup_{\theta' \in [\theta - \rho, \theta + \rho]} P(\theta')$ is not empty. Additionally, part \emph{(ii)} implies that $p_k(\bx)$ and hence, $p(\bx) = \sum_{k = 1}^K w_k p_k(\bx)$, is continuous everywhere as $\Phi(g_k(\bx))$ is continuous. Then,
given any $\theta' \in [\theta - \rho, \theta + \rho]$, for any convergent sequence  $\{\bx_m\}$ with $\bx_m \in P(\theta') \; \forall m$, i.e. 
\vspace{-10pt}
$$\bH \bx_m = \bh, \bA \bx_m \geq \bd, p(\bx_m) \geq \theta', \; \forall m, \;\; \text{we have}$$
\vspace{-25pt}
\[\lim_{m \to \infty}\bH \bx_m = \bH \bar{x} = \bh, \; \lim_{m \to \infty}\bA \bx_m = \bA \bar{\bx}\geq \bd, \; \lim_{m \to \infty} p(\bx_m) = p(\bar{\bx}) \geq \theta'.
\]
Thus, $P(\theta')$ also includes the limit point $\bar{\bx}$. Therefore, $\bigcup_{\theta' \in [\theta - \rho, \theta + \rho]} P(\theta')$ is closed. Additionally, since $\mX$ is a bounded set, $\mX \cap \{ \bx \; |\; p(\bx) \geq \theta'\} \subseteq \mX$ is also bounded for every $\theta' \in [\theta - \rho, \theta + \rho]$. Therefore,  $\bigcup_{\theta' \in [\theta - \rho, \theta + \rho]} P(\theta')$ is compact.
\end{proof}

\begin{lemma}\label{prop:diff-Phi}
$G(\bx)$ is continuously differentiable on $\mathbb{R}^n$ under Assumption~\ref{assmp:strict-strict-interior}. Specifically, for all $\bx \in \mathbb{R}^n$,
\vspace{-20pt}
\begin{align}\label{eq:chain-rule}
\hspace{-0.4cm}\nabla G(\bx) = \Phi'(g(\bx)) \nabla g(\bx)
=
\frac{1}{\sqrt{2\pi}} e^{-g(\bx)^2 / 2} \bigg\{ \frac{-\bmu}{(\bx^\top \bSigma \bx)^{1/2}} - \frac{(b - \bmu^\top \bx) \bSigma \bx }{(\bx^\top \bSigma \bx)^{3/2}} \bigg\}
\end{align}

\end{lemma}

\begin{proof}
Let $g(\bx) \coloneqq \frac{b-\bmu^\top \bx}{\sqrt{\bx^\top \bSigma \bx}}$. We consider two cases: $\bx \neq \mathbf{0}$ and $\bx = \mathbf{0}$. 

Following the steps in the proof of Lemma~\ref{lem:uniform-compactness}, it can be shown that $g'(\bx)$ exists for all $\bx \neq \mathbf{0}$, and since $\Phi(\cdot)$ is differentiable everywhere in $\mathbb{R}$,
$\nabla G(\bx)$ is well-defined and continuous at all $\bx \neq \mathbf{0}$. Hence, for all $\bx \neq \mathbf{0}$, we can apply the chain rule \cite{rudin1964principles} and obtain~\eqref{eq:chain-rule}.

To prove differentiability at $\bx = \mathbf{0}$ it suffices to prove it for $b > 0$ (Assumption~\ref{assmp:strict-strict-interior} and Lemma~\ref{lem:uniform-compactness}(i)). Recall that a function $G(\bx)$ is continuously differentiable at $\mathbf{0}$ if and only if its partial derivatives exist and are continuous at $\mathbf{0}$ \cite{marsden2003vector}. The partial derivatives of $G$ exist at $\mathbf{0}$ if following limits exist:
\vspace{-10pt}
\begin{align}\label{def:differentiability}
\frac{\partial G}{\partial \bx_j} \big|_{_{\bx = 0}} 
&= 
\underbrace{
\lim_{h \to 0^+}\frac{ \Phi(g(\mathbf{0} + h \bs{e}_j)) - G(\mathbf{0})}{h}
}_{\text{A}} 
= 
\underbrace{
\lim_{h \to 0^+}\frac{ \Phi(g(\mathbf{0} - h \bs{e}_j)) - G(\mathbf{0})}{h}
}_{\text{B}} 
\quad \forall j
\end{align}

Let $\|\bSigma_j\|$ be the norm of the $j^{th}$ column vector of $\bSigma$. Note that $g(\mathbf{0} + h \bs{e}_j) =  (b - \mu_j h)/h \|\bSigma_j\| = b/h\|\bSigma_j\| - \mu_j/ \|\bSigma_j\| = q_j/h - r_j$, where $q_j = b/\|\bSigma_j\| > 0$ since $b > 0$, and $r_j := \mu_j/ \|\bSigma_j\|$. Since $G(\mathbf{0}) = 1$, using the definition of $\Phi$ in term A of~\eqref{def:differentiability}, we have:
\vspace{-5pt}
\begin{align*}
A = \lim_{h \to 0^+}\frac{ \int_{-\infty}^{\frac{q_j}{h} - r_j} \frac{1}{\sqrt{2 \pi}} e^{-\frac{z^2}{2}} dz - 1}{h} = \lim_{h \to 0^+}\frac{-\frac{1}{\sqrt{2 \pi}} \int^{\infty}_{\frac{q_j}{h} - r_j}  e^{-\frac{z^2}{2}} dz}{h} \;\; \forall j.
\end{align*}
According to \cite[Proposition 3]{szarek1999nonsymmetric}, for $t > -1$, 
$\frac{2}{t+\sqrt{t^2 + 4}} \leq e^{t^2/2} \int^{\infty}_{t} e^{-z^2/2} dz \leq \frac{4}{3t+\sqrt{t^2 + 8}}$. Let $ t_j = q_j/h - r$. Then as $h \to 0^+$, $t_j \to \infty$ since $q_j >0$. Thus, the condition $t_j > -1$ is satisfied  for all $j$. Therefore,
\vspace{-20pt}
\begin{align*}
  \frac{2 e^{-(\frac{q_j}{h}-r)^2/2}}{h\big((\frac{q_j}{h}-r)+\sqrt{(\frac{q_j}{h}-r)^2 + 4}\big)} \leq \frac{\int^{\infty}_{\frac{q_j}{h}-r} e^{-z^2/2} dz}{h} \leq \frac{4 e^{-(\frac{q_j}{h}-r)^2/2}}{h \big(3 (\frac{q_j}{h}-r)+\sqrt{(\frac{q_j}{h}-r)^2 + 8}\big)}.
\end{align*}
Note that the lower bound $\frac{2 e^{{(\frac{q_j}{h}-r)^2}/{2}}}{h\left(\left(\frac{q_j}{h}-r\right)+\sqrt{\left(\frac{q_j}{h}-r\right)^2 + 2}\right)} =  \frac{2 e^{-\left(\frac{q_j}{h}-r\right)^2/2}}{(q_j - r h)+\sqrt{(q_j - r h)^2 + 2 h^2}} $ goes to 0 as $h \to 0^+$. Similarly, the upper bound also goes to 0 as $h \to 0$.

Now for the term $B$ in \eqref{def:differentiability}, $g(\mathbf{0} - h \bs{e}_j) =  (b + \mu_j h)/h \|\bSigma_j\| = b/h\|\bSigma_j\| + \mu_j/ \|\bSigma_j\| = q_j/h + r_j = t_j$. Thus, as $h \to 0^+$, $t_j \to \infty$ (since $q_j >0$). Hence, the condition $t_j > -1$ is satisfied. Therefore,
\vspace{-5pt}
\begin{align*}
  \frac{2 e^{-(\frac{q_j}{h}+r)^2/2}}{h\big((\frac{q_j}{h}+r)+\sqrt{(\frac{q_j}{h}+r)^2 + 4}\big)} \leq \frac{\int^{\infty}_{\frac{q_j}{h}+r} e^{-z^2/2} dz}{h} \leq \frac{4 e^{-(\frac{q_j}{h}+r)^2/2}}{h \big(3 (\frac{q_j}{h}+r)+\sqrt{(\frac{q_j}{h}+r)^2 + 8}\big)}.
\end{align*}
Similar to the proof for the part A in \eqref{def:differentiability}, we can show that both lower and upper bounds in the above inequality go to $0$ as $h \to 0^+$. Thus, $\nabla G(\bx) = \mathbf{0}$ at $\bx = \mathbf{0}$ for $ b > 0.$ It shows that $G(\bx)$ is differentiable at all $\bx$. 

We now show that  \eqref{eq:chain-rule} provides an expression for $\nabla G(\bx)$  at all $\bx$, including $\bx = \mathbf{0}$, and that $\nabla G(\bx) $ is continuous. From \eqref{eq:chain-rule},
\vspace{-5pt}
\begin{align}\label{eq:chain-rule-j}
\hspace{-0.4cm} \frac{\partial G(\bx)}{ \partial \bx_j} =
\frac{1}{\sqrt{2\pi}} e^{-g(\bx)^2 / 2} \bigg\{ \frac{-\bmu_j}{(\bx^\top \bSigma \bx)^{1/2}} - \frac{(b - \bmu^\top \bx) (\bSigma \bx)_j }{(\bx^\top \bSigma \bx)^{3/2}} \bigg\}.
\end{align}
 To prove continuity of $\nabla G(\bx)$ at $\bx = \mathbf{0}$, when $b > 0$, we show that 
\vspace{-5pt}
$$\ \lim_{\bx \to 0} \frac{\partial G(\bx)}{ \partial \bx_j} = \frac{\partial G(\bx)}{ \partial \bx_j}|_{\bx = \mathbf{0}}= 0, \; \forall j$$

\noindent  by showing that $\limsup_{\bx \to \mathbf{0}} \frac{\partial G(\bx)}{ \partial \bx_j} = \liminf_{\bx \to \mathbf{0}}\frac{\partial G(\bx)}{ \partial \bx_j} = 0$. Note that since $- \|\bSigma^{1/2}\| \|\bSigma^{1/2} \bx\| \leq (\bSigma \bx)_j = (\bSigma^{1/2} \bSigma^{1/2} \bx)_j \leq \|\bSigma^{1/2}\| \|\bSigma^{1/2} \bx\|$,
\vspace{-5pt}
\begin{align}
& \frac{1}{\sqrt{2 \pi}} \exp\{-(\frac{b-\bmu^\top \bx}{\sqrt{\bx^\top \bSigma \bx}})^2 \} \bigg\{ - \frac{|b - \bmu^\top \bx| \| \bSigma^{1/2} \| \| \bSigma^{1/2} \bx \| }{\|\bSigma^{1/2} \bx \|^3} - \frac{\|\mu\|}{\|\bSigma^{1/2} \bx \|} \bigg \} \leq \frac{\partial G(\bx)}{ \partial \bx_j} \nonumber\\
&  \leq \frac{1}{\sqrt{2 \pi}} \exp\{-(\frac{b-\bmu^\top \bx}{\sqrt{\bx^\top \bSigma \bx}})^2 \} \bigg\{ \frac{|b - \bmu^\top \bx| \| \bSigma^{1/2} \| \| \bSigma^{1/2} \bx \| }{\|\bSigma^{1/2} \bx \|^3} + \frac{\|\mu\|}{\|\bSigma^{1/2} \bx \|} \bigg \} \;\; \forall j. \label{eq:gradient-bound}
\end{align}

Let the left and right hand side of \eqref{eq:gradient-bound} be $L(\bx)$ and $U(\bx)$, respectively. It suffices to show that $\limsup_{\bx \to \mathbf{0}} L(\bx) = \liminf_{\bx \to \mathbf{0}} U(\bx) = 0.$

\noindent We first show that $\liminf_{\bx \to \mathbf{0}} U(\bx) = 0$.  Since $e^{y} = \sum_{\ell=0}^{\infty} \frac{(y)^\ell}{\ell!} = 1 + y + \frac{y^2}{2!} + \frac{y^3}{3!} + \cdots$, for any $j \in [n], \; \ell \in \mathbb{N}$ and for $\ell \geq 2$, 
\vspace{-10pt}
\begin{align}
& U(\bx) =  \frac{1}{\sqrt{2 \pi}} \bigg\{   \bigg \| \frac{\frac{|b - \bmu^\top \bx| \| \bSigma^{1/2} \|}{\|\bSigma^{1/2} \bx \|^2} }{1 + (\frac{b-\bmu^\top \bx}{\|\bSigma^{1/2} \bx \|})^2 + \frac{1}{2!} (\frac{b-\bmu^\top \bx}{\|\bSigma^{1/2} \bx \|})^4 + \frac{1}{3!} (\frac{b-\bmu^\top \bx}{\|\bSigma^{1/2} \bx \|})^6 + \cdots} \bigg \| \nonumber\\
& +  \bigg \| \frac{ \frac{\|\mu\|}{\|\bSigma^{1/2} \bx \|} }{1 + (\frac{b-\bmu^\top \bx}{\|\bSigma^{1/2} \bx \|})^2 + \frac{1}{2!} (\frac{b-\bmu^\top \bx}{\|\bSigma^{1/2} \bx \|})^4 + \frac{1}{3!} (\frac{b-\bmu^\top \bx}{\|\bSigma^{1/2} \bx \|})^6 + \cdots} \bigg \| \bigg\} \nonumber\\
& = \frac{1}{\sqrt{2 \pi}}  \bigg\{ \bigg \| \frac{|b - \bmu^\top \bx| \| \bSigma^{1/2} \| }{\|\bSigma^{1/2} \bx \|^2 + (b-\bmu^\top \bx)^2 + \frac{1}{2!} \frac{(b-\bmu^\top \bx)^4}{\|\bSigma^{1/2} \bx \|^2} + \frac{1}{3!} \frac{(b-\bmu^\top \bx)^6}{\|\bSigma^{1/2} \bx \|^4} + \cdots} \bigg \| \nonumber\\
& +  \bigg \| \frac{ \|\mu\| }{\|\bSigma^{1/2} \bx \| + \frac{(b-\bmu^\top \bx)^2}{\|\bSigma^{1/2} \bx \|} + \frac{1}{2!} \frac{(b-\bmu^\top \bx)^4}{\|\bSigma^{1/2} \bx \|^3} + \frac{1}{3!} \frac{(b-\bmu^\top \bx)^6}{\|\bSigma^{1/2} \bx \|^5} + \cdots} \bigg \| \bigg\} \nonumber \\
& \leq \frac{1}{\sqrt{2 \pi}}  \bigg\{ \bigg \| \frac{|b - \bmu^\top \bx| \| \bSigma^{1/2} \| }{\|\bSigma^{1/2} \bx \|^2 + (b-\bmu^\top \bx)^2 + \frac{1}{2!} \frac{(b-\bmu^\top \bx)^4}{\|\bSigma^{1/2} \bx \|^2} + \cdots +\frac{1}{\ell!} \frac{(b-\bmu^\top \bx)^{2\ell}}{\|\bSigma^{1/2} \bx \|^{2\ell - 2}}} \bigg \| \nonumber\\
& +  \bigg \| \frac{ \|\mu\| }{\|\bSigma^{1/2} \bx \| + \frac{(b-\bmu^\top \bx)^2}{\|\bSigma^{1/2} \bx \|} + \frac{1}{2!} \frac{(b-\bmu^\top \bx)^4}{\|\bSigma^{1/2} \bx \|^3} + \cdots + \frac{1}{\ell!} \frac{(b-\bmu^\top \bx)^{2\ell}}{\|\bSigma^{1/2} \bx \|^{2\ell -1}} } \bigg \| \bigg\} \nonumber\\
& = \frac{1}{\sqrt{2 \pi}}  \bigg\{ \bigg \| \frac{\ell!\|\bSigma^{1/2} \bx \|^{2\ell - 2} |b - \bmu^\top \bx| \| \bSigma^{1/2} \| }{\ell!\|\bSigma^{1/2} \bx \|^{2\ell} + (b-\bmu^\top \bx)^2 \ell!\|\bSigma^{1/2} \bx \|^{2\ell - 2} + \cdots + (b-\bmu^\top \bx)^{2\ell}} \bigg \| \nonumber\\
& +  \bigg \| \frac{\ell!\|\bSigma^{1/2} \bx \|^{2\ell - 1}  \|\mu\| }{\ell!\|\bSigma^{1/2} \bx \|^{2\ell} + (b-\bmu^\top \bx)^2 \ell!\|\bSigma^{1/2} \bx \|^{2\ell - 2} + \cdots + (b-\bmu^\top \bx)^{2\ell} } \bigg \| \bigg\}, \nonumber
\end{align}
where the first inequality in the above follows since summation is truncated and all terms are positive. Now, since $b > 0$, taking limit on both sides of the above expression,  we obtain
$ \liminf_{\bx \to \mathbf{0}} U(\bx) = 0$. A similar argument shows that $\limsup_{\bx \to \mathbf{0}} L(\bx) = 0$. Thus, $\nabla G(\bx)$ in~\eqref{eq:chain-rule} is continuous at $\bx = \mathbf{0}$ and $G(\bx)$ is continuously differentiable on $\mathbb{R}^n$.
\end{proof}

We are now in a position to present the main result of this section in Theorem~\ref{thm:convergence-outer} below. In particular, we show that using $O( \sqrt{\frac{1} {\tau}\log(\frac{1}{\tau})})$ breakpoints can provide a solution that achieves a user-specified tolerance $\hat{\tau}$ in the optimal objective value. We first establish the continuity of the optimal objective function of~\eqref{eq:ccp} by using a result from Gauvin and Tolle  \cite{gauvin1977differential}.

\begin{lemma}\label{lem:finite-convergence}
Let Assumptions~\ref{assmp:positive-definiteness}-\ref{assmp:LICQ} hold. Then, the optimal value function $Z^\star(\theta)$ is continuous at $\theta$.
\end{lemma}

\begin{proof}
By Lemma~\ref{lem:uniform-compactness}, $P(\theta)$ is uniformly compact near $\theta$. Additionally, Lemma~\ref{prop:diff-Phi} ensures that $p(\bx)$ is continuously differentiable on $\mathbb{R}^n$. The KKT condition then implies that for dual variables $\bu(\theta), \rho(\theta) \geq 0$, and $\bv(\theta)$,
\vspace{-10pt}
\begin{align*}
\bc - \bH^\top \bv(\theta) - \bA^\top \bu(\theta) - \rho(\theta) \nabla p(\bx^*(\theta)) = \bm{0}, \;\; \bu_i(\theta) (\bA_i \bx^*(\theta) - \bd_i) = 0 .
\end{align*} 
 Let $\bB \bx^*(\theta) = \bd_{\bB}$ be the set of binding linear inequality constraints at an optimal solution, $\bx^*(\theta)$.  Since by Assumption~\ref{assmp:LICQ}, linear independence constraint qualification (LICQ) holds at $\bx^*(\theta)$,  there do not exist non-zero dual variables $ \bu(\theta), \rho(\theta) \geq 0$ and $\bv(\theta)$ such that 
\vspace{-10pt}
\begin{align*}
    \bH^\top \bv(\theta) + \bB^\top \bu(\theta) + \rho(\theta) \nabla p(\bx^*(\theta)) = \bm{0}. 
\end{align*}
Gordan's theorem \cite[Theorem~2.4.9]{bazaraa2006nonlinear} then implies the existence of $\mathfrak{z}$ such that
\vspace{-10pt}
\begin{align*}
\bB \mathfrak{z} > \bm{0}, \qquad \nabla p(\bx^*(\theta))^\top \mathfrak{z} > 0.
\end{align*}
Thus, the Mangasarian-Fromovitz constraint qualification conditions hold at $\bx^*(\theta)$. Then,  \cite[Theorem~2.6]{gauvin1977differential}, implies that $Z^\star(\theta)$ is continuous at $\theta$. 
\end{proof}

The following result is a direct consequence of the above lemma. 

\begin{definition}
    A solution $\hat{\bx}$ is called $\hat{\tau}-$feasible to problem~\eqref{eq:ccp} if $\bH \hat{\bx} = \bh, \bA \hat{\bx} \geq \bd$ and $p(\hat{\bx}) \geq \theta - \hat{\tau}$. A solution $\hat{\bx}$ is called $\hat{\tau}-$optimal if it is $\hat{\tau}-$feasible and  $|\bc^\top \bx^*(\theta) - \bc^\top \hat{\bx} | \leq \hat{\tau}$.
\end{definition}

\begin{theorem}\label{thm:convergence-outer}
Let $\bx^*(\theta)$ be an optimal solution of~\eqref{eq:ccp}. Suppose Assumptions~\ref{assmp:positive-definiteness}-\ref{assmp:LICQ} hold, and we are given a user-specified tolerance $\hat{\tau}$. Then, there exists $\tau$ satisfying $\hat{\tau} \geq \tau > 0$, such that the outer and inner approximation problems~\eqref{eq:ccp_pwl-outer} and \eqref{eq:ccp_pwl-inner} 
generated using $O(\sqrt{\frac{1}{\tau}\log(\frac{1}{\tau})})$ breakpoints have optimal solutions $\bx^\mathrm{OA}$ and $\bx^\mathrm{IA}$ that are 
$\hat{\tau}$-optimal. Specifically,
\vspace{-5pt}
\begin{align*}
    |\bc^\top \bx^\mathrm{IA} - \bc^\top \bx^\mathrm{OA}| \leq |\bc^\top \bx^*(\theta+\tau) - \bc^\top \bx^*(\theta-\tau) | \leq \hat{\tau}.
\end{align*}
\end{theorem}

\begin{proof}
To show $\hat{\tau}-$ feasibility, let us denote the feasible set corresponding to the PWL-O formulation~\eqref{eq:ccp_pwl-outer} with $\tau-$accuracy as $\bar{P}(\theta, \tau)$. Then from Definition~\eqref{def: P_theta} and Theorem~\ref{thm:complexity-breakpoint} we have $P(\theta) \subseteq \bar{P}(\theta, \tau) \subseteq P(\theta - \tau)$.
Thus $\hat{\tau}-$feasibility follows since $\hat{\tau} \geq \tau$. 

To prove $\hat{\tau}-$optimality, using the continuity of $Z^\star(\theta)$ at $\theta$ from Lemma~\ref{lem:finite-convergence}, for every $\varepsilon > 0$, $\exists \; \tau^+(\varepsilon) >0$ such that $|\theta - \theta'| \leq \tau^+(\varepsilon)$ implies $|Z^\star(\theta) - Z^\star(\theta')| \leq \varepsilon$.  Let $\theta' \in [\theta-\tau, \theta + \tau]$ for some $\tau > 0$. Then, given the user-specified $\hat{\tau}$ we have a $\tau^+(\hat{\tau})$ for which $|Z^\star(\theta) - Z^\star(\theta')| \leq \hat{\tau}/2$ is satisfied for all sufficiently small $\hat{\tau} > 0$, with $\theta'$ satisfying $|\theta - \theta'| \leq \tau^+(\hat{\tau}/2) = \tau$. Now $|Z^\star(\theta + \tau) - Z^\star(\theta - \tau)| \leq |Z^\star(\theta + \tau) - Z^\star(\theta)| + |Z^\star(\theta) - Z^\star(\theta - \tau)| \leq \hat{\tau}$. The claim follows because $|\bc^\top \bx^\mathrm{IA} - \bc^\top \bx^\mathrm{OA}| \leq |Z^\star(\theta + \tau) - Z^\star(\theta - \tau)|$.
\end{proof}
An immediate consequence of Theorem~\ref{thm:convergence-outer} is that for any $\hat{\tau}$, there exists a $\tau > 0$ satisfying
$|\bc^\top \bx^*(\theta) - \bc^\top \bx^\mathrm{OA}| \leq \hat{\tau}$ and $|\bc^\top \bx^\mathrm{IA} - \bc^\top \bx^*(\theta)| \leq \hat{\tau}$. Although this is only an existence result, we note that we in practice can follow an iterative procedure to achieve the desired user-specified tolerance in the optimal objective value. In this procedure, we start with an initial value of $\tau$, and solve the outer and inner approximation problems. If the optimality gap determined from their solutions do not satisfy the desired tolerance $\hat{\tau}$, we reduce $\tau$ by a suitable factor (e.g., by a factor of 2) and update the previously generated break points. We can stop this procedure once the desired tolerance is satisfied.

\section{Numerical Experience with a Commercial Solver}\label{sec:numerical}
This section studies the computational performance of the PWL-I and PWL-O approximations, and the sample average approximation (SAA). Our discussion is based on extensive computational experiments conducted across a range of data generation parameters for the problems of different dimensions. We use Gurobi \cite{gurobi} to solve the instances. We describe problem generation in Section~\ref{sec: data-generation} and computational setup in Section~\ref{sec: exp-setup}. The break point selection in the PWL approximations is described in Section~\ref{sec:algorithm1_description}.

\subsection{Data Generation}\label{sec: data-generation}
We randomly generated problems of 100, 500, and 1000 dimensions, i.e., \(n \in \{100, 500, 1000\}\).  We set \(\theta \in \{0.95, 0.99, 0.999\}\) for chance constraint satisfaction probability. The value of $\theta = 0.999$ indicates a solution with a very high satisfaction probability. The number of Gaussian mixture components used was set at \(K \in \{5, 10, 15\}\). For the results reported in the main body of this section, the mixture weights were \(\bw = (0.05, 0.1, 0.2, 0.3, 0.35)\) for \(K=5\), \(\bw = (0.001,0.009,0.02,0.05, 0.08,0.09, 0.1,  0.15,0.2,0.3)\) for \(K=10\) and \(\bw = (0.001, 0.005,0.009, 0.01, 0.01,0.015, 0.02, 0.05,  0.08, 0.09,0.1,0.12, 0.13,0.17,0.19)\) for \(K=15\).  These weights are chosen to have both small and large values while ensuring that they sum to $1.0$.  We also conducted experiments using equal mixture weights. These results are reported in Appendix~\ref{appndx: Ommtd-Table}.

Mixture component mean vectors and covariance matrices were generated using component-mean and variance control hyper-parameters, denoted by $\varrho$ and $\zeta$, respectively. We used six different values of \(\varrho \in \{2, 5, 10, 15, 20, 25\}\). For each value of $\varrho > 0$, we randomly generated interval bounds 
$[\underline{\boldsymbol{\mu}}_k, \overline{\boldsymbol{\mu}}_k] 
\subset [0, \varrho \sqrt{n} \log n]^n$ 
for every $k \in [K]$. 
Given these bounds, each component $\mu_{k i}$ of $\boldsymbol{\mu}_k = (\mu_{k1}, \ldots, \mu_{kn})^\top$ 
is sampled as 
$\mu_{k i} \sim \mathrm{Unif}([\underline{\mu}_{k i},\, \overline{\mu}_{k i}])$ 
for all $i \in [n]$. Similarly, for covariance matrices, five different values of variance-control hyperparameter \(\varsigma \in \{2, 5, 10, 15, 20\}\) were considered. Eigenvalues \(\bs{\nu}_k, k \in [K]\) are uniformly sampled from \((0, \varsigma \ell_k]\), where \(\ell_k \in [0, 1]\) was chosen through quasi-Monte-Carlo (QMC) sampling. QMC sampling ensures that $\ell_k$ is not clustered. Given a $\ell_k$, we finally constructed \(\bs{\Sigma}_k \in \mathbb{S}^N_{++}\) for all $k \in [K]$ from an \(n\)-dimensional diagonal matrix of eigenvalues $\bD_k$, where each entry was uniformly sampled from $[0, \varsigma \ell_k]$. The eigenvector matrices $\bQ_k$ are randomly generated orthogonal matrices. They were generated by taking $2 n$ products of randomly generated $2 \times 2$ Givens rotation matrices \cite{givens1958computation}. The generated covariance matrix $\bSigma_k = \bQ^\top_k \bD_k \bQ_k$ is positive definite.

We generated samples for the SAA approach as follows. We used the scikit-learn (\texttt{sklearn}) package in Python~\cite{scikit-learn}. In particular, we used the \texttt{GaussianMixture} class from \texttt{sklearn.mixture}, initializing it with our Gaussian mixture parameters and then calling its \texttt{sample} method. Generated samples are then used in the general SAA model from \cite{luedtke2008sample}: 
\vspace{-0.5em}
\begin{align}\label{prob:SAA-GMM}
\min_{\bx \in \mathcal{X}, \by \in \{0, 1\}^{S}} \quad \bc^\top \bx \quad 
\text{s.t.}  \quad \bxi^s \bx - b \leq  M y^s  \quad \forall s \in [S] \qquad \sum_{s = 1}^{S} y^s \leq (1 - \theta) \; S, \,\, \tag{GMM-SAA}
\end{align}
where the binary vector $\by$ is of the same dimension as the sample size $S$. Big-$M$ value was set to $1e^6$.  Sample size \(S\) depends on  \(\theta\): for \(\theta = 0.95\) and \(\theta = 0.99\), we set \(S = 100/(1-\theta)\). For \(\theta = 0.999\), we set \(S = 20/(1-\theta)\), since generating problems with a large number of scenarios was not practical for $\theta = 0.999$. 

To generate the right-hand side \(b\) in constraint~\eqref{eq:cc}, we sampled 1000 decision vectors \(\bx\) from its domain. We then computed \(b\) as the average of \(\bx^\top \bmu_k + \sqrt{\bx^\top \bSigma_k \bx}\) over all \(k\) and sampled decisions \(\bx\). One standard deviation of \(\bx^\top \bmu_k\) (i.e., \(\bx^\top \bmu_k + \sqrt{\bx^\top \bSigma_k \bx}\)) was used to prevent an excessive proportion of infeasible instances. We did not use linear equality constraints and generated inequality constraints as follows:  $\bA$ and $\bd$ are randomly generated as \(\bA \stackrel{\mathcal{U}}{\sim} [0, 1]\) and \(\bd \stackrel{\mathcal{U}}{\sim} [n/4, 3N/4]\). For \(n = \{100, 500\}\) we specify the dimension of vector \(\bd\) as  \(n/10\); for \(n=1000\) we reduce this dimension to \(n/20\) to avoid too many infeasible instances. 

\subsection{Experimental Setup}\label{sec: exp-setup}
All computations were performed on a server with Intel Xeon 2.80 GHz CPUs. We used Python 3.9 and Gurobi 11.0.3. We set $(1-\theta)/10$ as the \texttt{optimality} \texttt{(MIP)} gap in Gurobi. We allow Gurobi to use up to 2 cores per instance with other parameters set at their default values. A 18-hour (64,800s)-time limit is set. Let the solution from PWL-I, PWI-O, and SAA models within the time-limit be $\check{\bx}^{I}, \check{\bx}^{O},$ and $\check{\bx}^S$. We substitute these solutions into~\eqref{eq:cc_probability_standardized_form}  to compute chance constraint probability satisfied  by them. These probabilities are denoted by  $\check{\theta}^I, \check{\theta}^O$ and $\check{\theta}^S$, respectively.


\subsection{Break point selection and number of pieces in Piecewise Linear Approximation}\label{sec:algorithm1_description}
We set $\tau = (1-\theta)/10$ in specifying piecewise linear approximation accuracy. For this value of $\tau$, we use Algorithm~\ref{alg:breakpoint_find}-\ref{alg:breakpoint_find_negative} to adaptively construct the breakpoint-array $\bzz$. We fix breakpoint-array endpoints for all $\tau$ by setting \(\zz_L = -6.466\) and \(\zz_R = 6.466\). It meets the requirement of Theorems~\ref{thm:complexity-breakpoint} and~\ref{thm:complexity-breakpoint-inner}$: 
|\bar{\Phi}(\zz_L) - \Phi(\zz_L)| = |\bar{\Phi}(\zz_R) - \Phi(\zz_R)| < 1 \times e^{-10}$ (for PWL-O), and
$|\underline{\Phi}(\zz_L) - \Phi(\zz_L)| = |\underline{\Phi}(\zz_R) - \Phi(\zz_R)| < 1 \times e^{-10}$ (for PWL-I).

\subsection{Computational Results and Discussion} \label{sec:ComputationalResults}

We present detailed numerical results for $\varrho = 2, \varsigma = 2$ and $\varrho = 5, \varsigma = 10$ with unequal mixture weight setting in Tables~\ref{tab:num-performance-2-2-diff}-\ref{tab:num-performance-2-5-diff}, respectively. 
Results for all the remaining parameter combinations of $(\varrho, \varsigma)$ with unequal weights are summarized in Table~\ref{tab:num-performance-sum-diff}. Similar results with equal mixture weights are presented in Appendix~\ref{appndx: Ommtd-Table}.


\scriptsize
\vspace{-10pt}
\begin{longtable}[!htbp]{c | r| r c c| c c c c| c c c}
\caption{\centering Comparison of PWL-I, PWL-O, and SAA approaches for $\varrho = 2, \varsigma=2$ with different mixture weights and 64,800s (18-hour) time limit} \label{tab:num-performance-2-2-diff} \\ [-10pt]
\hline
\multicolumn{2}{c|}{} & \multicolumn{3}{c|}{PWL-I} & \multicolumn{4}{c|}{PWL-O} & \multicolumn{3}{c}{SAA} \\ \hline
\multicolumn{1}{c|}{$K$} 
& \multicolumn{1}{c|}{$\theta$} 
& \multicolumn{1}{c|}{\makecell{Time \\ (sec)}} 
& \multicolumn{1}{c|}{Obj\specialfootnote{*}{Objective values are reported after truncating decimal digits}} 
& \multicolumn{1}{c|}{${\check{\theta}}^I$} 
& \multicolumn{1}{c|}{\makecell{Time \\ (sec)}} 
& \multicolumn{1}{c|}{\%- Obj}
& \multicolumn{1}{c|}{${\check{\theta}}^O$} 
& \multicolumn{1}{c|}{\%${\check{\theta}}$}
& \multicolumn{1}{c|}{\makecell{Time \\ (sec)}}  
& \multicolumn{1}{c|}{Obj\starfootnote} 
& \multicolumn{1}{c}{${\check{\theta}}^S$} \\ 
\hline

\multicolumn{12}{c}{$n = 100$} \\ \hline
\multirow{3}{*}{5} & 0.95  & 0.7 & -735 & 0.9500 & 0.4 & 0.0000 & 0.9499 & 0.011 & 0.5609 & -844 & 0.9565 \\ 
& 0.99  & 1.2 & -735 & 0.9900 & 0.9 & 0.0002 & 0.9899 & 0.010 & 215.58 & -844 & 0.9452 \\ 
& 0.999 & 5.0 & -734 & 0.9990 & 5.3 & 0.0005 & 0.9990 & 0.000 & 338.77 & -844 & 0.9555 \\ \hline
\multirow{3}{*}{10} & 0.95  & 2.2 & -727 & 0.9500 & 2.0 & 0.0002 & 0.9498 & 0.021 & 64,800 & -759 & 0.8636 \\ 
& 0.99  & 3.4 & -726 & 0.9900 & 2.7 & 0.0001 & 0.9900 & 0.000 & 64,800 & -743 & 0.9186 \\ 
& 0.999 & 90 & -725 & 0.9990 & 57.1 & 0.0001 & 0.9990 & 0.000 & 23,893 & -730 & 0.9799 \\ \hline
\multirow{3}{*}{15} & 0.95  & 3.5 & -704 & 0.9504 & 2.9 & 0.3555 & 0.9499 & 0.06 & 23,089 & -730 & 0.8790 \\ 
& 0.99  & 2.6 & -706 & 0.9900 & 3.1 & 0.0006 & 0.9900 & 0.000 & 64,800 & -707 & 0.9573 \\ 
& 0.999 & 54 & -705 & 0.9990 & 55.2 & 0.0000 & 0.9990 & 0.000 & 64,805 & -706 & 0.9799 \\ \hline
\multicolumn{12}{c}{$n=500$} \\ \hline
\multirow{3}{*}{5} & 0.95  & 111.9 & -2,454 & 0.9758 & 109.7 & 0.0000 & 0.9995\specialfootnote{†}{For these instances, chance constraint appears inactive at target satisfaction probability}& - & 5,238.8 & -2,474 & 0.5791 \\ 
& 0.99  & 157.6 & -2,454 & 0.9901 & 199.2 & 0.0002 & 0.9900 & 0.010 & 64,800 & -2,473 & 0.5307 \\ 
& 0.999 & 201.4 & -2,453 & 0.9990 & 227.7 & 0.0001 & 0.9990 & 0.000 & 64,801 & -2,473 & 0.6910 \\ \hline
\multirow{3}{*}{10} & 0.95  & 391.1 & -1,812 & 0.9500 & 329.0 & 0.0003 & 0.9498 & 0.021 & 64,800 & -1,857 & 0.6029 \\ 
& 0.99  & 252.0 & -1,810 & 0.9900 & 251.6 & 0.0002 & 0.9900 & 0.000 & 64,801 & -1,812 & 0.7697 \\ 
& 0.999 & 498.7 & -1,809 & 0.9990 & 2,225.8 & 0.0001 & 0.9990 & 0.000 & 64,801 & -1,812 & 0.8628 \\ \hline
\multirow{3}{*}{15} & 0.95  & 578.4 & -2,338 & 0.9501 & 658.4 & 0.0000 & 0.9496 & 0.053 & 64,800 & -2,665 & 0.6546 \\ 
& 0.99  & 748.4 & -2,338 & 0.9900 & 686.6 & 0.0004 & 0.9899 & 0.010 & 64,800 & -2,411 & 0.8926 \\ 
& 0.999 & 1,556  & -2,337 & 0.9990 & 6,116  & 0.0000 & 0.9990 & 0.000 & 64,801 & -2,338 & 0.9018 \\ \hline
\multicolumn{12}{c}{$n=1000$} \\ \hline
\multirow{3}{*}{5} & 0.95  & 839.8 & -7,794 & 0.9510 & 823.2 & 0.0001 & 0.9499 & 0.116 & 64,800 & -7,794 & 0.3706 \\ 
& 0.99  & 958.9 & -7,794 & 0.9901 & 760.5 & 0.0000 & 0.9900 & 0.010 & 64,800 & -7,794 & 0.7984 \\ 
& 0.999 & 798.4 & -7,793 & 0.9990 & 796.1 & 0.0000 & 0.9990 & 0.000 & 64,809 & -7,794 & 0.8662 \\ \hline
\multirow{3}{*}{10} & 0.95  & 3,343 & -5,810 & 0.9500 & 2,947 & 0.0000 & 0.9496 & 0.053 & 64,800 & -5,824 & 0.4744 \\ 
& 0.99  & 2,851 & -5,810 & 0.9900 & 2,363 & 0.0001 & 0.9899 & 0.010 & 64,800 & -5,819 & 0.5889 \\ 
& 0.999 & 8,008 & -5,809 & 0.9990 & 6,312 & 0.0000 & 0.9990 & 0.000 & 64,805 & -5,810 & 0.8211 \\ \hline
\multirow{3}{*}{15} & 0.95  & 3,069 & -4,349 & 0.9506 & 4,275 & 0.0000 & 0.9498 & 0.105 & 64,800 & -4,592 & 0.6923 \\ 
& 0.99  & 2,774 & -4,349 & 0.9900 & 4,587 & 0.0000 & 0.9899 & 0.010 & 64,801 & -4,445 & 0.8392 \\ 
& 0.999 & 2,836 & -4,348 & 0.9990 & 3,797 & 0.0000 & 0.9990 & 0.000 & 64,838 & -4,349 & 0.9049 \\ \hline
\end{longtable}

\vspace{-20pt}
\begin{longtable}{c | r| r c c| c c c c| c c c}
\caption{\centering PWL-I, PWL-O vs SAA for $\varrho = 2, \varsigma=5$ with different mixture weights} \label{tab:num-performance-2-5-diff} \\ 
\hline
\hline
\multicolumn{1}{c|}{$K$} 
& \multicolumn{1}{c|}{$\theta$} 
& \multicolumn{1}{c|}{\makecell{Time \\ (sec)}} 
& \multicolumn{1}{c|}{Obj\starfootnote} 
& \multicolumn{1}{c|}{${\check{\theta}}^I$} 
& \multicolumn{1}{c|}{\makecell{Time \\ (sec)}} 
& \multicolumn{1}{c|}{\%- Obj} 
& \multicolumn{1}{c|}{${\check{\theta}}^O$} 
& \multicolumn{1}{c|}{\%${\check{\theta}}$} 
& \multicolumn{1}{c|}{\makecell{Time \\ (sec)}}  
& \multicolumn{1}{c|}{Obj\starfootnote} 
& \multicolumn{1}{c}{${\check{\theta}}^S$} \\ \hline
\endfirsthead
\hline
\multicolumn{1}{c|}{$K$} 
& \multicolumn{1}{c|}{$\theta$} 
& \multicolumn{1}{c|}{\makecell{Time \\ (sec)}} 
& \multicolumn{1}{c|}{Obj} 
& \multicolumn{1}{c|}{${\check{\theta}}^I$} 
& \multicolumn{1}{c|}{\makecell{Time \\ (sec)}}  
& \multicolumn{1}{c|}{\%- Obj} 
& \multicolumn{1}{c|}{${\check{\theta}}^O$} 
& \multicolumn{1}{c|}{\%${\check{\theta}}$} 
& \multicolumn{1}{c|}{\makecell{Time \\ (sec)}}  
& \multicolumn{1}{c|}{Obj} 
& \multicolumn{1}{c}{${\check{\theta}}^S$} \\ \hline
\endhead

\multicolumn{12}{r}{\textit{(continued on next page)}}\\
\endfoot

\endlastfoot
\multicolumn{12}{c}{$n = 100$} \\ \hline
\multirow{3}{*}{5}  & 0.95  & 0.53 & -988 & 0.9500 & 0.32  & 0.0000 & 0.9447 & 0.560 & 0.341 & -894 & 1.000 \\ 
 & 0.99  & 0.36 & -988 & 0.9900 & 0.476 & 0.0000 & 0.9899 & 0.010 & 3.504 & -894 & 1.000 \\ 
 & 0.999 & 0.68 & -988 & 0.9990 & 0.76  & 0.0001 & 0.9990 & 0.000 & 11.25 & -894 & 1.000 \\ \hline
\multirow{3}{*}{10} & 0.95  & 2.18 & -891 & 0.9500 & 1.37  & 0.0000 & 0.9499 & 0.011 & 1,978.9 & -894 & 0.9458 \\ 
 & 0.99  & 1.82 & -890 & 0.9900 & 4.02  & 0.0000 & 0.9889 & 0.111 & 64,853 & -892 & 0.9882 \\ 
 & 0.999 & 14.3 & -890 & 0.9995 & 60.25 & 0.0190 & 0.9989 & 0.060 & 19,829 & -890 & 0.9970 \\ \hline
\multirow{3}{*}{15} & 0.95  & 3.39 & -928 & 0.9500 & 3.65  & 0.0003 & 0.9499 & 0.011 & 30.99 & -981 & 0.9434 \\ 
 & 0.99  & 11.8 & -926 & 0.9900 & 25.1  & 0.0000 & 0.9899 & 0.010 & 36,789 & -976 & 0.9857 \\ 
 & 0.999 & 58.91 & -922 & 0.9990 & 69.24 & 0.0001 & 0.9989 & 0.010 & 27,442 & -924 & 0.9922 \\ \hline
\multicolumn{12}{c}{$n=500$} \\ \hline
\multirow{3}{*}{5}  & 0.95  & 115 & -2,344 & 0.9673 & 117.4 & 0.0010 & 0.9664\daggerfootnote{} & - & 928.68 & -2,334 & 0.7964 \\ 
 & 0.99  & 166.2 & -2,343 & 0.9900 & 213.8 & 0.0005 & 0.9899 & 0.010 & 64,803  & -2,332 & 0.8411 \\ 
 & 0.999 & 292.6 & -2,342 & 0.9990 & 268.3 & 0.0002 & 0.9989 & 0.010 & 64,801  & -2,330 & 0.8527 \\ \hline
\multirow{3}{*}{10} & 0.95  & 410.4 & -2,500 & 0.9502 & 235.6 & 0.0002 & 0.9498 & 0.042 & 64,800  & -2,642 & 0.9194 \\ 
 & 0.99  & 267.7 & -2,499 & 0.9900 & 257.5 & 0.0003 & 0.9899 & 0.010 & 64,801  & -2,639 & 0.9826 \\ 
 & 0.999 & 14,400 & -2,498 & 0.9990 & 2,460  & 0.0002 & 0.9989 & 0.010 & 25,398  & -2,593 & 0.9817 \\ \hline
\multirow{3}{*}{15} & 0.95  & 737.7 & -2,188 & 0.9501 & 779.5 & 0.0001 & 0.9499 & 0.021 & 64,800  & -2,421 & 0.6018 \\ 
 & 0.99  & 633.6 & -2,187 & 0.9900 & 818.6 & 0.0006 & 0.9899 & 0.010 & 64,800  & -2,207 & 0.9324 \\ 
 & 0.999 & 4,694 & -2,185 & 0.9990 & 15,318 & 0.0001 & 0.9989 & 0.010 & 64,804  & -2,186 & 0.9427 \\ \hline
\multicolumn{12}{c}{$n=1000$} \\ \hline
\multirow{3}{*}{5}  & 0.95  & 858 & -4,556 & 0.9509 & 930   & 0.0004 & 0.9498 & 0.116 & 64,800  & -4,558 & 0.4818 \\ 
 & 0.99  & 1,176 & -4,554 & 0.9900 & 968.5 & 0.0001 & 0.9899 & 0.010 & 64,801  & -4,556 & 0.7318 \\ 
 & 0.999 & 948  & -4,553 & 0.9990 & 1,007  & 0.0001 & 0.9989 & 0.010 & 64,806  & -4,554 & 0.9052 \\ \hline
\multirow{3}{*}{10} & 0.95  & 2,580 & -7,286 & 0.9500 & 2,596  & 0.0000 & 0.9495 & 0.053 & 64,800  & -7,599 & 0.5847 \\ 
 & 0.99  & 2,540 & -7,285 & 0.9900 & 2,126  & 0.0001 & 0.9899 & 0.010 & 64,800  & -7,557 & 0.8751 \\ 
 & 0.999 & 6,267 & -7,285 & 0.9990 & 4,117  & 0.0000 & 0.9989 & 0.010 & 64,801  & -7,285 & 0.9507 \\ \hline
\multirow{3}{*}{15} & 0.95  & 3,652 & -4,046 & 0.9500 & 3,747  & 0.0001 & 0.9497 & 0.032 & 64,800  & -4,365 & 0.7277 \\ 
 & 0.99  & 564.5 & -4,046 & 0.9900 & 2,769  & 0.0000 & 0.9899 & 0.010 & 64,800  & -4,183 & 0.8842 \\ 
 & 0.999 & 25,548 & -4,045 & 0.9990 & 2,933  & 0.0000 & 0.9989 & 0.010 & 64,801  & -4,046 & 0.9584 \\ \hline
\end{longtable}

\vspace{-20pt}
\begin{longtable}{c r|r r r|r|r r r |c c c}
\caption{\centering Computational performance summary of PWL-I, PWL-O and SAA approaches  with different mixture weights across hyperparameters} \label{tab:num-performance-sum-diff} \\ 
\hline
 & & \multicolumn{3}{c|}{\makecell{Average Time \\ (sec)}} & ${\check{\theta}^S}$ & \multicolumn{3}{c|}{\makecell{\% of Instances not \\ Solved to \\ Optimality Gap }} & \multicolumn{3}{c}{\makecell{Avg. ${\check{\theta}^I, \check{\theta}^O, \check{\theta}^S}$ for \\ Instances not \\ Solved to \\ Optimality Gap}} \\ \hline
$K$ & $\theta$ & PWL-I & PWL-O & SAA & SAA & PWL-I & PWL-O & SAA & PWL-I & PWL-O & SAA \\ \hline
\endfirsthead
\hline 
$K$ & $\theta$ & PWL-I & PWL-O & SAA & SAA & PWL-I & PWL-O & SAA & PWL-I & PWL-O & SAA \\ \hline
\endhead

\multicolumn{12}{r}{\textit{(continued on next page)}}\\
\endfoot

\endlastfoot
\multicolumn{12}{c}{$n = 100$} \\ \hline
\multirow{3}{*}{5}  & 0.95  & 0.6   & 0.5   & 0.4    & 0.9927 & 0 & 0 & 0     & - & - & - \\ 
                   & 0.99  & 0.9   & 0.6   & 52      & 0.9908 & 0 & 0 & 0     & - & - & - \\ 
                   & 0.999 & 2.2   & 2.1   & 72      & 0.9925 & 0 & 0 & 0     & - & - & - \\ \hline
\multirow{3}{*}{10} & 0.95  & 2.6   & 1.56  & 15,108   & 0.9313 & 0 & 0 & 16.7  & - & - & 0.86 \\ 
                   & 0.99  & 14.0  & 14.4  & 64,809   & 0.9698 & 0 & 0 & 100   & - & - & 0.96 \\ 
                   & 0.999 & 52.5  & 52.0  & 50,488   & 0.9895 & 0 & 0 & 66.7  & - & - & 0.99 \\ \hline
\multirow{3}{*}{15} & 0.95  & 51.3  & 2.6   & 26,046   & 0.9054 & 0 & 0 & 28.6  & - & - & 0.89 \\ 
                   & 0.99  & 46.3  & 40.5  & 50,059   & 0.9604 & 0 & 0 & 66.7  & - & - & 0.94 \\ 
                   & 0.999 & 178   & 198.8 & 55,462   & 0.9793 & 0 & 0 & 75.0  & - & - & 0.97 \\ \hline

\multicolumn{12}{c}{$n=500$} \\ \hline
\multirow{3}{*}{5}  & 0.95  & 123   & 133   & 38,760   & 0.4934 & 0 & 0 & 50.0  & - & - & 0.52 \\ 
                   & 0.99  & 144   & 183   & 64,802   & 0.4253 & 0 & 0 & 100   & - & - & 0.42 \\ 
                   & 0.999 & 732   & 655   & 64,801   & 0.5539 & 0 & 0 & 100   & - & - & 0.55 \\ \hline
\multirow{3}{*}{10} & 0.95  & 501   & 250   & 64,800   & 0.8916 & 0 & 0 & 100   & - & - & 0.89 \\ 
                   & 0.99  & 260   & 366   & 64,813   & 0.8410 & 0 & 0 & 100   & - & - & 0.84 \\ 
                   & 0.999 & 2,504  & 2,384  & 58,235   & 0.9267 & 0 & 0 & 83.3  & - & - & 0.92 \\ \hline
\multirow{3}{*}{15} & 0.95  & 6,298  & 5,327  & 58,910   & 0.7653 & 0 & 0 & 66.7  & - & - & 0.76 \\ 
                   & 0.99  & 6,835  & 10,636 & 64,800   & 0.9090 & 0 & 0 & 100   & - & - & 0.91 \\ 
                   & 0.999 & 11,958 & 21,322 & 64,802   & 0.9063 & 0 & 13.3 & 100 & - & 0.9989 & 0.91 \\ \hline

\multicolumn{12}{c}{$n=1000$} \\ \hline
\multirow{3}{*}{5}  & 0.95  & 679   & 742   & 40,204   & 0.6654 & 0 & 0 & 50.0  & - & - & 0.53 \\ 
                   & 0.99  & 996   & 830   & 64,807   & 0.8376 & 0 & 0 & 100   & - & - & 0.84 \\ 
                   & 0.999 & 2,112  & 2,360  & 64,804   & 0.8836 & 0 & 0 & 100   & - & - & 0.88 \\ \hline
\multirow{3}{*}{10} & 0.95  & 2,400  & 3,294  & 54,106   & 0.6501 & 0 & 0 & 83.3  & - & - & 0.65 \\ 
                   & 0.99  & 2,400  & 1,976  & 64,812   & 0.8077 & 0 & 0 & 100   & - & - & 0.81 \\ 
                   & 0.999 & 2,537  & 8,525  & 64,802   & 0.8911 & 0 & 3.3 & 100  & - & 0.9999 & 0.89 \\ \hline
\multirow{3}{*}{15} & 0.95  & 4,929  & 13,426 & 64,800   & 0.7053 & 20 & 16.7 & 100 & 0.9978 & 0.9989 & 0.70 \\ 
                   & 0.99  & 17,313 & 22,778 & 64,817   & 0.8828 & 30 & 26.7 & 100 & 0.9989 & 0.9991 & 0.88 \\ 
                   & 0.999 & 21,022 & 31,354 & 64,863   & 0.9310 & 33.3 & 43.3 & 100 & 0.9997 & 0.9998 & 0.93 \\ \hline

\end{longtable}

\vspace{-10pt}
\normalsize
Tables~\ref{tab:num-performance-2-2-diff}--\ref{tab:num-performance-2-5-diff} show the time required to solve the instances, objective value at termination, and the values of  ${\check{\theta}}^I, {\check{\theta}}^O, {\check{\theta}}^S$. For PWL-O, we also report $\%$-Obj\dollarfootnote{} and $\%\check{\theta}$\treefootnote{} respectively defined as:  $\max\big(\bc^\top (\check{\bx}^I - \check{\bx}^O), 0\big)/\bc^\top \check{\bx}^I \times 100$ and $(\check{\theta}^I - \check{\theta}^O)/\theta \times 100$.  The results in these tables show that SAA reaches the 18-hour limit for nearly all instances (even some with $K=5$ components), whereas PWL-I and PWL-O are able to close the optimality gap much faster. In particular, PWL-I and PWL-O instances take at most 60 seconds for $n=100$, and less than 6 and 9 hours for $n=500$, and $1000$, respectively. Runtimes of PWL-I and PWL-O grow with $n$ and $K$. Both approximations take a similar amount of time.

Both the reported values of ${\check{\theta}}^O$ and ${\check{\theta}}^I$ are within the approximation accuracy ($\tau$) of the  target value of $\theta$. The few exceptions highlighted by symbol\daggerfootnote{} indicate that the solution we obtain for them lies in the feasible interior of the chance constraint. Except for those few instances, chance constraint satisfaction $\check{\theta}^I$ from PWL-I, in general, is at least as large as the target satisfaction probability $\theta$ and satisfaction $\check{\theta}^O$ obtained from PWL-O. Such an observation is consistent with the statements of Lemmas~\ref{lem: epsilon-reltn-z-phi-outer} and~\ref{lem: epsilon-reltn-z-phi-inner} made for the outer and inner approximation models, respectively. In contrast, chance constraint satisfaction probability ${\check{\theta}}^S$ at the SAA solutions is below the target value frequently—sometimes markedly below—suggesting inadequacy of the number of samples in the sample average approximation. Finally, PWL-I and PWL-O achieve almost the same optimal objective, with PWL-O having an objective value that is slightly smaller (as expected since the feasible set from inner approximation is contained in that by outer approximation). SAA objective values, however, are typically much smaller than those of PWL-O. Such smaller values come at the expense of their poor chance constraint satisfaction probability.

\subsubsection{Discussion on the summary Table~\ref{tab:num-performance-sum-diff} }
The results for PWL-I, PWL-O, and SAA summarized in Tables~\ref{tab:num-performance-sum-diff} report average computation time across instances, percentage of instances failing to achieve the desired optimality gap within 18 hours, and $\check{\theta}^I, \check{\theta}^O, \check{\theta}^S$ for the unsolved instances. This summary also highlights the superior performance of the PWL-I and PWL-O approximations over SAA in terms of both computational efficiency and probability satisfied by their solutions.  The aggregate performance statistics are highlighted below:

\begin{itemize}[leftmargin=1.0em]
    \item \textbf{Average computation time-based comparison:} For SAA, approximately 60\% of instances reach the 18-hour time limit, and in 74\% of the cases the computation time exceeds 14 hours. On the other hand, in any of the cases, the average time reported for PWL-I and PWL-O never reaches the time limit. Additionally, no significant difference is observed in the average time required to solve PWL-I and PWL-O.

    \item \textbf{Comparison of failure rates in achieving desired optimality gap:} In \(55\%\) of the test settings, SAA fails to achieve the prescribed gap \((\frac{1-\theta}{10}) \) within the time limit. In the remaining \(45\%\) cases, many SAA instances with \(K=10\) or \(K=15\) also fail to achieve the desired gap. In contrast, PWL-I and PWL-O approximations almost always achieve the gap, except for \(n=1000\) and \(\;K=15\) instances. Even in these cases, the percentage of instances that fail to meet the prescribed gap stays below \(50\%\). Additionally, for $n=1000, K=10$ and $n=500, K=15$, fewer than $10\%$ of the PWL-O instances fail to meet the required gap, on average.
    
    \item \textbf{Comparison based on the satisfied chance constraint probability: }For PWL-I and PWL-O approximations instances that failed to meet the optimality gap, the reported average $\check{\theta}^I, \check{\theta}^O$ values are mostly observed to be greater than the target probability $\theta$. This indicates that even when the desired optimality gap is not attained, solutions we obtain at termination satisfy the required chance constraint probability on average. This however may come at the cost of increased objective value. 
    
    On the other hand, SAA-generated solutions yield $\check{\theta}^S$ significantly below the target $\theta$, even when averaged over those instances that meet the required optimality gap. Understandably, average $\check{\theta}^S$ values become worse when instances that fail to satisfy the optimality requirement are also incorporated (see last column of Table~\ref{tab:num-performance-sum}). Furthermore, in comparison to PWL models, worsening performance trends of SAA are more noticeable with increasing $n$. Even with $n=500$ and $K=5$, some SAA runs yield $\check{\theta}^S$ values 2.5 times below the target, and with $n=1000$, these values can be 3-fold smaller.  Interestingly, the $\check{\theta}^S$ values for SAA appear to improve with $K$.  
\end{itemize}

\subsubsection{Effect on solution time with number of breakpoints}
We conducted additional experiments to examine the effect of the number of breakpoints and linear pieces on computation time and optimality gap at termination for both PWL-I and PWL-O approximations. For \(n=500\) and \(n=1000\), we randomly selected about 10\% instances from those used in our previous experimentation settings. We, however, now set the approximation accuracy \(\tau\) to \((1-\theta)/2\) and \((1-\theta)/50\), which require respectively fewer and more breakpoints (and linear pieces) than the baseline value \((1-\theta)/10\). Tables~\ref{tab: breakpoint-effect-PWL-I}-\ref{tab: breakpoint-effect-PWL-O} and Tables~\ref{tab: breakpoint-effect-PWL-I-1000}-\ref{tab: breakpoint-effect-PWL-O-1000} report the findings respectively for $n=500$ and $n=1000$. The first four columns of these tables show the values of mean and variance control parameters \(\varrho\) and \(\varsigma\), the number of components \(K\), and the probability threshold \(\theta\). As expected, increasing accuracy requirement increases solution time for both inner and outer approximations, with the effect being more pronounced for \(n=1000\) than for \(n=500\). For \(n=500\), every instance is solved within the 18-hour limit for both error tolerances and all values of \(K\). In contrast, for \(n=1000\) cases, many instances fail to reach the desired optimality gap that was set in Gurobi. 
About 40\% of instances with $n=1000$ yield no feasible solution within 18 hours for \(\tau = (1-\theta)/50\). However, in several instances, the optimality gap at termination is almost identical for \((1-\theta)/2\) and \((1-\theta)/50\) cases, indicating that Gurobi cannot tighten the optimality gap further regardless of the approximation accuracy. Such observation is more noticeable for PWL-I than PWL-O approximation, indicating that outer approximation is easier to solve, especially when a finer approximation is used. The reasons are unclear. However, we hypothesize that it is because the feasible set of PWL-I approximation is more restricted, thus making it difficult for the MIP solver to identify feasible solutions.

\subsubsection{Solution time and mixture distribution}
To examine the effect of the number of Gaussian mixture components and their parameters on solution time, we perform linear regression with the following predictors:
\begin{itemize}[leftmargin=2cm,labelwidth=1.0cm,labelsep=0.3cm,align=left]
    \item[(i)] $n$ – problem dimension 
    \item[(ii)] $K$ – number of mixture components
    \item[(iii)] $\varrho$ – mean-control hyperparameter
    \item[(iv)] $\varsigma$ – variance-control hyperparameter
    \item[(v)] $\bar d_{\mu}$ – average of the distances between any two Gaussian component means
    \item[(vi)] $d^{\max}_{\mu}$ – maximum of the distances between any two Gaussian component means
    \item[(vii)] average value of eigenvalue ratio $\nu^{\max}/\nu^{\min}$ across mixture components
    \item[(viii)] maximum value of eigenvalue ratio $\nu^{\max}/\nu^{\min}$ across mixture components
\end{itemize}

Analyses are performed separately for different values of \(\theta\). Regression results are reported in Appendix~\ref{appndx:regression-table} for both PWL-I and PWL-O approximations. Predictably, solution time increases with both \( n \) and \( K \), with \( K \) showing a greater effect. Other observations are as follows:

\begin{itemize}
    \item For both PWL-I and PWL-O approximations, predictors $\varsigma, \Bar{d}_{\mu}$ and \( d^{\max}_{\mu} \) exhibit statistically significant influence on  solution time, particularly with larger \(\theta\). 
    \item  For both PWL-I and PWL-O approximations, solution time and optimality gap tend to increase with greater variance in the component distribution. Such effects become more pronounced at higher $\theta$ values. The influence of variance, as captured by the regression coefficient and statistical significance of $\varsigma$, is slightly stronger for PWL-O than PWL-I approximation.
    \item  For PWL-I approximation with $\theta = 0.99, 0.999$,  it is observed that the smaller the $\Bar{d}_{\mu}$ value (\emph{average} distance across all pairs), the faster it is to reach the desired optimality gap. However, the effect of $d^{\max}_{\mu}$ value (maximum distance across all pairs) appears to be opposite.
     \item Similar effect of $\Bar{d}_{\mu}$ and $d^{\max}_{\mu}$ are observed for PWL-O approximation but they are statistically significant only when $\theta = 0.999$. Additionally, corresponding regression coefficient values are observed to be smaller than those for PWL-I approximation.
\end{itemize}

\section{Conclusion}
We studied a linear chance-constrained problem with continuous decision variables, assuming that the data follows a GMM distribution. We related the approximation accuracy of $\Phi$ to a perturbation in chance constraint satisfaction probability $\theta$ and showed that the desired optimality gap can be ensured through controlling the approximation accuracy. An extensive computational study with thousands of instances gives a clear indication of the superior performance of piecewise-linear inner and outer approximation over the popular SAA approach, in terms of both computation time and the chance constraint probability satisfied by the solution of these problems.  
We believe that insights from our numerical findings could facilitate customized algorithm development for solving problems with multiple or joint GMM-based chance constraints.
Extending this study to discrete decision variables and studying distributionally robust chance constraints for mixture models, with unknown component parameters, weights, and counts, are some promising future research directions.


\bibliographystyle{plain}
\bibliography{reference}

\begin{APPENDICES}
\normalsize
\section{Proofs}\label{appndx:proof}
\paragraph{\textbf{Proof of Proposition~\ref{prop:reformulation0}:}}
    Let $\bx$ be any feasible solution of~\eqref{eq:ccp}.
    First, suppose $\bx = \mathbf{0}$. Then $p(\mathbf{0}) = \mathbb{P} \left[ 0 \leq b \right]$, which is equal to $1$ if $b \geq 0$. If $b < 0$, then $p(\bx) = 0 \not\geq \theta$, contradicting feasibility of $\bx$. Thus, it must be true that $b \geq 0$ for $\bx = \mathbf{0}$.
    Now, observe that one constructs a feasible solution in~\eqref{eq:basic-GMM-cc} with the same objective value by setting $\bx = \mathbf{0}$, $\lambda_k = 0$, $\zeta_k = \theta$, and $z_k = \Phi^{-1}(\theta)$ for all $k \in [K]$.
    Suppose now that $\bx = \bar{\bx} \neq \mathbf{0}$.
    Again, we construct a feasible solution in~\eqref{eq:basic-GMM-cc} with the same objective value by setting $\bx = \bar{\bx}$, $\lambda_k = \sqrt{\bar{\bx}^\top \bSigma_k \bar{\bx}}$, $z_k = \frac{b - \bar{\bx}^\top \bmu_k}{\sqrt{\bar{\bx}^\top \bSigma_k \bar{\bx}}}$, and $\zeta_k = \Phi(z_k)$ for all $k \in [K]$.

    Now, let $(\bx, \bz, \bzeta, \blambda)$ be any feasible solution for~\eqref{prop:reformulation0}.
    We claim that $\bx$ is also feasible to~\eqref{eq:ccp} with the same objective value.
    First, suppose $\bx = \mathbf{0}$. Then for all $k \in [K]$, $\lambda_k = 0$ implies that $b \geq 0$ for any arbitrary value of $z_k$. Hence, $\Phi(z_k) \geq \zeta_k$ is always satisfied for all $k \in [K]$ and therefore $\sum_{k \in [K]} w_k \zeta_k \geq \theta$. 
    Suppose now that $\bx \neq \mathbf{0}$. Then
    \[
    p(\bx) = \sum_{k=1}^K w_k \Phi \left( \frac{b - \bx^\top \bmu_k}{\sqrt{\bx^\top \bSigma_k \bx}} \right)
    =
    \sum_{k=1}^K w_k \Phi \left( \frac{b - \bx^\top \bmu_k}{\lambda_k} \right)
    \geq
    \sum_{k=1}^K w_k \Phi \left( z_k \right)
    \geq 
    \sum_{k=1}^K w_k \zeta_k
    \geq
    \theta.
    \]
    The equality follows from~\eqref{eq:cc_probability_standardized_form} because of Assumption~\ref{assmp:positive-definiteness} and that $\bx \neq \mathbf{0}$, whereas the inequalities  follow from feasibility of $(\bx, \bz, \bzeta, \blambda)$ and monotonicity of $\Phi$. Hence, $p(\bx) \geq \theta$ of ~\eqref{eq:ccp} is satisfied having same objective value as that of~\eqref{prop:reformulation0}. \hfill\qedsymbol

\vspace{1em}
\paragraph{\textbf{Proof of Proposition~\ref{prop:phi_pwl_outer_reformulation}:}}
First, consider the forward direction of the claim starting from $\bar{\Phi} \left(z ; \bzz\right)  \geq \zeta$. Let us take some $z\geq 0, \; \zeta$ such that $\bar{\Phi}(z; \bzz) \geq \zeta$. Then, $z = y_3 \geq 0$ with $t_3 = 1$ implies that $t_1, t_2 = 0, \, g_i z + g_i^0 \geq \zeta, \; i \in [R]_0$ and $1 \geq \zeta$. Thus, for a given solution $z \geq 0$ satisfying $\bar{\Phi}(z; \bzz) \geq \zeta$, one can construct a feasible solution to the set $\mH(\bzz)$ using $t_1, t_2 = 0$, $t_3 = 1$, $\bs{\alpha} = \bs{0}$, $y_1, y_2 = 0, z = y_3 \geq 0$. 

Next we consider some $z \in [\check{z}_{-L},  0)$. Then, $z$ will be in $[\check{z}_{-i}, \check{z}_{-i+1}]$ for some $i \in [L]$ which indicates $t_2 = 1, \, t_1 = t_3 = 0$, and therefore, $y_1 = y_3 = 0$. Then there exists some weights $\hat{\alpha}_i, \hat{\alpha}_{i-1} \geq 0$ such that $\hat{\alpha}_i + \hat{\alpha}_{i-1} = 1 = t_2$. Furthermore, from constraint $y_2 = \sum_{i=0}^L \alpha_i \zz_{-i}$ we obtain $y_2 = \hat{\alpha}_i \check{z}_{-i} +  \hat{\alpha}_{i-1} \check{z}_{-i+1}$. Hence, $ \alpha_{i} \Phi(\zz_{-i}) + \alpha_{i-1} \Phi(\zz_{-i+1}) \geq \zeta$. Thus for any given $z \in [\check{z}_{-L},  0)$, one can find an index $i \in [L]$ such that $ \alpha_i = \hat{\alpha}_i, \, \alpha_{i-1} = \hat{\alpha}_{i-1}$ with $y_2 = \hat{\alpha}_i \check{z}_{-i} +  \hat{\alpha}_{i-1} \check{z}_{-i+1}, \, y_1 = y_3 = 0$, and  $\alpha_j = 0 \;\forall j \in [L] \backslash \{i-1, i\}$, which satisfy all the constraints of $\mathcal{H}(\bzz)$. Similarly, if we take some $z < \bzz_{-L}$, then $t_1 = 1,$ and $y_1 = z$. This implies $t_2, t_3, y_2, y_3 = 0, \Phi(\check{z}_{-L}) \geq  \zeta$. Therefore, $\mathcal{H}(\bzz) \neq \emptyset$.

\vspace{1em}
Now we prove that any point in the set $\mathcal{H}(\bzz)$ also satisfies $\bar{\Phi} \left(z ; \bzz\right)  \geq \zeta$ based on the definition of  $\bar{\Phi} \left(z ; \bzz\right)$ in~\eqref{def:def_phi_pwl_outer}. It is formally shown below using three separate cases:

\begin{enumerate}
    \item[i)] Let us choose $ y_3 = z \geq 0$ with $t_3 = 1$ and $ \, t_1, t_2, y_1, y_2 = 0$ such that $ g_i y_3 + g_i^0 \geq \zeta$ for all $i \in [R]_0$. Then  $\bar{\Phi} \left(z ; \, \bzz\right) = \min_{i \in [R]_0} \{ g_i y_3 + g_i^0\} \geq \zeta$.
    \item[ii)] Let $\zz_{-L} \leq y_2 = z < 0$ with $ t_2=1$ and  $t_1, t_3, y_1, y_3 = 0$ such that $y_2 = \alpha_{\check{i}} \zz_{-\check{i}} + \alpha_{\check{i} - 1} \zz_{-\check{i} + 1}, \quad$  $\alpha_{\check{i}} + \alpha_{\check{i} - 1} = 1 = t_2, $ and $ \alpha_{\check{i}} \Phi(\zz_{-\check{i}}) + \alpha_{\check{i} - 1}  \Phi(\zz_{-\check{i} + 1}) \geq \zeta $ for some $\check{i} \in [L]$ and $\alpha_{j} = 0$ for all $j \in [L] \backslash \{\check{i} - 1, \check{i} \}$. Thus there exists an index $\check{i} \in [L]$ that satisfies $\bar{\Phi} \left(z ; \bzz\right) =  h_{\check{i}} y_2 + h_{\check{i}}^0  \geq \zeta$. Hence, $\bar{\Phi} \left(z ; \bzz\right) =  \max_{i \in [L]}h_{i} y_2 + h_{i}^0  \geq \zeta$ is satisfied.
    \item[iii)] Finally, $ y_1 = z < \check{z}_{-L}$ with $ t_1 = 1, \, t_2, t_3, y_2, y_3 = 0$ indicates that $\Phi(\zz_{-L}) = \bar{\Phi} \left(z ; \bzz\right) \geq \zeta$. 
\end{enumerate}
Additionally, since all of these cases allow $\zeta$ to be at most 1, using the definition of $\bar{\Phi}(z; \bzz)$ in~\eqref{def:def_phi_pwl_outer}, $\bar{\Phi}(z; \bzz)  \geq \zeta$ is satisfied. \hfill\qedsymbol

\vspace{1em}
\paragraph{\textbf{Proof of \textit{Lemma~\ref{lem: epsilon-reltn-z-phi-outer}}:}}
We split the proof into two cases depending on the convexity/concavity of $\Phi$. Recall that for any $\hat{z}$ between $z_1$ and $z_2$, and for some $\mathfrak{z} \in (z_1,z_2)$, the Taylor expansion with Lagrange remainder around $z_1$ gives
\begin{align}\label{eq:taylor-expansion}
    \Phi(\hat{z}) = \Phi(z_1) + \Phi'(z_1)(\hat{z}-z_1) + \tfrac{1}{2}\Phi''(\mathfrak{z})\,(\hat{z}-z_1)^2.
\end{align}

\textbf{Case 1 (Secant approximation when $z<0$).}  
Let $z_1 < \hat{z} < z_2 < 0$.  
For notational convenience let the secant line be $S(z) = \bar{\Phi}(z; \zz)$ for $z < 0$: 
\begin{align}\label{eq: s-express}
    S(z) := \Phi(z_1) + \frac{\Phi(z_2)-\Phi(z_1)}{z_2-z_1}(z-z_1).
\end{align}

From Taylor's theorem at $z_1$, we have
\[
\Phi(\hat{z}) = \Phi(z_1) + \Phi'(z_1)(\hat{z}-z_1) + \tfrac{1}{2}\Phi''(\mathfrak{z})\,(\hat{z}-z_1)^2,
\]
for some $\mathfrak{z} \in (z_1,\hat{z}).$
Similarly, expanding at $z_1$ and plugging $z=z_2$ gives
\[
\Phi(z_2) = \Phi(z_1) + \Phi'(z_1)(z_2-z_1) + \tfrac{1}{2}\Phi''(\eta)\,(z_2-z_1)^2,
\]
for some $\eta \in (z_1,z_2).$
Hence, the slope of the secant is
\[
\frac{\Phi(z_2)-\Phi(z_1)}{z_2-z_1} 
= \Phi'(z_1) + \tfrac{1}{2}\Phi''(\eta)(z_2-z_1).
\]
Substituting this slope expression into~\eqref{eq: s-express} and using $z = \hat{z}$, we obtain
\[
S(\hat{z}) = \Phi(z_1) + \left(\Phi'(z_1) + \tfrac{1}{2}\Phi''(\eta)(z_2-z_1)\right)(\hat{z}-z_1).
\]
Next, subtracting~\eqref{eq:taylor-expansion} from $S(\hat{z})$, we find
\[
S(\hat{z}) - \Phi(\hat{z}) 
= \tfrac{1}{2}\Phi''(\eta)(z_2-z_1)(\hat{z}-z_1) - \tfrac{1}{2}\Phi''(\mathfrak{z})(\hat{z}-z_1)^2.
\]
Factoring out $(\hat{z}-z_1)$,
\begin{align}\label{eq: S_z-Phi_z}
    S(\hat{z}) - \Phi(\hat{z}) 
= \tfrac{1}{2}(\hat{z}-z_1)\Big(\Phi''(\eta)(z_2-z_1) - \Phi''(\mathfrak{z})(\hat{z}-z_1)\Big).
\end{align}

Next, we show that for some $\tilde{\mathfrak{z}}\in(\hat{z},z_2)$,
\begin{align}
 & \Phi''(\eta)(z_2-z_1) - \Phi''(\mathfrak{z})(\hat{z}-z_1) 
= \Phi''(\tilde{\mathfrak{z}})(z_2-\hat{z}),   \label{eq: IVT}
\end{align}

\noindent
using the fact that $\Phi$ is twice continuously differentiable on an interval containing $[z_1,z_2]$
with $z_1 < \hat{z} < z_2$ and applying the Mean Value Theorem to $\Phi'$ respectively on the intervals $[z_1,\hat{z}]$, $[\hat{z},z_2]$ and $[z_1,z_2]$.

\begin{itemize}[label={}, leftmargin=0.5em]
    \item By the Mean Value Theorem on $[\hat{z},z_2]$, there exists
$\tilde{\mathfrak{z}} \in (\hat{z},z_2)$ such that
\[
    \Phi'(z_2) - \Phi'(\hat{z}) = \Phi''(\tilde{\mathfrak{z}})\,(z_2 - \hat{z}), \,\,\, \text{and}
\]
    \item By the Mean Value Theorem on $[z_1,\hat{z}]$, there exists 
    $\mathfrak{z} \in (z_1,\hat{z})$ such that
    \[
        \Phi'(\hat{z}) - \Phi'(z_1) = \Phi''(\mathfrak{z})\,(\hat{z} - z_1).
    \]
    \item By adding the above two equalities, we get: $ \,\,\,  \Phi'(z_2) - \Phi'(z_1) =\Phi''(\tilde{\mathfrak{z}})\,(z_2 - \hat{z}) + \Phi''(\mathfrak{z})\,(\hat{z} - z_1)$.

    \item Now, by the Mean Value Theorem on $[z_1,z_2]$, there exists 
    $\eta \in (z_1,z_2)$ such that
    \[
        \Phi'(z_2) - \Phi'(z_1) = \Phi''(\eta)\,(z_2 - z_1).
    \]
\end{itemize}

\noindent Combining the last two equality expressions, we obtain~\eqref{eq: IVT}. Thus, plugging~\eqref{eq: IVT} into~\eqref{eq: S_z-Phi_z} gives us:
\[
S(\hat{z}) - \Phi(\hat{z}) = \tfrac{1}{2}\,\Phi''(\tilde{\mathfrak{z}})(z_2-\hat{z})(\hat{z}-z_1),
\quad \text{for some} \,\,\,\, \tilde{\mathfrak{z}}\in(\hat{z},z_2).
\]

Now, since $\Phi$ is convex on $(-\infty,0)$, we have $\Phi''(\Tilde{\mathfrak{z}})\ge 0$. Hence,
\[
0 \;\le\; S(\hat{z}) - \Phi(\hat{z}) \;\le\; \tfrac{1}{2} C(z_1,z_2) \, (z_2-\hat{z})(\hat{z}-z_1),
\]
where $C(z_1,z_2) := \max_{z \in [z_1,z_2]} |\Phi''(z)|$. Furthermore, by geometric inequality
\[
(z_2-\hat{z})(\hat{z}-z_1) \;\le\; \tfrac{1}{4}(z_2-z_1)^2.
\]
This inequality holds as equality at $\hat{z} = \tfrac{z_1+z_2}{2}$.  
Therefore, the secant-based approximation error is bounded by
\[
S(\hat{z}) - \Phi(\hat{z}) \;\le\; \tfrac{1}{8} C(z_1,z_2) \, (z_2-z_1)^2.
\]

\textbf{Case 2  (Tangent approximation when $z \geq 0$).}  
Let $z_1 < \hat{z} < z_2$ with $z_1\geq 0$.  
Consider the tangent line:
\[
T(z) := \Phi(z_1) + \Phi'(z_1)(z-z_1).
\]
Subtracting the Taylor expansion~\eqref{eq:taylor-expansion} from $T(z)$ gives
\[
T(\hat{z}) - \Phi(\hat{z}) = -\tfrac{1}{2}\Phi''(\mathfrak{z})\,(\hat{z}-z_1)^2
\]
for some $\mathfrak{z} \in (z_1,\hat{z})$.  
Since $\Phi(z)$ is concave in $z$ for $z \geq 0$, we have $\Phi''(\mathfrak{z})\le 0$. Hence, the error is nonnegative:
\[
0 \;\le\; T(\hat{z}) - \Phi(\hat{z}) \;\le\; \tfrac{1}{2} C(z_1,z_2)\,(\hat{z}-z_1)^2 \;\le\; \tfrac{1}{2} C(z_1,z_2)\,(z_2-z_1)^2.
\]
Similarly, expanding around $z_2$ yields
\[
0 \;\le\;T(\hat{z}) - \Phi(\hat{z}) \;\le\; \tfrac{1}{2} C(z_1,z_2)\,(z_2-\hat{z})^2 \;\le\; \tfrac{1}{2} C(z_1,z_2)\,(z_2-z_1)^2.
\]\hfill \qedsymbol


\begin{proposition}\label{prop: curv-property}
Let $z \in \mathbb{R}$. Then
\begin{equation*}
\Phi''(z) \quad \text{is} \quad 
\begin{cases}
\text{strictly increasing for} \,\, z \in (-\infty,-1]\\[4pt]
\text{strictly decreasing for} \,\, z \in [-1,1], \\[4pt]
\text{strictly increasing for} \,\, z \in [1,\infty), \\[6pt]
\end{cases}
\end{equation*}
with $\max_{z \leq 0} \Phi''(z) = \Phi''(-1)=\frac{1}{\sqrt{2\pi}}e^{-1/2}$ and $\min_{z \geq 0} \Phi''(z) = \Phi''(1)=-\frac{1}{\sqrt{2\pi}}e^{-1/2}$, respectively.
\end{proposition}

\begin{proof}
Note that $\Phi''(z) = \frac{d}{dz}\Bigl(\phi(z)\Bigr) = \frac{d}{dz}\Bigl(\tfrac{1}{\sqrt{2\pi}}e^{-z^2/2}\Bigr) = -z\,\phi(z)$. Set
\(\kappa(z) := -z\,\phi(z).\)
Then
\[
\kappa'(z) = -\phi(z) - z\,\phi'(z).
\]
Using $\phi'(z) = -z\,\phi(z)$, we get
\[
\kappa'(z) = -\phi(z) - z(-z\,\phi(z))
           = \phi(z)\,(z^2 - 1).
\]
Since $\phi(z) > 0$ for all $z$, the sign of $\kappa'(z)$ is the sign of $z^2-1$ and we obtain: 
\begin{itemize}[label={}]
    \item $\kappa'(z) > 0 \text{ for } z \in  (-\infty, -1), \quad
\kappa'(-1) = 0 \, $ with $\, \kappa(-1) = \frac{1}{\sqrt{2\pi}}e^{-1/2}$,
\item $\kappa'(z) < 0 \text{ for } z\in(-1, 1), \quad \kappa'(1) = 0 \,$ with $\,\kappa(1) = -\frac{1}{\sqrt{2\pi}}e^{-1/2}, \quad $ and
\item $ \kappa'(z) > 0 \;\text{ for } z \in  (1, \infty),$
\end{itemize}
which gives us the desired result.
\end{proof}

\vspace{1em}
\paragraph{\textbf{Proof of \textit{Theorem~\ref{thm:complexity-breakpoint}}}:}
We prove for the cases: $z \geq 0$ and $z <0$. 

First, let the number of linear pieces required for the outer approximating standard normal CDF curve when $0 \leq z \leq 1$ be $N_1$. Observe that corresponding iterative rule in Part 1 of Algorithm~\ref{alg:breakpoint_find} is such that $\sum_{i=1}^{N_1} (\check{z}_{i} - \check{z}_{i-1}) = 1$ where $\check{z}_{0} = 0, \check{z}_{N_1} = 1$ (line 1 of Algorithm~\ref{alg:breakpoint_find}). However, by Lemma~\ref{lem: epsilon-reltn-z-phi-outer} and since $\max_{z \in \mathbb{R}} |\Phi''(z)| = \frac{1}{\sqrt{2 \pi}}e^{-1/2}$,

$$1 = \sum_{i=1}^{N_1} (\check{z}_{i} - \check{z}_{i-1}) \geq  N_1 \sqrt{\frac{2 \tau}{\frac{1}{\sqrt{2 \pi}}e^{-1/2}}} = (N_1) \sqrt{2 \tau \sqrt{2 \pi} e^{0.5}}.$$
Hence, we have $N_1 \leq \sqrt{\frac{1}{2 \tau \sqrt{2 \pi} e^{0.5}}}$. 

For Part 2 of Algorithm~\ref{alg:breakpoint_find} ($z \geq 1$), let the number of pieces  be $N_2 + 1$ starting from $z = 1$ and the corresponding constructed breakpoints be
$$\check{z}_{i} - \check{z}_{i-1} = \sqrt{\frac{2 \tau}{\phi(\check{z}_{i-1}) \check{z}_{i-1}}} \qquad \forall i \in [N_2].$$

 Let $I_i = \check{z}_{i} - \check{z}_{i-1}, \;\; \forall i \in [N_2]$. By construction and curvature property of $\Phi(\cdot)$ from Proposition~\ref{prop: curv-property}, clearly $I_i \geq I_{i-1} \, \forall i \in [N_2]$. Then,

\begin{align}
    & \hat{I} := \sum_{i=1}^{N_2} I_i \geq I_{N_2} = \sqrt{\frac{2 \tau \sqrt{2 \pi} e^{0.5 (\check{z}_{N_2-1})^2}}{\check{z}_{(N_2-1)}}} \geq \sqrt{\frac{2 \tau \sqrt{2 \pi} e^{0.5 (\check{z}_{N_2-1})^2}}{\check{z}_{R}}} \nonumber\\
    & \implies ( \check{z}_{N_2-1} )^2 \leq 2 \log \frac{\hat{I}^2 \zz_R}{2 \tau \sqrt{2 \pi}}  \label{eq:break_point_cal}
 \end{align}

However, since Part 2 of Algorithm~\ref{alg:breakpoint_find} starts from $z=1$, 
$$\check{z}_{N_2-1} = 1 + \sum_{i=1}^{N_2-1} I_i \geq 1 + (N_2-1) I_1 = 1 + (N_2-1) \sqrt{2 \tau \sqrt{2 \pi} e^{0.5}}.$$

\noindent Plugging this $\check{z}_{N_2-1}$ value into \eqref{eq:break_point_cal} gives us:

$$(1 + (N_2-1) \sqrt{2 \tau \sqrt{2 \pi} e^{0.5}})^2 \leq 2 \log \frac{\hat{I}^2 \zz_R}{2 \tau \sqrt{2 \pi}}. $$

\noindent Note that since $\hat{I} \leq \zz_R$, $\,\, \hat{I}^2 \zz_R \leq \zz^3_R$. Hence, it follows that 
$$ N_2 \leq \frac{1}{\sqrt{2 \tau \sqrt{2 \pi} e^{0.5}}}\bigg(\sqrt{2 \log \frac{\zz^3_R}{2 \tau \sqrt{2 \pi}}} - 1\bigg) + 1.$$

\noindent Thus, the total number of breakpoints we obtain using Algorithm~\ref{alg:breakpoint_find} is

$$N_1 + N_2 + 1 \leq \sqrt{\frac{1}{2 \tau \sqrt{2 \pi} e^{0.5}}} + \frac{1}{\sqrt{2 \tau \sqrt{2 \pi} e^{0.5}}}\bigg(\sqrt{2 \log \frac{\zz^3_R}{2 \tau \sqrt{2 \pi}}} - 1\bigg) + 2,$$

\noindent i.e., $|\mathcal{A}^{R}| = O\bigg(\max\{\frac{1}{\sqrt{\tau}}, \frac{1}{\sqrt{\tau}} \sqrt{log(\frac{\zz^3_R}{\tau})}\}\bigg)$. Since the allowed approximation error is at most $\tau$ and ${\zz^3_R} \gg {\tau}$, $\sqrt{log(\frac{\zz^3_R}{\tau})} > 1$ and hence $\frac{1}{\sqrt{\tau}} \sqrt{log(\frac{\zz^3_R}{\tau})} > \frac{1}{\sqrt{\tau}}$. Therefore, $|\mathcal{A}^{R}| = O\bigg(\frac{1}{\sqrt{\tau}} \sqrt{log(\frac{\zz^3_R}{\tau})}\bigg)$. 

Because of the symmetric nature of the curvature of $\Phi(\cdot)$ with respect to $z=0$, one can derive the same complexity order for Algorithm~\ref{alg:breakpoint_find_negative}. 
Hence, if the choice of $\zz_R < \infty$ is such that $1 - \Phi(\zz_R) \leq \tau$ and $\zz_L > -\infty$ is such that $\Phi(\zz_L)\leq \tau$, we obtain the worst-case complexity $O\bigg(\frac{1}{\sqrt{\tau}} \sqrt{log(\frac{1}{\tau})}\bigg)$. \hfill\qedsymbol

\vspace{1em}
\paragraph{\textbf{Proof of Proposition~\ref{prop:phi_pwl_inner_reformulation}:}}
First,  consider the forward direction of the claim starting from $\underline{\Phi} \left(z ; \bzz\right)  \geq \zeta$. Let $0 \leq z < \zz_R$. Then $z= y_3$ and $t_3 = 1$ must hold in~\eqref{eq:reform-inner-sub}. Hence, $t_1, t_2, t_4, y_1, y_2, y_4 = 0$, which leads to $\, g_i y_3 + g_i^0 t_3  \geq \zeta, i \in [R]$ and $1 \geq \zeta$.  Thus, $\mG(\bzz) \neq \emptyset$. Similarly, if $\zz_R \leq z $, then $z = y_4,$ and  $t_4 = 1$ enforcing that $t_1, t_2, t_3, y_1, y_2, y_3 = 0$. Therefore, $\Phi(\zz_R) \geq \zeta$. Hence, feasibility is satisfied. 

Next, for $\zz_{-L} \leq z < 0$, there must exist an $ \alpha_{\hat{i}} =1$ for some $\hat{i} \in [L-1]_0$ implying that $t_2 = 1, y_2 = \mathfrak{z}_{\hat{i}} = z$. Therefore, $y_1, y_3, y_4, t_1, t_3, t_4 = 0$ and feasibility holds since $1 \geq \zeta$ and $h_{\hat{i}} y_2 + h^0_{\hat{i}} \geq \zeta$. 

Finally for $z < \zz_{-L}$, $t_1 = 1$, $z = y_1$ and $t_2, t_3, t_4, y_2, y_3, y_4 = 0$  is a feasible choice that satisfies $h_L y_1 + h^0_L \geq \zeta, \; 1 \geq \zeta, \; h_L y_1 + h^0_L \geq 0$. Therefore, for any $z \leq 0, \zeta \in \mathbb{R}$ such that $\underline{\Phi}(z, \bzz) = \max\{0, \max_{i \in [L]_0} h_i z + h^0_i\} \geq \zeta$ also satisfy the constraints defining the set $\mG(\bzz)$ in~\eqref{eq:reform-inner-sub}.

\vspace{0.5em}
We now prove the other direction. We claim that any point in the set $\mathcal{G}(\bzz)$ also satisfies $\underline{\Phi} \left(z ; \bzz\right)  \geq \zeta$, which follows from the definition of  $\underline{\Phi} \left(z ; \bzz\right)$ in~\eqref{def:def_phi_pwl_inner} as follows:
\begin{enumerate}
    \item[i)] $  y_4 = z \geq \check{z}_R$ with $t_4 = 1$ and $t_1,  t_2, t_3, y_1, y_2, y_3 = 0$ indicates $\underline{\Phi} \left(z ; \bzz\right)  = \Phi(\zz_{R}) \geq \zeta$,
    
    \item[ii)] $\check{z}_R > z = y_3 \geq 0$ with $t_3 = 1$ and $t_1, t_2, t_4, y_1, y_2, y_4 = 0$  satisfies $g_i y_3 + g^0_{\check{i}} \geq \zeta$ for all $i \in [R]$,
    
    \item[iii)] $\zz_{-L} \leq y_2 = z < 0$ with $t_2 = 1$ and $t_1, t_3, t_4, y_1, y_3, y_4 = 0$ indicates that there exists some $\check{i} \in [L-1]$ such that $y_2 = \mathfrak{z}_{\check{i}}, \alpha_{\check{i}} = 1$ satisfying  $\underline{\Phi} \left(z ; \bzz\right) = h_{\check{i}} \mathfrak{z}_{\check{i}} + h_{\check{i}}^0 \geq \zeta$, and 
    
    \item[iv)] finally, $ y_1 = z < \check{z}_{-L}$ with $ t_1 = 1$ and $t_2, t_3, y_2, y_3 = 0$ leads to $h_L y_1 + h_L^0 \geq \zeta \geq 0$.
\end{enumerate}
Therefore, $\zeta$ satisfy both the secant and tangent cases of  definition~\eqref{def:def_phi_pwl_inner} and $\underline{\Phi}(z; \bzz)  \geq \zeta$ holds. \hfill\qedsymbol

\vspace{1em}

\section{Inner Approximation (PWL-I)}\label{appndx: PWL-I}

\begin{lemma}\label{lem: epsilon-reltn-z-phi-inner}
Let, for any $z_1, z_2 \in \mathbb R$,  $C(z_1, z_2) \coloneqq \max_{z \in [z_1, z_2]} |\Phi''(z)|$, $\bzz$ be any valid array of breakpoints as defined in~\eqref{eqn:ValidArrayBreakPoints}, and $\tau >0$. Then the following holds:
\begin{itemize}
    \item[a)] If $\zz_{i} - \zz_{i-1} \leq 2 \sqrt{\frac{2 \tau}{C(\zz_{i-1}, \zz_{i})}}$ for all $i \in [R]$, then $0 \leq \Phi(z) - \underline{\Phi}(z; \bzz) \leq \tau$ for all $z \geq 0$.

    \item[b)] If $\zz_{-i+1} - \zz_{-i} \leq \sqrt{\frac{2 \tau}{C(\zz_{-i}, \zz_{-i+1})}}$ for all $i \in [L]$, $0 \leq \Phi(z) - \underline{\Phi}(z; \bzz) \leq \tau $ for all $z < 0$.
\end{itemize}
\end{lemma}

\begin{proof}
    Proof of this Lemma is similar to that of Lemma~\ref{lem: epsilon-reltn-z-phi-outer}.
\end{proof}
The next result depends on an algorithm that is the counterpart to Algorithms~\ref {alg:breakpoint_find} and ~\ref{alg:breakpoint_find_negative} for inner approximation. 

\begin{theorem}\label{thm:complexity-breakpoint-inner}
Let $\zz_{-L}, \zz_R$  be such that $\max\big\{1 - \Phi(\zz_{R}), \Phi(\zz_{-L})\big\} \leq \tau$, where $0 < \tau \ll \min\{\zz_R, \, |\zz_{-L}|\}$. Then $O(\frac{1} {\sqrt{\tau}} \sqrt{log(\frac{1}{\tau})})$ breakpoints are sufficient to ensure $0 \leq \Phi(z) - \underline{\Phi}(z; \bzz) \leq \tau, \, \forall z \in \mathbb{R}$.
\end{theorem}

\begin{proof}
   Omitted because of the similarity to the proof of Theorem~\ref{thm:complexity-breakpoint}.
\end{proof}


\section{Additional Computational Results}\label{appndx: Ommtd-Table}

\vspace{-10pt}
\scriptsize
\begin{longtable}{c | r| r c c| r c c c| r c c}
\caption{\centering Comparison of PWL-I, PWL-O, and SAA approximations for $\varrho = 2, \varsigma=2$ with equal mixture weight and 64,800s (18-hour) time limit} \label{tab:num-performance-2-2-equal} \\ 
\hline
\multicolumn{2}{c|}{} & \multicolumn{3}{c|}{PWL-I} & \multicolumn{4}{c|}{PWL-O} & \multicolumn{3}{c}{SAA} \\ \hline
\multicolumn{1}{c|}{$K$} 
& \multicolumn{1}{c|}{$\theta$} 
& \multicolumn{1}{c|}{\makecell{Time \\ (sec)}} 
& \multicolumn{1}{c|}{Obj} 
& \multicolumn{1}{c|}{${\check{\theta}}^I$} 
& \multicolumn{1}{c|}{\makecell{Time \\ (sec)}} 
& \multicolumn{1}{c|}{\%- Obj}
& \multicolumn{1}{c|}{${\check{\theta}}^O$} 
& \multicolumn{1}{c|}{\%${\check{\theta}}$} 
& \multicolumn{1}{c|}{\makecell{Time \\ (sec)}}  
& \multicolumn{1}{c|}{Obj\starfootnote} 
& \multicolumn{1}{c}{${\check{\theta}}^S$} \\ 
\hline
\endfirsthead
\hline
\multicolumn{1}{c|}{$K$} 
& \multicolumn{1}{c|}{$\theta$} 
& \multicolumn{1}{c|}{\makecell{Time \\ (sec)}} 
& \multicolumn{1}{c|}{Obj} 
& \multicolumn{1}{c|}{${\check{\theta}}^I$} 
& \multicolumn{1}{c|}{\makecell{Time \\ (sec)}}  
& \multicolumn{1}{c|}{\%- Obj} 
& \multicolumn{1}{c|}{${\check{\theta}}^O$} 
& \multicolumn{1}{c|}{\%${\check{\theta}}$} 
& \multicolumn{1}{c|}{\makecell{Time \\ (sec)}}  
& \multicolumn{1}{c|}{Obj} 
& \multicolumn{1}{c}{${\check{\theta}}^S$} \\ \hline
\endhead

\multicolumn{12}{r}{\textit{(continued on next page)}}\\
\endfoot

\endlastfoot
\multicolumn{12}{c}{$n = 100$} \\ \hline
\multirow{3}{*}{5} & 0.95  & 0 & -735 & 0.9500 & 0 & 0.0003 & 0.9496 & 0.042 & 0.7 & -844 & 0.9129 \\  
& 0.99  & 1 & -735 & 0.9900 & 1 & 0.0002 & 0.9900 & 0.000 & 528 & -844 & 0.9032 \\ 
& 0.999 & 4 & -734 & 0.9990 & 4 & 0.0005 & 0.9990 & 0.000 & 109 & -844 & 0.9036 \\ \hline
\multirow{3}{*}{10} & 0.95  & 1 & -726 & 0.9500 & 2 & 0.0003 & 0.9496 & 0.042 & 64,800 & -727 & 0.8551 \\ 
& 0.99  & 16 & -726 & 0.9900 & 15 & 0.0002 & 0.9899 & 0.010 & 64,800 & -726 & 0.9092 \\ 
& 0.999 & 71 & -725 & 0.9990 & 31 & 0.0002 & 0.9990 & 0.000 & 64,800 & -726 & 0.9474 \\ \hline
\multirow{3}{*}{15} & 0.95  & 2 & -706 & 0.9501 & 1 & 0.0005 & 0.9499 & 0.021 & 64,800 & -707 & 0.9099 \\ 
& 0.99  & 10 & -705 & 0.9900 & 7 & 0.0005 & 0.9900 & 0.000 & 64,800 & -706 & 0.9416  \\ 
& 0.999 & 15 & -705 & 0.9990 & 66 & 0.0002 & 0.9990 & 0.000 & 64,800  & -706 & 0.9648\\ \hline
\multicolumn{12}{c}{$n=500$} \\ \hline
\multirow{3}{*}{5} & 0.95  & 120 & -2,455 & 0.9506 & 107 & 0.0004 & 0.9497 & 0.095 & 1,198 & -2,474 & 0.5709 \\ 
& 0.99  & 109 & -2,454 & 0.9900 & 141 & 0.0001 & 0.9900 & 0.000 & 64,801 & -2,473 & 0.4687 \\ 
& 0.999 & 147 & -2,453 & 0.9990 & 147 & 0.0002 & 0.9990 & 0.000 & 64,804 & -2,473 & 0.4958 \\ \hline
\multirow{3}{*}{10} & 0.95  & 293 & -1,811 & 0.9504 & 321 & 0.0007 & 0.9497 & 0.074 & 64,800 & -1,812 & 0.7463 \\ 
& 0.99  & 192 & -1,810 & 0.9900 & 276 & 0.0004 & 0.9900 & 0.000 & 64,810 & -1,812 & 0.8386 \\ 
& 0.999 & 2,021 & -1,809 & 0.9990 & 1,159 & 0.0001 & 0.9990 & 0.000 & 64,801 & -1,811 & 0.8799 \\ \hline
\multirow{2}{*}{15} & 0.95  & 493 & -2,338 & 0.9502 & 600 & 0.0002 & 0.9497 & 0.053 & 64,802 & -2,339 & 0.7268 \\ 
& 0.99  & 575 & -2,337 & 0.9900 & 714 & 0.0001 & 0.9900 & 0.000 & 64,801 & -2,338 & 0.8804 \\ \hline
\multicolumn{12}{c}{$n=1000$} \\ \hline
\multirow{3}{*}{5} & 0.95  & 904 & -7,794 & 0.9506 & 858 & 0.0000 & 0.9496 & 0.105 & 63,161 & -7,794 & 0.3805 \\ 
& 0.99  & 993 & -7,793 & 0.9901 & 798 & 0.0000 & 0.9900 & 0.010 & 64,801 & -7,794 & 0.6332 \\ 
& 0.999 & 923 & -7,793 & 0.9990 & 876 & 0.0000 & 0.9990 & 0.000 & 64,800 & -7,794 & 0.8626 \\ \hline
\multirow{3}{*}{10} & 0.95  & 3,069 & -5,810 & 0.9501 & 2,651 & 0.0001 & 0.9496 & 0.053 & 64,800 & -5,811 & 0.5688 \\ 
& 0.99  & 2,847 & -5,810 & 0.9900 & 2,147 & 0.0001 & 0.9899 & 0.010 & 64,801 & -5,810 & 0.6370 \\ 
& 0.999 & 6,530 & -5,809 & 0.9990 & 3,722 & 0.0000 & 0.9990 & 0.000 & 64,801 & -5,810 & 0.7392 \\ \hline
\multirow{3}{*}{15} & 0.95  & 3,306 & -4,345 & 0.9748 & 3,691 & 0.0820 & 0.9801\specialfootnote{†}{For these instances, chance constraint appears inactive at target satisfaction probability} & - & 64,800 & -4,349 & 0.7403 \\ 
& 0.99  & 3,427 & -4,348 & 0.9901 & 4,593 & 0.0000 & 0.9900 & 0.010 & 64,800 & -4,349 & 0.7599 \\ 
& 0.999 & 2,984 & -4,348 & 0.9990 & 3,098 & 0.0000 & 0.9990 & 0.000 & 64,800 & -4,349 & 0.8167 \\ \hline
\end{longtable}

\vspace{-20pt}
\begin{longtable}{c | r| r c c| c c c c| c c c}
\caption{\centering PWL-I, PWL-O vs SAA for $\varrho = 2, \varsigma=5$ with equal mixture weight} \label{tab:num-performance-2-5-equal} \\ 
\hline
\hline
\multicolumn{1}{c|}{$K$} 
& \multicolumn{1}{c|}{$\theta$} 
& \multicolumn{1}{c|}{\makecell{Time \\ (sec)}} 
& \multicolumn{1}{c|}{Obj} 
& \multicolumn{1}{c|}{${\check{\theta}}^I$} 
& \multicolumn{1}{c|}{\makecell{Time \\ (sec)}}  
& \multicolumn{1}{c|}{\%- Obj} 
& \multicolumn{1}{c|}{${\check{\theta}}^O$} 
& \multicolumn{1}{c|}{\%${\check{\theta}}$} 
& \multicolumn{1}{c|}{\makecell{Time \\ (sec)}}  
& \multicolumn{1}{c|}{Obj} 
& \multicolumn{1}{c}{${\check{\theta}}^S$} \\ \hline
\endfirsthead
\hline
\multicolumn{1}{c|}{$K$} 
& \multicolumn{1}{c|}{$\theta$} 
& \multicolumn{1}{c|}{\makecell{Time \\ (sec)}} 
& \multicolumn{1}{c|}{Obj} 
& \multicolumn{1}{c|}{${\check{\theta}}^I$} 
& \multicolumn{1}{c|}{\makecell{Time \\ (sec)}}  
& \multicolumn{1}{c|}{\%- Obj} 
& \multicolumn{1}{c|}{${\check{\theta}}^O$} 
& \multicolumn{1}{c|}{\%${\check{\theta}}$} 
& \multicolumn{1}{c|}{\makecell{Time \\ (sec)}}  
& \multicolumn{1}{c|}{Obj} 
& \multicolumn{1}{c}{${\check{\theta}}^S$} \\ \hline
\endhead

\multicolumn{12}{r}{\textit{(continued on next page)}}\\
\endfoot

\endlastfoot
\multicolumn{12}{c}{$n = 100$} \\ \hline
\multirow{3}{*}{5}  & 0.95  & 0.44 & -988 & 0.9500 & 0.337 & 0.0000 & 0.9498 & 0.021 & 0.244  & -894 & 1.0000 \\ 
 & 0.99  & 0.45 & -988 & 0.9900 & 0.364 & 0.0000 & 0.9900 & 0.000 & 4.01   & -894 & 1.0000 \\ 
 & 0.999 & 0.67 & -988 & 0.9990 & 0.592 & 0.0001 & 0.9989 & 0.010 & 11.56  & -894 & 1.0000 \\ \hline
\multirow{3}{*}{10} & 0.95  & 1.34 & -891 & 0.9500 & 3.578 & 0.0000 & 0.9499 & 0.011 & 797.5  & -891 & 0.9438 \\ 
 & 0.99  & 2.10 & -890 & 0.9900 & 1.921 & 0.0001 & 0.9899 & 0.010 & 64,819  & -891 & 0.9758 \\ 
 & 0.999 & 16.1 & -890 & 0.9990 & 24.87 & 0.0003 & 0.9989 & 0.010 & 16,043  & -890 & 0.9973 \\ \hline
\multirow{3}{*}{15} & 0.95  & 1.73 & -927 & 0.9500 & 0.937 & 0.0003 & 0.9499 & 0.011 & 5,404   & -928 & 0.9388 \\ 
 & 0.99  & 20.1 & -924 & 0.9900 & 13.51 & 0.0001 & 0.9899 & 0.010 & 64,800  & -925 & 0.9815 \\ 
 & 0.999 & 10.1 & -922 & 0.9990 & 25.11 & 0.0009 & 0.9989 & 0.010 & 64,801  & -924 & 0.9940 \\ \hline
\multicolumn{12}{c}{$n=500$} \\ \hline
\multirow{3}{*}{5}  & 0.95  & 117 & -2,344 & 0.9508 & 112.3 & 0.0007 & 0.9496 & 0.021 & 1,009   & -2,332 & 0.7247 \\ 
 & 0.99  & 118 & -2,343 & 0.9902 & 154.6 & 0.0007 & 0.9899 & 0.030 & 33,944  & -2,332 & 0.7345 \\ 
 & 0.999 & 169 & -2,341 & 0.9990 & 175.2 & 0.0003 & 0.9989 & 0.010 & 64,804  & -2,330 & 0.7553 \\ \hline
\multirow{3}{*}{10} & 0.95  & 284 & -2,500 & 0.9500 & 234.2 & 0.0002 & 0.9497 & 0.032 & 64,801  & -2,501 & 0.9152 \\ 
 & 0.99  & 185 & -2,499 & 0.9900 & 251.9 & 0.0003 & 0.9899 & 0.010 & 64,804  & -2,500 & 0.9239 \\ 
 & 0.999 & 685 & -2,498 & 0.9990 & 668   & 0.0002 & 0.9989 & 0.010 & 64,800  & -2,499 & 0.9697 \\ \hline
\multirow{3}{*}{15} & 0.95  & 586   & -2,187 & 0.9503 & 709.4 & 0.0008 & 0.9497 & 0.063 & 64,800  & -2,188 & 0.8253 \\ 
 & 0.99  & 691   & -2,185 & 0.9900 & 814.1 & 0.0004 & 0.9899 & 0.010 & 64,801  & -2,187 & 0.8708 \\ 
 & 0.999 & 19,812 & -2,184 & 0.9990 & 7,811  & 0.0001 & 0.9989 & 0.010 & 64,806  & -2,186 & 0.9031 \\ \hline
\multicolumn{12}{c}{$n=1000$} \\ \hline
\multirow{3}{*}{5}  & 0.95  & 1,033 & -4,556 & 0.9507 & 1,007 & 0.0003 & 0.9498 & 0.095 & 32,337  & -4,557 & 0.5991 \\ 
 & 0.99  & 1,197 & -4,554 & 0.9900 & 942.3 & 0.0001 & 0.9899 & 0.010 & 64,800  & -4,556 & 0.8160 \\ 
 & 0.999 & 1,087 & -4,553 & 0.9990 & 1,008 & 0.0001 & 0.9989 & 0.010 & 57,555  & -4,555 & 0.9204 \\ \hline
\multirow{3}{*}{10} & 0.95  & 2,324 & -7,286 & 0.9503 & 2,414 & 0.0001 & 0.9495 & 0.084 & 64,800  & -7,286 & 0.6534 \\ 
 & 0.99  & 1,440 & -7,285 & 0.9900 & 2,099 & 0.0000 & 0.9899 & 0.010 & 64,800  & -7,286 & 0.8194 \\ 
 & 0.999 & 5,683 & -7,285 & 0.9990 & 2,975 & 0.0000 & 0.9989 & 0.010 & 64,802  & -7,285 & 0.8988 \\ \hline
\multirow{3}{*}{15} & 0.95  & 3,960 & -4,046 & 0.9501 & 3,492 & 0.0000 & 0.9500 & 0.011 & 64,800  & -4,047 & 0.8121 \\ 
 & 0.99  & 557.2 & -4,046 & 0.9901 & 6,801 & 0.0000 & 0.9900 & 0.010  & 64,800  & -4,047 & 0.8997 \\ 
 & 0.999 & 19,580 & -4,044 & 0.9990 & 8,935 & 0.0000 & 0.9989 & 0.010 & 64,800  & -4,046 & 0.9431 \\ \hline
\end{longtable}

\vspace{-20pt}
\begin{longtable}{c r|r r r|c|r r r |r r r}
\caption{\centering Computational performance summary of PWL-I, PWL-O and SAA based approaches with equal mixture weights and 18-hour (64,800s) time limit across different hyperparameters} \label{tab:num-performance-sum} \\ 
\hline
 & & \multicolumn{3}{c|}{\makecell{Average Time \\ (sec)}} & ${\check{\theta}^S}$ & \multicolumn{3}{c|}{\makecell{\% of Instances not \\ solved to \\ Optimality Gap }} & \multicolumn{3}{c}{\makecell{Avg. ${\check{\theta}^I, \check{\theta}^O, \check{\theta}^S}$ for \\ Instances not \\ Solved to \\ Optimality Gap}} 
 \\ \hline
$K$ & $\theta$ & PWL-I & PWL-O & SAA & SAA & PWL-I & PWL-O & SAA & PWL-I & PWL-O & SAA \\ \hline
\endfirsthead
\hline 
$K$ & $\theta$ & PWL-I & PWL-O & SAA & SAA & PWL-I & PWL-O & SAA & PWL-I & PWL-O & SAA \\ \hline
\endhead

\multicolumn{12}{r}{\textit{(continued on next page)}}\\
\endfoot

\endlastfoot
\multicolumn{12}{c}{$n = 100$} \\ \hline
\multirow{3}{*}{5}  & 0.95  & 0.65  & 0.47  & 0.34    & 0.9854 & 0 & 0 & 0     & - & - & - \\ 
                   & 0.99  & 0.80  & 0.68  & 108     & 0.9838 & 0 & 0 & 0     & - & - & - \\ 
                   & 0.999 & 2.20  & 2.10  & 39      & 0.9839 & 0 & 0 & 0     & - & - & - \\ \hline
\multirow{3}{*}{10} & 0.95  & 2.30  & 1.70  & 34,648   & 0.9182 & 0 & 0 & 50.0  & - & - & 0.89 \\ 
                   & 0.99  & 4.50  & 3.40  & 64,804   & 0.9625 & 0 & 0 & 100   & - & - & 0.96 \\ 
                   & 0.999 & 27.4  & 33.0  & 56,674   & 0.9817 & 0 & 0 & 83.3  & - & - & 0.97 \\ \hline
\multirow{3}{*}{15} & 0.95  & 15.4  & 2.50  & 35,150   & 0.9201 & 0 & 0 & 40.0  & - & - & 0.89 \\ 
                   & 0.99  & 65.8  & 16.7  & 64,801   & 0.9552 & 0 & 0 & 100   & - & - & 0.95 \\ 
                   & 0.999 & 143   & 161   & 64,800   & 0.9749 & 0 & 0 & 100   & - & - & 0.97 \\ \hline

\multicolumn{12}{c}{$n = 500$} \\ \hline
\multirow{3}{*}{5}  & 0.95  & 100   & 104   & 39,315   & 0.4511 & 0 & 0 & 50.0  & - & - & 0.47 \\ 
                   & 0.99  & 104   & 143   & 59,658   & 0.4745 & 0 & 0 & 83.3  & - & - & 0.42 \\ 
                   & 0.999 & 696   & 555   & 63,038   & 0.5061 & 0 & 0 & 83.3  & - & - & 0.49 \\ \hline
\multirow{3}{*}{10} & 0.95  & 954   & 901   & 64,800   & 0.7256 & 0 & 0 & 100   & - & - & 0.72 \\ 
                   & 0.99  & 364   & 696   & 64,806   & 0.8744 & 0 & 0 & 100   & - & - & 0.87 \\ 
                   & 0.999 & 2,227  & 1,923  & 64,801   & 0.8918 & 0 & 0 & 100   & - & - & 0.89 \\ \hline
\multirow{3}{*}{15} & 0.95  & 10,555 & 9502 & 64801  & 0.7393 & 0 & 0 & 100   & - & - & 0.73 \\ 
                   & 0.99  & 5,960  & 4,290  & 64,801   & 0.8490 & 0 & 0 & 100   & - & - & 0.84 \\ 
                   & 0.999 & 6,585  & 21,120 & 64,813   & 0.8924 & 0 & 16.7 & 100 & - & 0.999 & 0.89 \\ \hline

\multicolumn{12}{c}{$n=1000$} \\ \hline
\multirow{3}{*}{5}  & 0.95  & 670   & 1,136  & 53,418   & 0.4996 & 0 & 0 & 50.0  & - & - & 0.36 \\ 
                   & 0.99  & 778   & 1,355  & 64,805   & 0.7633 & 0 & 0 & 100   & - & - & 0.76 \\ 
                   & 0.999 & 952   & 1,764  & 63,593   & 0.9059 & 0 & 0 & 83.3  & - & - & 0.90 \\ \hline
\multirow{3}{*}{10} & 0.95  & 8,414  & 10,823 & 64,800   & 0.6878 & 0 & 0 & 100   & - & - & 0.68 \\ 
                   & 0.99  & 5,789  & 9,306  & 64,801   & 0.8001 & 0 & 3.4 & 100  & - & 0.9999 & 0.80 \\ 
                   & 0.999 & 5,926  & 7,439  & 64,803   & 0.8724 & 0 & 0 & 100   & - & - & 0.87 \\ \hline
\multirow{3}{*}{15} & 0.95  & 18,411 & 6,833  & 64,800   & 0.7293 & 13.3 & 6.7 & 100  & 0.9999 & 0.9939 & 0.72 \\ 
                   & 0.99  & 19,892 & 21,198 & 64,800   & 0.8548 & 27.5 & 26.7 & 100 & 0.9998 & 0.9998 & 0.85 \\ 
                   & 0.999 & 22,210 & 28,666 & 64,801   & 0.8915 & 13.3 & 36.7 & 100 & 0.9998 & 0.9996 & 0.89 \\ \hline

\end{longtable}

\begin{longtable}{r r r r|r c|r c|r c|r c}
\caption{\centering Effect of number of breakpoints on solution time and optimality gap for PWL-I $(n=500)$. Solution time is reported in \emph{hours} (hr) from now onwards for ease of reading the table. Gap \% represents the optimality gap (in percentage) obtained at termination from Gurobi.} \label{tab: breakpoint-effect-PWL-I}\\
\hline
\multirow{3}{*}{$\varrho$} & 
\multirow{3}{*}{$\varsigma$} & 
\multirow{3}{*}{$K$} & 
\multirow{3}{*}{$\theta$} &
  \multicolumn{4}{c|}{\textit{Equal weight}} &
  \multicolumn{4}{c}{\textit{Different weight}} \\
\cline{5-12}
& & & 
  & \multicolumn{2}{c|}{$\tau = (1-\theta)/2$} 
  & \multicolumn{2}{c|}{$\tau = (1-\theta)/50$}
  & \multicolumn{2}{c|}{$\tau = (1-\theta)/2$}
  & \multicolumn{2}{c}{$\tau = (1-\theta)/50$} \\
\cline{5-12}
& & & 
  & Time (hr) & Gap \% 
  & Time (hr) & Gap \% 
  & Time (hr) & Gap \% 
  & Time (hr) & Gap \% \\ 
\hline
\endfirsthead

\multicolumn{12}{c}{{\tablename\ \thetable\ -- \textit{continued from previous page}}} \\
\hline
\multirow{3}{*}{$\varrho$} & 
\multirow{3}{*}{$\varsigma$} & 
\multirow{3}{*}{$K$} & 
\multirow{3}{*}{$\theta$} &
  \multicolumn{4}{c|}{\textit{Equal weight}} &
  \multicolumn{4}{c}{\textit{Different weight}} \\
\cline{5-12}
& & & 
  & \multicolumn{2}{c|}{$\tau = (1-\theta)/2$} 
  & \multicolumn{2}{c|}{$\tau = (1-\theta)/50$}
  & \multicolumn{2}{c|}{$\tau = (1-\theta)/2$}
  & \multicolumn{2}{c}{$\tau = (1-\theta)/50$} \\
\cline{5-12}
& & & 
  & Time (hr) & Gap \% 
  & Time (hr) & Gap \% 
  & Time (hr) & Gap \% 
  & Time (hr) & Gap \% \\ 
\hline
\endhead
\hline \multicolumn{12}{r}{\textit{(continued on next page)}}\\
\endfoot
\endlastfoot

10 &  5 & 10 & 0.99
   & 0.0\footnote{Since time is reported in hours, we report 0.0 when consumed time is less than a minute} & 0     & 0.1 & 0    
   & 3.3 & 0     & 4.6 & 0    \\  
10 &  5 & 10 & 0.95
   & 3.8 & 0     & 4.7 & 0    
   & 3.8 & 0     & 4.8 & 0    \\
10 &  5 & 15 & 0.99
   & 4.7 & 0     & 6.1 & 0    
   & 5.0 & 0     & 6.4 & 0    \\
10 & 10 & 15 & 0.99
   & 5.4 & 0     & 6.2 & 0    
   & 5.3 & 0     & 6.2 & 0    \\
10 & 10 & 15 & 0.95
   & 2.1 & 0     & 6.5 & 0    
   & 5.0 & 0     & 6.8 & 0    \\
10 & 15 & 10 & 0.95
   & 3.1 & 0     & 3.8 & 0    
   & 1.1 & 0     & 1.4 & 0    \\
10 & 20 & 15 & 0.99
   & 5.1 & 0     & 7.0 & 0    
   & 3.1 & 0     & 3.4 & 0    \\
10 & 20 & 15 & 0.95
   & 2.1 & 0     & 3.2 & 0    
   & 3.5 & 0     & 5.1 & 0    \\
15 &  5 & 10 & 0.95
   & 2.5 & 0     & 3.0 & 0    
   & 1.8 & 0     & 2.2 & 0    \\
15 &  5 & 15 & 0.999
   & 0.1 & 0     & 0.0 & 0    
   &10.7 & 0     &13.5 & 0    \\
15 & 10 & 10 & 0.99
   & 1.1 & 0     & 2.3 & 0    
   & 0.8 & 0     & 1.1 & 0    \\
15 & 15 & 10 & 0.99
   & 0.0 & 0     & 3.0 & 0    
   & 3.4 & 0     & 5.0 & 0    \\
15 & 20 & 10 & 0.95
   & 0.0 & 0     & 0.0 & 0    
   & 0.1 & 0     & 0.1 & 0    \\
20 &  2 & 15 & 0.95
   & 2.1 & 0     & 3.5 & 0    
   & 3.2 & 0     & 3.5 & 0    \\
20 &  5 & 15 & 0.999
   & 2.1 & 0     &18.0 & 0    
   & 5.3 & 0     &14.2 & 0    \\
20 & 10 & 10 & 0.99
   & 2.1 & 0     & 4.0 & 0    
   & 5.6 & 0     & 6.0 & 0    \\
20 & 10 & 15 & 0.999
   & 8.8 & 0     &11.4 & 0    
   & 5.2 & 0     &11.4 & 0    \\
20 & 15 & 10 & 0.99
   & 2.2 & 0     & 3.1 & 0    
   & 8.3 & 0     &10.3 & 0    \\
20 & 15 & 15 & 0.95
   & 3.1 & 0     &17.8 & 0    
   & 5.0 & 0     &12.7 & 0    \\
20 & 20 & 10 & 0.99
   & 0.0 & 0     & 0.0 & 0    
   & 0.0 & 0     & 0.0 & 0    \\
20 & 20 & 15 & 0.95
   & 5.8 & 0     &17.9 & 0    
   & 3.7 & 0     &17.0 & 0    \\
25 &  2 & 15 & 0.95
   & 3.7 & 0     & 4.0 & 0    
   & 4.5 & 0     & 5.9 & 0    \\
25 & 10 & 15 & 0.999
   & 5.8 & 0     & 6.8 & 0    
   & 5.8 & 0     & 6.8 & 0    \\
25 & 15 & 10 & 0.99
   & 1.0 & 0     & 1.5 & 0    
   & 9.6 & 0     &10.8 & 0    \\

\bottomrule
\end{longtable}

\begin{longtable}{r r r r|r c|r c|r c|r c}
\caption{\centering Effect of number of breakpoints on solution time and optimality gap for PWL-O $(n=500)$. Solution time is reported in \emph{hours} (hr). Gap \% represents the optimality gap (in percentage) obtained at termination from Gurobi.}
\label{tab: breakpoint-effect-PWL-O}\\
\hline
\multirow{3}{*}{$\varrho$} & 
\multirow{3}{*}{$\varsigma$} & 
\multirow{3}{*}{$K$} & 
\multirow{3}{*}{$\theta$} &
  \multicolumn{4}{c|}{\textit{Equal weight}} &
  \multicolumn{4}{c}{\textit{Different weight}} \\
\cline{5-12}
& & & 
  & \multicolumn{2}{c|}{$\tau = (1-\theta)/2$} 
  & \multicolumn{2}{c|}{$\tau = (1-\theta)/50$}
  & \multicolumn{2}{c|}{$\tau = (1-\theta)/2$}
  & \multicolumn{2}{c}{$\tau = (1-\theta)/50$} \\
\cline{5-12}
& & & 
  & Time (hr) & Gap \% 
  & Time (hr) & Gap \% 
  & Time (hr) & Gap \% 
  & Time (hr) & Gap \% \\ 
\hline
\endfirsthead

\multicolumn{12}{c}{{\tablename\ \thetable\ -- \textit{continued from previous page}}} \\
\hline
\multirow{3}{*}{$\varrho$} & 
\multirow{3}{*}{$\varsigma$} & 
\multirow{3}{*}{$K$} & 
\multirow{3}{*}{$\theta$} &
  \multicolumn{4}{c|}{\textit{Equal weight}} &
  \multicolumn{4}{c}{\textit{Different weight}} \\
\cline{5-12}
& & & 
  & \multicolumn{2}{c|}{$\tau = (1-\theta)/2$} 
  & \multicolumn{2}{c|}{$\tau = (1-\theta)/50$}
  & \multicolumn{2}{c|}{$\tau = (1-\theta)/2$}
  & \multicolumn{2}{c}{$\tau = (1-\theta)/50$} \\
\cline{5-12}
& & & 
  & Time (hr) & Gap \% 
  & Time (hr) & Gap \% 
  & Time (hr) & Gap \% 
  & Time (hr) & Gap \% \\ 
\hline
\endhead
\hline \multicolumn{12}{r}{\textit{(continued on next page)}}\\
\endfoot
\endlastfoot

 5 &  5 & 10 & 0.95  
   & 0.1 & 7e-4 & 0.1 & 7e-4 & 0.1 & 4e-4 & 0.1 & 4e-4 \\

 5 &  5 & 15 & 0.999 
   & 0.1 & 2e-4 & 0.1 & 2e-4 & 0.2 & 2e-4 & 0.4 & 2e-5 \\

10 &  5 & 10 & 0.95  
   & 0.1 & 8e-5 & 0.1 & 8e-5 & 0.1 & 1e-4 & 0.1 & 1e-4 \\

10 &  5 & 10 & 0.99  
   & 0.1 & 3e-4 & 0.1 & 3e-4 & 0.1 & 3e-4 & 0.2 & 5e-5 \\

10 &  5 & 15 & 0.99  
   & 0.1 & 2e-4 & 0.2 & 2e-4 & 0.1 & 2e-4 & 0.1 & 2e-4 \\

10 & 10 & 15 & 0.95  
   & 0.1 & 1e-2 & 0.1 & 1e-2 & 0.1 & 2e-3 & 0.2 & 2e-3 \\

10 & 10 & 15 & 0.99  
   & 0.2 & 8e-4 & 1.8 & 2e-4 & 3.6 & 4e-4 & 4.5 & 2e-4 \\

10 & 15 & 10 & 0.95  
   & 1.6 & 2e-4 & 1.9 & 2e-4 & 0.1 & 1e-4 & 0.1 & 1e-4 \\

10 & 20 & 15 & 0.95  
   & 2.2 & 2e-4 & 2.7 & 2e-4 & 2.5 & 1e-4 & 3.2 & 1e-4 \\

10 & 20 & 15 & 0.99  
   & 2.0 & 5e-4 & 2.0 & 2e-4 & 5.7 & 4e-4 & 5.8 & 2e-4 \\

15 &  5 & 10 & 0.95  
   & 0.1 & 1e-4 & 0.1 & 1e-4 & 0.1 & 3e-5 & 0.1 & 3e-5 \\

15 &  5 & 15 & 0.999 
   & 2.6 & 1e-4 &18.0 & 2e-5 & 3.3 & 1e-4 & 3.4 & 1e-4 \\

15 & 10 & 10 & 0.99  
   & 1.0 & 3e-4 & 1.7 & 3e-4 & 0.1 & 2e-4 & 0.1 & 2e-4 \\

15 & 15 & 10 & 0.99  
   & 0.1 & 2e-4 & 0.1 & 2e-4 & 0.1 & 1e-3 & 0.1 & 1e-3 \\

15 & 20 & 10 & 0.95  
   & 0.1 & 1e-4 & 0.1 & 1e-4 & 0.1 & 6e-5 & 0.1 & 6e-5 \\

20 &  2 & 15 & 0.95  
   & 0.1 & 2e-4 & 0.1 & 2e-4 & 0.1 & 1e-4 & 0.1 & 1e-4 \\

20 &  5 & 15 & 0.999 
   & 5.1 & 1e-4 &18.0 & 6e-5 & 5.3 & 3e-5 & 6.3 & 1e-5 \\

20 & 10 & 10 & 0.99  
   & 0.1 & 1e-4 & 0.1 & 1e-4 & 0.1 & 1e-4 & 0.1 & 1e-4 \\

20 & 10 & 15 & 0.999 
   & 2.5 & 6e-5 &12.3 & 2e-5 & 8.3 & 1e-4 &11.4 & 2e-5 \\

20 & 15 & 10 & 0.99  
   & 0.1 & 2e-4 & 0.1 & 2e-4 & 0.1 & 2e-4 & 0.1 & 2e-4 \\

20 & 20 & 10 & 0.99  
   & 0.1 & 2e-4 & 0.1 & 2e-4 & 0.1 & 2e-4 & 0.1 & 2e-4 \\

20 & 20 & 15 & 0.95  
   & 4.1 & 5e-5 &18.0 & 1e-3 & 3.1 & 3e-5 & 3.4 & 3e-5 \\

25 &  2 & 15 & 0.95  
   & 0.1 & 4e-5 & 0.1 & 4e-5 & 0.1 & 1e-5 & 0.1 & 1e-5 \\

25 & 10 & 15 & 0.999 
   &18.0 & 3e-5 &18.0 & 3e-5 &18.0 & 3e-5 &18.0 & 3e-5 \\

25 & 15 & 10 & 0.99  
   & 1.8 & 4e-4 & 1.9 & 4e-4 & 0.1 & 3e-3 & 0.1 & 3e-3 \\
\hline
\end{longtable}

\begin{longtable}{r r r r|r c|r c|r c|r c}
\caption{\centering Effect of the number of breakpoints on solution time and optimality gap for PWL-I $(n=1000)$. Solution time is reported in \emph{hours} (hr). Gap \% represents the optimality gap (in percentage) obtained at termination from Gurobi.} 
\label{tab: breakpoint-effect-PWL-I-1000}\\
\hline
\multirow{3}{*}{$\varrho$} & 
\multirow{3}{*}{$\varsigma$} & 
\multirow{3}{*}{$K$} & 
\multirow{3}{*}{$\theta$} &
  \multicolumn{4}{c|}{\textit{Equal weight}} &
  \multicolumn{4}{c}{\textit{Different weight}} \\
\cline{5-12}
& & & 
  & \multicolumn{2}{c|}{$\tau = (1-\theta)/2$} 
  & \multicolumn{2}{c|}{$\tau = (1-\theta)/50$}
  & \multicolumn{2}{c|}{$\tau = (1-\theta)/2$}
  & \multicolumn{2}{c}{$\tau = (1-\theta)/50$} \\
\cline{5-12}
& & & 
  & Time (hr) & Gap \% 
  & Time (hr) & Gap \% 
  & Time (hr) & Gap \% 
  & Time (hr) & Gap \% \\ 
\hline
\endfirsthead

\multicolumn{12}{c}{{\tablename\ \thetable\ -- \textit{continued from previous page}}} \\
\hline
\multirow{3}{*}{$\varrho$} & 
\multirow{3}{*}{$\varsigma$} & 
\multirow{3}{*}{$K$} & 
\multirow{3}{*}{$\theta$} &
  \multicolumn{4}{c|}{\textit{Equal weight}} &
  \multicolumn{4}{c}{\textit{Different weight}} \\
\cline{5-12}
& & & 
  & \multicolumn{2}{c|}{$\tau = (1-\theta)/2$} 
  & \multicolumn{2}{c|}{$\tau = (1-\theta)/50$}
  & \multicolumn{2}{c|}{$\tau = (1-\theta)/2$}
  & \multicolumn{2}{c}{$\tau = (1-\theta)/50$} \\
\cline{5-12}
& & & 
  & Time (hr) & Gap \% 
  & Time (hr) & Gap \% 
  & Time (hr) & Gap \% 
  & Time (hr) & Gap \% \\ 
\hline
\endhead
\hline \multicolumn{12}{r}{\textit{(continued on next page)}}\\
\endfoot
\endlastfoot


  2 & 10 & 15 & 0.999 
    & 0.3 & 0     & 0.3 & 0    
    & 0.3 & 0     & 0.3 & 0    \\

  2 & 15 & 15 & 0.999 
    & 0.2 & 0     & 0.2 & 0    
    & 0.3 & 0     & 0.5 & 0    \\

  2 & 20 & 15 & 0.999 
    & 0.2 & 0     & 0.3 & 0    
    & 1.0 & 0     & 1.2 & 0    \\

 10 &  2 & 15 & 0.999 
    & 0.3 & 0     & 0.3 & 0    
    &  -- & --    &  -- & --   \\

 10 & 10 & 15 & 0.95  
    & 0.2 & 0     & 0.2 & 0    
    & 0.3 & 0     & 0.3 & 1.1e-5 \\

 10 & 20 & 15 & 0.999 
    & 0.4 & 0     & 0.7 & 0    
    &  -- & --    &  -- & --   \\

 15 &  5 & 10 & 0.999 
    & 0.3 & 0     & 0.3 & 0    
    &  -- & --    &  -- & --   \\

 20 &  5 & 10 & 0.999 
    & 0.1 & 0     & 0.2 & 0    
    &  -- & --    &  -- & --   \\

 25 &  5 & 15 & 0.95  
    &18.0 & 0.19     &18.0 & 0.19    
    &  -- & --    &  -- & --   \\

 25 & 20 & 15 & 0.999 
    &  -- & --    &  -- & --   
    &18.0 & 0.008     &18 & 0.008  \\
\hline
\end{longtable}

\begin{longtable}{r r r r|r c|r c|r c|r c}
\caption{\centering Effect of the number of breakpoints on solution time and optimality gap for PWL-O $(n=1000)$. Solution time is reported in \emph{hours} (hr). Gap \% represents the optimality gap (in percentage) obtained at termination from Gurobi.} 
\label{tab: breakpoint-effect-PWL-O-1000}\\
\hline
\multirow{3}{*}{$\varrho$} & 
\multirow{3}{*}{$\varsigma$} & 
\multirow{3}{*}{$K$} & 
\multirow{3}{*}{$\theta$} &
  \multicolumn{4}{c|}{\textit{Equal weight}} &
  \multicolumn{4}{c}{\textit{Different weight}} \\
\cline{5-12}
& & & 
  & \multicolumn{2}{c|}{$\tau = (1-\theta)/2$} 
  & \multicolumn{2}{c|}{$\tau = (1-\theta)/50$}
  & \multicolumn{2}{c|}{$\tau = (1-\theta)/2$}
  & \multicolumn{2}{c}{$\tau = (1-\theta)/50$} \\
\cline{5-12}
& & & 
  & Time (hr) & Gap \% 
  & Time (hr) & Gap \% 
  & Time (hr) & Gap \% 
  & Time (hr) & Gap \% \\ 
\hline
\endfirsthead

\multicolumn{12}{c}{{\tablename\ \thetable\ -- \textit{continued from previous page}}} \\
\hline
\multirow{3}{*}{$\varrho$} & 
\multirow{3}{*}{$\varsigma$} & 
\multirow{3}{*}{$K$} & 
\multirow{3}{*}{$\theta$} &
  \multicolumn{4}{c|}{\textit{Equal weight}} &
  \multicolumn{4}{c}{\textit{Different weight}} \\
\cline{5-12}
& & & 
  & \multicolumn{2}{c|}{$\tau = (1-\theta)/2$} 
  & \multicolumn{2}{c|}{$\tau = (1-\theta)/50$}
  & \multicolumn{2}{c|}{$\tau = (1-\theta)/2$}
  & \multicolumn{2}{c}{$\tau = (1-\theta)/50$} \\
\cline{5-12}
& & & 
  & Time (hr) & Gap \% 
  & Time (hr) & Gap \% 
  & Time (hr) & Gap \% 
  & Time (hr) & Gap \% \\ 
\hline
\endhead

\hline \multicolumn{12}{r}{\textit{(continued on next page)}}\\
\endfoot
\endlastfoot


 2 & 10 & 15 & 0.999 
   & 1.1   & 2.6e-5 & 18.0  & 4.1e-4
   & 0.8   & 0.0009  & 18.0  & 2.6e-4 \\

 2 & 15 & 15 & 0.999 
   & 0.1   & 0.0011  & 18.0  & 0.0011
   & 18.0  & 0.0075  & 18.0  & 0.0075 \\

 2 & 20 & 15 & 0.999 
   & 18.0  & 0.0009  & 18.0  & 0.0009
   & 1.0   & 0.0009  & 18.0  & 0.0009 \\

 5 & 20 & 15 & 0.99  
   & 18.0  & 0.0083  & 18.0  & 0.0087
   &   –   & –       &   –   & –      \\

10 &  2 & 15 & 0.999 
   & 18.0  & 0.3056  & 18.0  & 0.3056
   &   –   & –       &   –   & –      \\

10 &  5 & 15 & 0.999 
   & 1.1   & 5.0e-5  & 1.2   & 1.7e-5
   & 10.9  & 0.0002  & 18.0  & 0.0002 \\

10 & 10 & 15 & 0.95  
   & 1.1   & 3.9e-5  & 8.1   & 0.0048
   & 18.0  & 2.0e-5  & 0.8   & –      \\

10 & 20 & 15 & 0.999 
   & 0.8   & 0.0002  & 3.7   & 1.9e-5
   & 18.0  & 0.0059  & 18.0  & 0.0059 \\

15 &  5 & 10 & 0.999 
   & 13.3  & 1.1e-5  & 18.0  & 2.6e-5
   &   –   & –        &   –   & –      \\

15 & 15 & 15 & 0.999 
   & 1.0   & 0.0002  & 18.0  & 9.7e-5
   & 18.0  & 0.0404  & 18.0  & 0.0404  \\

20 &  5 & 10 & 0.999 
   & 0.5   & 2.4e-5  & 0.8   & 2.0e-5
   & 12.9  & 2.1e-6  & 13.0  & 2.4e-5 \\

20 & 10 & 15 & 0.999 
   & 0.8   & 5.3e-5  & 18.0  & 0.0008
   &   –   & –        &   –   & –      \\

20 & 20 & 15 & 0.95  
   & 0.9   & 0.0056  & 11.0  & 0.0056
   &   –   & –        &   –   & –      \\

25 &  5 & 10 & 0.999 
   & 0.5   & 3.0e-1  & 0.7   & 1.3e-5
   &   –   & –        &   –   & –      \\

25 &  5 & 15 & 0.95  
   & 0.2   & 0.0043  & 0.3   & 0.0043
   & 18.0  & 7.4e-6  & 18.0  & 0.0043 \\

25 & 15 & 15 & 0.95  
   & 0.6   & 7.4e-6  & 18.0  & 0.0047
   & 15.9  & 1.7e-6  & 18.0  & 1.7e-6 \\

25 & 20 & 15 & 0.999 
   & 18.0  & 0.0009  & 18.0  & 0.0009
   & 18.0  & 0.0368  & 18.0  & 0.0368  \\
\hline
\end{longtable}

\section{Regression Tables}\label{appndx:regression-table}

\begin{table}[h]
\small
    \caption{Regression Results for PWL-I Approximations}
    \label{tab:regression_results1}
    \centering
    \begin{tabular}{l|cc|cc|cc}
        \hline
        & \multicolumn{2}{c}{$\theta = 0.95$} & \multicolumn{2}{c}{$\theta = 0.99$} & \multicolumn{2}{c}{$\theta = 0.999$} \\
        \cline{2-7}
        & Estimate & Pr($>|t|$) & Estimate & Pr($>|t|$) & Estimate & Pr($>|t|$) \\
        \hline
        Intercept        & -8.12e+3 & 2.37e-05 ***     & -6.60e+3  & 0.00135 **      & -1.06e+4  & 1.23e-06 *** \\
        
        $n$              & 4.29e+1  & 0.0154 *         & 6.09e+1   & 0.00150 **      & 1.23e+2  & 2.00e-09 ***\\
        
        $K$              & 6.27e+2  & 1.99e-08 ***     & 5.79e+2   & 1.20e-06 ***    & 8.46e+2  & 2.86e-11 ***\\
        
        $\varrho$        & -4.44e+1 & 0.5046           & -5.70e+1  & 0.42722         & -8.40e+1 & 0.26805 \\
        
        $\varsigma$      & 1.26e+2  & 0.0459 *         & 6.83e+1  & 0.31215         & 2.32e+2  & 0.00124 **\\
        
        MIP-gap          & 2.63e+2  & $<$2e-16 ***     & 2.50e+2  & $<$2e-16 ***    & 3.61e+2  & 0.00882 ** \\
        
        avg$\frac{\nu^{\max}}{\nu^{\min}}$  & -1.09e+2 & 0.7093 & 5.12e+1 & 0.87082 & -1.48e+2 & 0.65528 \\
        
        $\bar{d}_{\mu}$  & 2.06e+1  & 0.7439          & -1.19e+2 & 0.08258 .      & -2.81e+2 & 9.85e-05 ***\\
        
        max$\frac{\nu^{\max}}{\nu^{\min}}$   & -2.89e+0 & 0.9063         & -1.49e+1 & 0.57307         & 1.51e+1  & 0.58955 \\
        
        $d^{\max}_{\mu}$ & 9.98e+0  & 0.7206          & 6.76e+1  & 0.02670 *       & 1.28e+2  & 7.25e-05 ***\\
        
        weight           & 1.43e+3  & 0.0782 .        & -2.11e+2 & 0.80926         & -1.20e+3 & 0.19495\\
        \hline
        \multicolumn{7}{c}{Signif. codes:   $<$0.001 ‘***’ 0.001 ‘**’ 0.01 ‘*’ 0.05 ‘.’} \\
        \hline
    \end{tabular}
\end{table}

\begin{table}[h]
\vspace{-10pt}
\small
    \caption{Regression Results for PWL-O Approximations}
    \label{tab:regression_results2}
    \centering
    \begin{tabular}{l|cc|cc|cc}
        \hline
        & \multicolumn{2}{c}{$\theta = 0.95$} & \multicolumn{2}{c}{$\theta = 0.99$} & \multicolumn{2}{c}{$\theta = 0.999$} \\
        \cline{2-7}
        & Estimate & Pr($>|t|$) & Estimate & Pr($>|t|$) & Estimate & Pr($>|t|$) \\
        \hline
        Intercept        & -6.76e+03 & 3.12e-05 ***    & -1.00e+04 & 1.26e-06 ***    & -1.54e+04 & 4.03e-08 *** \\
        $n$              & 5.32e+01  & 3.97e-04 ***    & 9.19e+01  & 1.86e-06 ***    & 1.33e+02  & 3.51e-07 *** \\
        $K$              & 4.72e+02  & 5.49e-07 ***    & 7.89e+02  & 4.75e-11 ***    & 1.41e+03  & $<$2e-16 *** \\
        $\varrho$        & -1.57e+01 & 7.80e-01       & -4.65e+01 & 5.17e-01       & -1.65e+02 & 9.02e-02 .  \\
        $\varsigma$      & 1.39e+02  & 9.02e-03 **    & 1.75e+02  & 9.44e-03 **    & 3.20e+02  & 5.07e-04 *** \\
        MIP-gap          & 2.77e+02  & $<$2e-16 ***   & 2.40e+02  & $<$2e-16 ***   & 2.13e+02  & 1.34e-10 *** \\
        avg$\frac{\nu^{\max}}{\nu^{\min}}$  & 1.05e+00  & 9.97e-01       & -9.94e+01 & 7.53e-01       & -1.58e+02 & 7.13e-01 \\
        $\Bar{d}_{\mu}$  & -3.72e+01 & 4.83e-01       & 2.74e+00  & 9.68e-01       & -2.10e+02 & 2.39e-02 *  \\
        max$\frac{\nu^{\max}}{\nu^{\min}}$   & -5.46e+00  & 7.93e-01       & -1.86e+00  & 9.44e-01       & 1.26e+01  & 7.27e-01 \\
        $d^{\max}_{\mu}$ & 1.99e+01  & 3.98e-01       & -2.95e+00 & 9.23e-01       & 1.02e+02  & 1.31e-02 *  \\
        weight           & 9.32e+02  & 1.74e-01       & -7.68e+02 & 3.80e-01       & -5.68e+02 & 6.32e-01 \\
        \hline
        \multicolumn{7}{c}{Signif. codes:  $<$0.001 ‘***’ 0.001 ‘**’ 0.01 ‘*’ 0.05 ‘.’ } \\
        \hline
    \end{tabular}
\end{table}

\end{APPENDICES}
\end{document}